\documentclass[a4paper, 11pt]{amsart}
\usepackage{longtable, stackrel, multicol, fancybox, tcolorbox}
\usepackage{bm,amsfonts,amsmath,amssymb,mathrsfs,bbm,amsthm} 
\usepackage{graphicx, color,hyperref,eurosym, eucal,palatino}
\usepackage[normalem]{ulem}

\definecolor{darkblue}{rgb}{0.2,0.2,0.6}
\definecolor{darkblue2}{rgb}{0.2,0.2,0.9}
\definecolor{superdarkblue}{rgb}{0.2,0.2,0.3}
\hypersetup{colorlinks=true, linkcolor=darkblue,citecolor=darkblue, 
	urlcolor = superdarkblue}
\definecolor{citegreen}{rgb}{0.2,0.2,0.6}
\usepackage{marginnote}
\usepackage{geometry}
\hypersetup{colorlinks=true, linkcolor=blue,citecolor=citegreen}
\definecolor{citegreen}{rgb}{0.2,0.2,0.6}
\definecolor{darkblue}{rgb}{0.1,0.1,0.6}
\definecolor{md}{rgb}{0.7,0.3,0.}

\usepackage{xcolor}

\newcommand\Omg{\Omega}

\newcommand\Sg{\Sigma}
\newcommand\rank{{\rm rank}\,}
\renewcommand\gg{\gamma}

\newcommand\Op{\sfA_{\aa}}
\newcommand\Opf{\sfA_0}
\newcommand\OpD{\sfA_{\rm D}^\Omg}
\newcommand\bA{\mathbf{A}}

\newcommand{\map}{\mathsf{j}}

\newcommand\nb{\nabla}
\newcommand\Dl{\Delta}
\newcommand\nbA{\nb_\bA}

\renewcommand\and{\qquad\text{and}\qquad}

\newcommand\sm{\setminus}

\newcommand\fra{\mathfrak{a}}
\newcommand\frh{\mathfrak{h}}
\newcommand\frm{\fra_\aa}

\newcommand\dl{\delta}

\newcommand{\comm}[1]{}

\newcommand\G{\Gamma}

\renewcommand\aa{\alpha}
\newcommand\lm{\lambda}
\newcommand\Lm{\Lambda}
\newcommand\s{\sigma}
\newcommand\p{\partial}

\newcommand\omg{\omega}

\newcommand\kp{\kappa}

\newcommand\ii{{\mathsf{i}}}

\newcommand\arr{\rightarrow}

\renewcommand\sp{\sigma_{\rm p}}
\newcommand\sd{\sigma_{\rm disc}}
\newcommand\sess{\sigma_{\rm ess}}

\newcommand\dd{{\mathsf{d}}}
\newcommand\uhr{\upharpoonright}

\newcounter{counter_a}
\newenvironment{myenum}{\begin{list}{{\rm(\roman{counter_a})}}%
{\usecounter{counter_a}
\setlength{\itemsep}{1.ex}\setlength{\topsep}{0.8ex}
\setlength{\leftmargin}{5ex}\setlength{\labelwidth}{5ex}}}{\end{list}}

\usepackage[latin1]{inputenc}
\usepackage[T1]{fontenc}

\numberwithin{figure}{section}
\numberwithin{equation}{section}
\theoremstyle{plain}
\newtheorem*{thm*}{Theorem}
\newtheorem{thm}{Theorem}[section]
\newtheorem{hyp}[thm]{Hypothesis}
\newtheorem{lem}[thm]{Lemma}
\newtheorem{prop}[thm]{Proposition}

\newtheorem{cor}[thm]{Corollary}

\newtheorem{dfn}[thm]{Definition}
\theoremstyle{remark}
\newtheorem{remark}[thm]{Remark}
\theoremstyle{plain}

\newcommand\ov{\overline}

\newcommand\wh{\widehat}

\def\ov{\overline}
\newcommand\ran{{\rm ran\,}}

   \def\sB{{\mathfrak B}}

\def\sS{{\mathfrak S}}

\def\frb{{\mathfrak b}}

\def\frt{{\mathfrak t}}

      \def\dC{{\mathbb C}}

   \def\dN{{\mathbb N}}   
      \def\dR{{\mathbb R}}

\def\sfA{{\mathsf A}}      
\def\sfD{{\mathsf D}}      
   \def\sfH{{\mathsf H}}

      \def\sfR{{\mathsf R}}

   \def\cB{{\mathcal B}}   
\def\cD{{\mathcal D}}      \def\cF{{\mathcal F}}
\def\cG{{\mathcal G}}   \def\cH{{\mathcal H}}   \def\cI{{\mathcal I}}
\def\cJ{{\mathcal J}}   \def\cK{{\mathcal K}}   \def\cL{{\mathcal L}}
   \def\cN{{\mathcal N}}   \def\cO{{\mathcal O}}
\def\cP{{\mathcal P}}      
\def\cS{{\mathcal S}}

\newcommand{\dom}{\mathrm{dom}\,}

\makeatletter
\def\section{\@startsection{section}{1}\z@{.9\linespacing\@plus\linespacing}%
	{.7\linespacing} {\fontsize{13}{14}\selectfont\bfseries\centering}}
\def\paragraph{\@startsection{paragraph}{4}%
	\z@{0.3em}{-.5em}%
	{$\bullet$ \ \normalfont\itshape}}

\@addtoreset{equation}{section}
\makeatother

\newcommand\supp{{\rm supp}\,}
\newcommand\bb{\beta}
\newcommand\eps{\varepsilon}

\makeatletter
\newcommand\footnoteref[1]{\protected@xdef\@thefnmark{\ref{#1}}\@footnotemark}
\makeatother

\title[Landau Hamiltonian with $\delta$-potentials]{\textsc{The Landau Hamiltonian with \boldmath{$\delta$}-potentials supported on curves}}

\author[J. Behrndt]{Jussi Behrndt}
\address{Institut f\"{u}r Angewandte Mathematik\\
Technische Universit\"{a}t Graz\\
 Steyrergasse 30, 8010 Graz, Austria\\
E-mail: {behrndt@tugraz.at}}

\author[P. Exner]{Pavel Exner}
\address{Doppler Institute for Mathematical Physics and Applied Mathematics\\ 
Czech Technical University in Prague\\ B\v{r}ehov\'{a} 7, 11519 Prague, Czech Republic,
{\rm and}
Department of Theoretical Physics\\
Nuclear Physics Institute, Czech Academy of Sciences, 
25068 \v{R}e\v{z}, Czech Republic\\
E-mail: {exner@ujf.cas.cz}
}

\author[M. Holzmann]{Markus Holzmann}
\address{Institut f\"{u}r Angewandte Mathematik\\
Technische Universit\"{a}t Graz\\
 Steyrergasse 30, 8010 Graz, Austria\\
E-mail: {holzmann@math.tugraz.at}}

\author[V. Lotoreichik]{Vladimir Lotoreichik}
\address{
Department of Theoretical Physics\\
Nuclear Physics Institute, Czech Academy of Sciences, 
25068 \v{R}e\v{z}, Czech Republic\\
E-mail: {lotoreichik@ujf.cas.cz }
}

\begin{document}
\begin{abstract}
The spectral properties of the singularly perturbed self-adjoint Landau Hamiltonian
$\Op =(\ii\nb + \bA)^2 + \aa\dl_\Sg$ in $L^2(\dR^2)$ with a $\dl$-potential supported on a finite $C^{1,1}$-smooth curve $\Sigma$ are studied.
Here $\bA = \frac{1}{2} B (-x_2, x_1)^\top$ is the vector potential, $B>0$ is the strength of the homogeneous magnetic
field, and $\alpha\in L^\infty(\Sigma)$
is a position-dependent real coefficient modeling the strength of the singular interaction on the curve $\Sigma$.
After a general discussion of the qualitative spectral properties of $\Op$ and its resolvent, one of the main objectives in the present paper
is a local spectral analysis of $\Op$
near the Landau levels $B(2q+1)$, $q\in\dN_0$.
Under various conditions on $\aa$ it is shown that the perturbation smears the Landau levels
into eigenvalue clusters, and the accumulation rate of the eigenvalues within these clusters is determined in terms of the capacity of the support of $\aa$.
Furthermore, the use of Landau Hamiltonians with $\delta$-perturbations as model operators for more realistic quantum systems is justified
by showing that $\Op$ can be approximated in the norm resolvent sense by a family of
Landau Hamiltonians with suitably scaled regular potentials.
\end{abstract}

\maketitle

\section{Introduction} \label{section_introduction}

Quantum motion in a geometrically complicated background is often modeled by networks of
\emph{leaky quantum wires}, which are mathematically
described by Schr\"odin\-ger operators
with singular potentials supported on families of curves, see, e.g., the monograph \cite[Chapter~10]{EK15}, the papers
\cite{BLL13, BEKS94, E08, MPS16,stollmann_voigt1996}, and the references therein. 
Such models based on PDEs are mathematically more involved than the alternative concept of {\it quantum graphs} \cite{BK13} based on ODEs,
but have serious advantages from
the physical point of view since they do not neglect quantum tunnelling between parts of the network. Although there is nowadays a comprehensive
literature on spectral and scattering properties of Schr\"odinger operators
with singular potentials, only few mathematical contributions are concerned with the influence of
magnetic fields (see \cite{EK18,ELP17,ET04,EY02,HH04,O06}), despite the fact that applications of such fields, local or global, are an important area in modern physics.
Magnetic Schr\"odinger operators with surface interactions appear, e.g., in the analysis of the non-linear Ginzburg-Landau equation, cf. \cite{FH,Ra}.

The present paper can be regarded as a first step towards a general treatment of Landau Hamiltonians with singular potentials supported on curves.
Throughout this paper let
the strength $B>0$ of the homogeneous magnetic field be fixed, let the corresponding vector potential in the symmetric gauge
be $\bA := \frac{1}{2} B (-x_2, x_1)^\top$, and define the magnetic gradient by
\begin{equation}\label{eq:mag_grad}
    \nbA := \ii \nb + \bA.
\end{equation}
Our main goal is to construct a class of singular perturbations of the Landau Hamiltonian $\Opf  = \nbA^2$
by $\delta$-potentials
supported on finite curves. We study the spectral properties of these singularly perturbed Landau Hamiltonians in detail and we justify their use
as model operators for more realistic quantum systems by showing that they can be approximated in the norm resolvent sense by a family of
Landau Hamiltonians with suitably scaled regular potentials. In order to explain our strategy and results more precisely,
assume that $\Sg$ is the boundary of a compact $C^{1,1}$-domain, let $\aa\in L^\infty(\Sg)$ be a real function, consider the sesquilinear form
\begin{equation}\label{eq:form}
  \frm[f,g]
   =
  \big(\nbA f, \nbA g)_{L^2(\dR^2, \dC^2)}
  +
  \big(\aa f|_\Sg, g|_\Sg\big)_{L^2(\Sg)},\quad
  \dom\frm  = \cH_\bA^1(\dR^2),
\end{equation}
where $\cH_\bA^1(\dR^2)=\{f\in L^2(\dR^2): |\nbA f|\in L^2(\dR^2)\}$
is the magnetic Sobolev space,
and denote the corresponding self-adjoint operator in $L^2(\dR^2)$ by $\Op$. If $\dl_\Sg$ denotes the $\dl$-distribution supported on the
curve $\Sg$ then on a formal level
\begin{equation}\label{ai}
    \Op = \nbA^2 +\aa\dl_\Sg = \sfA_0 + \aa\dl_\Sg.
\end{equation}

Our approach to the spectral analysis of the Landau Hamiltonians with singular potentials is via abstract techniques from extension theory of symmetric operators.
Here we shall use the notion of quasi boundary triples and their Weyl functions from \cite{BL07,BL12} to first determine the operator $\Op$ associated to
$\frm$ and its domain via explicit interface conditions at $\Sg$. As a byproduct we obtain a Birman-Schwinger principle and the useful resolvent formula
\begin{equation}\label{ki}
    (\Op -\lm)^{-1}
    =
    (\sfA_0-\lm)^{-1} -
    \gg(\lm)\big(1+\aa M(\lm)\big)^{-1}\aa\gg(\ov\lm)^*,
\end{equation}
where $\gamma$ and $M$ are the $\gamma$-field and Weyl function, respectively, 
corresponding to a suitable quasi boundary triple $\{L^2(\Sigma),\Gamma_0,\Gamma_1\}$.
We refer the reader to Appendix~\ref{app:A} for a brief introduction to quasi boundary triples and Weyl functions, and here we only mention that 
$\gamma(\lambda)\colon L^2(\Sigma)\arr L^2(\dR^2)$
and $M(\lambda)\colon L^2(\Sigma)\rightarrow L^2(\Sigma)$ in \eqref{ki} can also be viewed as (boundary) integral operators with the Green function of $\Opf$ as 
integral kernel. 
The formula~\eqref{ki} can be seen as an interpretation of the formal equality~\eqref{ai}: the resolvent difference is essentially
reduced to the term $(1+\aa M(\lm))^{-1}\aa$, which is localized on the curve $\Sg$ and contains the main information on the spectrum of
$\Op$. In fact, our further investigations are based on a detailed analysis of the perturbation term
\begin{equation}\label{eq:W}
    W_\lm =
    -
    \gg(\lm)\big(1+\aa M(\lm)\big)^{-1}\aa\gg(\ov\lm)^*
\end{equation}
in the resolvent formula~\eqref{ki}. Since $\Sigma$ is a compact curve,
the Rellich-Kondrachov embedding theorem implies that $W_\lm$ is compact in $L^2(\dR^2)$ and as an immediate consequence we conclude
\[
    \sess(\Op) =\sess(\Opf)= \s(\Opf) =
    \{B(2q+1) \colon q\in\dN_0\},
\]
where $\Lambda_q=B(2q+1)$, $q=0,1,2,\dots$ are infinite dimensional eigenvalues of $\Opf$, usually called Landau levels.
It is well known that perturbations of the Landau Hamiltonian $\Opf$
can generate accumulation of discrete eigenvalues to the Landau levels $\Lambda_q$.
For additive perturbations of $\Opf$ by an electric potential this was shown by G. Raikov in~\cite{R90},
see also~\cite{FP06, KR09,MR03,PRV13,RW02,RS08,RT08}. More recently similar results were proved
in~\cite{BMR14,BM18,GKP16,P09,RP07} 
for Landau Hamiltonians on domains with Dirichlet, Neumann, and Robin boundary conditions;
for closely related results in three-dimensional situation we refer to~\cite{BBR14,BS16} and the references therein.

Our first main objective is to observe a similar phenomenon on the accumulation of discrete eigenvalues of $\Op$ to the Landau levels $\Lambda_q$,
and to prove singular value estimates and regularized summability properties of the discrete eigenvalues.
For this reason we are particularly interested in the compression
$P_q W_\lm P_q$ of the perturbation term onto the eigenspace $\ker(\Opf-\Lm_q)$ of the unperturbed Landau Hamiltonian.
The operators $P_qW_\lm P_q$ are the analogues of the Toeplitz operators
appearing in this connection in~\cite{FP06, RP07, PR11, RT08}, and we note in this context that some of our observations rely
on deep results in the theory of Toeplitz operators, and conversely that our approach and some of our considerations lead to new results for Toeplitz operators.

If the strength $\aa$ in \eqref{eq:form}--\eqref{ai} is positive (negative) on $\Sg$  we show in Theorem~\ref{thm:acc} that
the discrete spectrum of $\Op$ accumulates to each Landau level $\Lm_q$ from above (below, respectively). Combining our technique
with the constructions in~\cite{FP06, RP07, RS10},
we obtain in Theorem~\ref{thm:acc2} the same result for the lowest Landau level $\Lm_0=B$ under the weaker assumption that $\alpha\not\equiv 0$ is nonnegative (nonpositive),
and in Proposition~\ref{proposition_lower_bound} for the higher Landau levels assuming that $\supp\aa$ contains a 
$C^\infty$-smooth arc on which $\aa$ is positive (negative, respectively).
Relying on the analysis of $P_qW_\lm P_q$, we also estimate the rate of the eigenvalue accumulation in Theorem~\ref{thm:C2est}.
Although the upper bounds on the accumulation rate
of the discrete eigenvalues hold also for sign-changing $\aa$ it is a challenging open problem
to show that the eigenvalue accumulation is indeed present in this situation. 
Furthermore, making use of the technique from~\cite{FP06, RP07} we prove in Theorem~\ref{thm:asymp} spectral asymptotics 
if $\supp\aa$ is a 
$C^\infty$-smooth arc $\Gamma$ and $\alpha$ is uniformly positive (uniformly negative) in the interior of $\Gamma$. 
More precisely, if, e.g., $\alpha>0$ inside the
$C^\infty$-smooth arc $\Gamma=\supp\aa$ then  
the discrete
eigenvalues (counted with multiplicities) of $\Op$
in the interval $(\Lm_q, \Lm_q + B]$, $q\in\dN_0$, form a sequence $\lm_1^+(q) \ge \lm_2^+(q) \ge \dots \ge \Lm_q$
with the asymptotic behaviour
\begin{equation}\label{asyint}
\lim_{k\arr\infty}\big(k!\,(\lm_k^+(q) - \Lm_q)\big)^{1/k} = \frac{B}{2}\big({\rm Cap}\,(\G)\big)^2,
\end{equation}
where ${\rm Cap}\,(\G)$ is the \emph{logarithmic capacity}
of $\G$. We also mention that the eigenvalue asymptotics in \eqref{asyint} comply with \cite[Remark 2 and Theorem 2]{FP06}.

Besides the spectral analysis of the operators $\Op$ in \eqref{ai} our second main objective in this paper is to justify the use of such singular perturbations of the Landau
Hamiltonian for more realistic model operators with regular potentials. The approximation problem of singular potentials by regular ones has been discussed in the absence
of magnetic fields for $\delta$-point interactions in great detail in the monograph \cite{AGHH05}, and for $\delta$-surface interactions in \cite{BEHL17,EI01, EK03} and
\cite{BO07, G15, O06, P95, S92}, see also \cite{AK99, stollmann_voigt1996} for more abstract approaches.
We show in Theorem~\ref{theorem_approximation} and Corollary~\ref{corollary_approximation}
that for real $\alpha\in L^\infty(\Sigma)$ the singular Landau Hamiltonian $\Op$
can be approximated in the norm resolvent sense by a family of regular Landau Hamiltonians with potentials suitably scaled in the direction perpendicular to $\Sigma$.
The choice of the approximating sequence of potentials
is essentially the same as, e.g., in \cite{BEHL17,EI01, EK03}, but the technique of the proof is significantly different and more efficient.

\subsection*{Organization of the paper}
Section~\ref{sec:prelim} contains some preliminary material concerning the
       unperturbed Landau Hamiltonian, properties of Schatten-von
       Neumann ideals and some aspects of perturbation theory. In
       Subsection~\ref{section_Toeplitz} we discuss a class of Toeplitz-like operators
       related to Landau Hamiltonians.
In Section~\ref{sec:qbt} we make use of the abstract concept of quasi boundary triples and their Weyl functions (see Appendix~\ref{app:A} for a brief introduction)
in order to study  Landau Hamiltonians with $\dl$-potentials supported on curves.
Using a suitable quasi boundary triple we show
self-adjointness of $\Op$, provide qualitative spectral properties, and derive the Krein-type
resolvent formula~\eqref{ki}. The approximation of $\Op$
by magnetic Schr\"odinger operators with scaled
regular potentials is also discussed; the proof is technical and therefore outsourced to Appendix~\ref{appendix_approximation}.
Section~\ref{sec:asymptotics} is devoted to the spectral
analysis of the compressed resolvent difference $P_qW_\lm P_q$. Under various assumptions we obtain spectral estimates
and spectral asymptotics for this operator.
Based on these results
we provide our main results on the eigenvalue clusters of $\Op$
at Landau levels in Section~\ref{sec:clusters}.

\subsection*{Acknowledgement}
The authors gratefully acknowledge financial support under the Czech-Austrian grant 7AMBL7ATO22 and CZ 02/2017. 
The research of P.$\,$E. and V.$\,$L.
is
supported by the Czech Science Foundation (GA\v{C}R) under Grant No. 17-01706S. P.$\,$E. also
acknowledges the support by the European Union within the project CZ.02.1.01/0.0/0.0/16 019/0000778. 
The authors also wish to thank V.~Bru\-neau, B.~Helffer, A.~Pushnitski, and G.~Raikov for fruitful discussions and helpful remarks and references.

\section{Preliminaries}\label{sec:prelim}

In this section we provide  useful notions and techniques that are needed in our analysis of magnetic Schr\"odinger operators with singular
interactions.
In Subsection~\ref{ssec:basics} we introduce the Landau Hamiltonian, in Subsection~\ref{section_Schatten-von_Neumann} some important properties of the
Schatten-von Neumann ideals of compact operators are discussed, and
in Subsections \ref{compis} and~\ref{section_Toeplitz} we collect some useful
facts from perturbation theory and
Toeplitz operators that will be needed in the main part of the paper.

\subsection{The Landau Hamiltonian}\label{ssec:basics}
In order to introduce the Landau Hamiltonian, that is, the unperturbed magnetic Schr\"o\-dinger operator with 
homogeneous magnetic field,
recall the definition of the magnetic gradient from~\eqref{eq:mag_grad} and define the first order $L^2$-based magnetic Sobolev space by
\begin{equation}\label{def_magnetic_Sobolev_space}
  \cH_\bA^1(\dR^2)
  := 
  \big\{ f \in L^2(\dR^2) \colon 
  |\nbA f| \in L^2(\dR^2) \big\},
\end{equation}
which becomes a Hilbert space if it is endowed with the inner product 
\begin{equation*}
  (f, g)_{\cH^1_\bA(\dR^2)} 
  := 
  (f, g)_{L^2(\dR^2)} 
  + 
  \big(\nbA f,\nbA g\big)_{L^2(\dR^2;\dC^2)},
  \quad f, g \in \cH^1_\bA(\dR^2).
\end{equation*} 
The space $C_0^\infty(\dR^2)$ of smooth compactly supported functions is dense in $\cH^1_\bA(\dR^2)$, see, e.g., \cite[Theorem~7.22]{LL01}. 
Note that for $B = 0$ the space $\cH_\bA^1(\dR^2)$ coincides with the usual first order Sobolev space $H^1(\dR^2)$; if $B\not= 0$ then still 
$\cH_\bA^1(\dR^2)$
and $H^1(\dR^2)$ coincide locally.
The standard Sobolev spaces of order $s \in \mathbb{R}$ will be denoted in this paper by $H^s(\dR^2)$.

Next consider the symmetric sesquilinear form 
\begin{equation} \label{def_form_free_Op}
  \fra_0[f,g] 
  := 
  \big( \nbA f, \nbA g\big)_{L^2(\dR^2; \dC^2)},
  \qquad \dom \fra_0 = \cH^1_\bA(\dR^2),
\end{equation}
and note that this form is densely defined, nonnegative, and closed in $L^2(\dR^2)$. Hence it gives rise to a uniquely determined nonnegative self-adjoint operator 
$\Opf$, which is given by
\begin{equation} \label{def_free_Op}
  \Opf f  = \nbA^2 f,\qquad
  \dom \Opf  = \cH^2_\bA(\dR^2) := 
  \bigl\{ f \in \cH^1_\bA(\dR^2)\colon 
      \nbA^2 f \in L^2(\dR^2) \bigr\}.
\end{equation}
Note also that $C^\infty_0(\dR^2)$
is a core for the sesquilinear form $\fra_0$ since $C^\infty_0(\dR^2)$ is dense in $\cH^1_\bA(\dR^2)$.
The spectral properties 
and the Green function of the Landau Hamiltonian are recalled in the following proposition; 
cf. \cite[\S 10.4.1]{He},~\cite[\S 2.5.2]{HS02},~\cite[Section 2]{O06}, and~\cite{DMM75}.

\begin{prop} \label{proposition_Landau_level}
	Let $\Opf$ be the Landau Hamiltonian in~\eqref{def_free_Op}. Then 
	\begin{equation*}
		\s(\Opf) = \sess(\Opf) = 
		\{ B (2 q + 1)\colon q \in \dN_0 \},
	\end{equation*}
	i.e. the spectrum of $\Opf$ consists only of the eigenvalues $\Lm_q = B (2 q + 1)$, which are called Landau levels 
	and have infinite multiplicity.
	If $\lm \notin \s(\Opf)$, then the resolvent of $\Opf$ is given by
	\begin{equation*}
		((\Opf - \lm)^{-1} f)(x) 
		= 
		\int_{\dR^2} G_\lm(x, y) f(y) \dd y,
		\qquad f \in L^2(\dR^2),
	\end{equation*}
	with the Green function 
	\begin{equation} \label{def_G_lambda}
		G_\lambda(x,y) 
		= 
		\frac{1}{4 \pi} \Phi_B(x,y) \Gamma \left( \frac{B - \lambda}{2 B} \right) 
		U \left( \frac{B - \lambda}{2 B}, 1; \frac{B}{2} |x-y|^2 \right),
	\end{equation}
	where $U$ is the irregular confluent hypergeometric function (see \cite[$\S 13.1$]{AS64}), 
	$\Gamma$ denotes the Euler gamma function and 
	\begin{equation*}
		\Phi_B(x,y) 
		= 
		\exp\left[ 
			-\frac{\ii B}{2} (x_1 y_2 - x_2 y_1) 
			-\frac{B}{4} |x-y|^2 
		\right].
	\end{equation*}
\end{prop}

In the next proposition two variants of the so-called diamagnetic inequality are provided.
\begin{prop} \label{proposition_diamagnetic_inequality}
	Let $-\Delta$ be the self-adjoint Laplace operator in $L^2(\dR^2)$ defined on $H^2(\dR^2)$. 
	Then 
	for $\bb > 0$, 
	$\lm < 0$, and $f \in L^2(\dR^2)$ one has pointwise a.e. in $\dR^2$
	\begin{equation} \label{diamagnetic_inequality_resolvent}
		\bigl| (\Opf - \lm)^{-\bb} f \bigr| 
		\le
		(-\Dl - \lm)^{-\bb} |f|.
	\end{equation}
	Moreover, if $f \in \cH^1_\bA(\dR^2)$, then $|f|$ belongs to $H^1(\dR^2)$ and one has pointwise a.e. in $\dR^2$
	\begin{equation} \label{diamagnetic_inequality_gradient}
		\big| \nb|f| \big| \leq \big| \nbA f \big|.
	\end{equation}
\end{prop}
\begin{proof}
	Recall that by~\cite[Proposition 3.3.5]{H06} the formula
	\begin{equation*}
		(\sfA - \lm)^{-\bb}f 
		= 
		\frac{1}{\G(\bb)} 
		\int_0^\infty t^{\bb - 1} e^{\lm t} 
		\big( e^{-t \sfA} f\big) \dd t,\qquad \lm < 0,
	\end{equation*}
	holds for any self-adjoint nonnegative
	operator $\sfA$ acting in a Hilbert space $\cH$
	and for any $f\in\cH$; here $\G$ denotes the Euler gamma function.	
	Hence, the inequality
	\[
		|e^{-t\Opf}f| \le e^{-t\Dl}|f|
	\]  
	pointwise a.e. in $\dR^2$ (see, e.g., \cite[eq. (1.8)]{CFKS87})
	yields 
	\[
	\begin{aligned}
		\left| (\Opf - \lm)^{-\bb} f \right| 
		& =
		\frac{1}{\G(\bb)} 
		\left|
		\int_0^\infty t^{\bb - 1} e^{\lm t} 
		\big( e^{-t \sfA_0} f\big) \dd t\right|\\
		& \le
		\frac{1}{\G(\bb)} 
		\int_0^\infty t^{\bb - 1} e^{\lm t} 
		\big| e^{-t \sfA_0} f\big| \dd t\\
		& \le
		\frac{1}{\G(\bb)} 
		\int_0^\infty t^{\bb - 1} e^{\lm t} 
		\big( e^{-t \Dl} |f|\big) \dd t
		=
		 (-\Dl -\lm)^{-\bb} |f|.
	\end{aligned}
	\]
	The inequality~\eqref{diamagnetic_inequality_gradient} can be found in, e.g., \cite[Theorem~7.21]{LL01}.
\end{proof}

Using the diamagnetic inequality we can show that functions in $\cH^1_\bA(\dR^2)$ have 
traces in $L^2(\Sg)$. Here, and in the following, $\Sigma$ is the boundary of a bounded $C^{1,1}$-domain $\Omega\subset\dR^2$.

\begin{cor} \label{corollary_Sobolev_spaces}
	The mapping $C^\infty_0(\dR^2) \ni f \mapsto f|_\Sg$
	can be extended by continuity to a bounded operator 
	$\cH^1_\bA(\dR^2) \ni f\mapsto f|_\Sg \in L^2(\Sg)$. Moreover, for all $\eps > 0$  there exists $c(\eps) >0$ such that
	\begin{equation*}
		\| f|_\Sg \|_{L^2(\Sg)}^2 
		\leq 
		\eps\| \nbA f \|_{L^2(\dR^2;\dC^2)}^2
		+ 
		c(\eps) \| f \|_{L^2(\dR^2)}^2
	\end{equation*}
	holds for all $f \in \cH^1_\bA(\dR^2)$.
\end{cor}
\begin{proof}
	Let $\eps > 0$ and $f \in C_0^\infty(\dR^2)$. It is well known that there exists a constant $c(\eps) > 0$ independent of $f$ such that
	\begin{equation*}
		\| f|_\Sg \|_{L^2(\Sg)}^2 
		= 
		\big\| |f| \big|_\Sg \big\|_{L^2(\Sg)}^2
		\le 
		\eps 
		\big\| \nb |f| \big\|_{L^2(\dR^2;\dC^2)}^2
		+ c(\eps) 
		\big\| |f| \big\|_{L^2(\dR^2)}^2.
	\end{equation*}
	Using the diamagnetic inequality~\eqref{diamagnetic_inequality_gradient} we obtain 
	\begin{equation*}
		\| f|_\Sg \|_{L^2(\Sg)}^2
		\le 
		\eps\big\| \nbA f\big\|_{L^2(\dR^2;\dC^2)}^2
		+ 
		c(\eps) \| f \|_{L^2(\dR^2)}^2.
	\end{equation*}
	Since $C_0^\infty(\dR^2)$ is dense in the magnetic Sobolev space $\cH^1_\bA(\dR^2)$, the claim follows.
\end{proof}

Next we recall the definition of the Landau Hamiltonian on a domain $\Omg$ with Dirichlet boundary conditions. 
It is assumed here that $\Omg$ is 
either a bounded $C^{1,1}$-domain in $\dR^2$ or the complement of a bounded $C^{1,1}$-domain; then the compact boundary $\Sg  :=\p\Omg$ is a 
$C^{1,1}$-smooth curve. In analogy to \eqref{def_magnetic_Sobolev_space} the first order $L^2$-based magnetic Sobolev space is defined by
\begin{equation*} 
  \cH_\bA^1(\Omg)
  := 
  \big\{ f \in L^2(\Omg) \colon 
  |\nbA f| \in L^2(\Omg) \big\}
\end{equation*}
and is equipped with the Hilbert space inner product
\begin{equation*}
  (f, g)_{\cH^1_\bA(\Omg)} 
  := 
  (f, g)_{L^2(\Omg)} 
  + 
  \big(\nbA f,\nbA g\big)_{L^2(\Omg;\dC^2)},
  \quad f, g \in \cH^1_\bA(\Omg).
\end{equation*}
Note that $\cH_\bA^1(\Omg)$ coincides with $H^1(\Omega)$ if $\Omega$ is bounded or if $B=0$; if $B\not= 0$ then still 
$\cH_\bA^1(\Omg)$
and $H^1(\Omg)$ coincide locally. The standard Sobolev spaces on $\Omega$ and the boundary $\Sg$ are
denoted by $H^s(\Omega)$ and $H^t(\Sg)$, respectively. 
The magnetic counterpart
of the Sobolev space $H^1_0(\Omg)$ is defined as
\begin{equation*}
	\cH^1_{\bA, 0}(\Omg) 
	:= \ov{C_0^\infty(\Omg)}^{\| \cdot \|_{\cH^1_\bA(\Omg)}}.
\end{equation*}
Now consider the symmetric sesquilinear form
\begin{equation} \label{def_Dirichlet_form}
  \fra_{\rm D}^\Omg[f, g] 
  := 
  \big( \nbA f, \nbA g\big)_{L^2(\Omg, \dC^2)},
  \qquad \dom \fra_{\rm D}^\Omg = 
  \cH^1_{\bA, 0}(\Omg),
\end{equation}
and observe that $\fra_{\rm D}^\Omg$ is nonnegative, closed, and densely defined in $L^2(\Omg)$. The nonnegative 
self-adjoint operator $\OpD$
corresponding to  $\fra_{\rm D}^\Omg$ is the {\em Landau Hamiltonian on $\Omega$ with Dirichlet boundary conditions on $\Sg$}.
It is useful to note that for a bounded domain $\Omg$ the space $\cH^1_{\bA, 0}(\Omg)=H^1_0(\Omg)$ is compactly embedded in $L^2(\Omg)$ and hence
\begin{equation} \label{spectrum_Dirichlet_op}
	\sess(\OpD) = \varnothing.
\end{equation}

\subsection{Schatten von-Neumann ideals}
\label{section_Schatten-von_Neumann}

In this subsection we recall the definition and some properties of the Schatten-von Neumann ideals, which are used
in the proofs of our main results. We partially follow the presentation in~\cite{BLL13,BLL13_2}, where further references can be found. 
A very useful result on the Schatten-von Neumann property of operators that map into Sobolev spaces $H^s(\Sg)$ with $s>0$
is provided in Proposition~\ref{proposition_singular_values}.

Let $\cH, \cG$, and $\cK$ be separable Hilbert spaces. We denote the linear space of all 
bounded and everywhere defined operators from $\cH$ into $\cG$ by $\sB(\cH, \cG)$ and we write $\sB(\cH):=\sB(\cH, \cH)$.
We use the symbol $\sS_\infty(\cH, \cG)$ for the space of all compact
operators from $\cH$ to $\cG$ and 
$\sS_\infty(\cH) := \sS_\infty(\cH, \cH)$. The singular values $s_k(K)$, $k \in \dN$, of $K \in \sS_\infty(\cH, \cG)$ are the eigenvalues of the self-adjoint, 
nonnegative operator $(K^* K)^{1/2}\in\sS_\infty(\cH)$, which are ordered in a nonincreasing way 
with multiplicities taken into account. Note that
$s_k(K) = s_k(K^*)$ for $k\in\dN$.
For $p > 0$
the Schatten-von Neumann ideal of order $p$ is defined by
\begin{equation*}
	\sS_p(\cH, \cG) 
	:= 
	\left\{ 
		K \in \sS_\infty(\cH, \cG)\colon
			\sum_{k=1}^\infty s_k(K)^p < \infty 
	\right\}
\end{equation*}
and the weak Schatten-von Neumann ideal of order $p$ is defined by
\begin{equation*}
	\sS_{p, \infty}(\cH, \cG) 
	:= 
	\left\{ 
		K \in \sS_\infty(\cH, \cG)\colon
		s_k(K) = \cO(k^{-1/p}) 
	\right\}.
\end{equation*}
The (weak) Schatten-von Neumann ideals are ordered in the sense that for $0 < p < q$ one has
$\sS_p(\cH, \cG) \subset \sS_q(\cH, \cG)$
and $\sS_{p, \infty}(\cH, \cG) \subset 
\sS_{q, \infty}(\cH, \cG)$.
Moreover, we have
\begin{equation*}
	\sS_p(\cH, \cG)\subset \sS_{p,\infty}(\cH, \cG)
	\and
	\sS_{p, \infty}(\cH, \cG) \subset \sS_q(\cH, \cG).
\end{equation*}
The Schatten-von Neumann ideals are two-sided ideals, that is, for $K \in \sS_p(\cH,\cG)$ and $A \in \sB(\cH)$, $B\in\sB(\cG)$ 
one has $B K A \in \sS_p(\cH,\cG)$.
The analogous ideal property holds for the weak Schatten-von Neumann ideals.
Eventually, if 
$p, q > 0$ and $r$ are chosen such that $\frac{1}{r} = \frac{1}{p} + \frac{1}{q}$, then for $K_1 \in \sS_{p, \infty}(\cH,\cG)$
and $K_2 \in \sS_{q, \infty}(\cG,\cK)$ the product of these operators satisfies
\begin{equation} \label{equation_product_Schatten_von_Neumann} 
  K_2 K_1 \in \sS_{r, \infty}(\cH,\cK).
\end{equation}
Finally, let $\Sg \subset \dR^2$ be the boundary of a sufficiently smooth bounded domain.
It will be shown in the next proposition that operators with range in the Sobolev space $H^s(\Sg)$
belong to certain weak Schatten-von Neumann ideals. In the special case that $\Sigma$ is the boundary of a $C^\infty$-domain 
this property is known; cf.~\cite[Lemma 2.11]{BLL13}.
\begin{prop} \label{proposition_singular_values}
	Let $k \in\dN$ and let $\Sg$ be the boundary of a bounded $C^{k,1}$-domain $\Omg_{\rm i}\subset\dR^2$.
	Let $\cH$ be a separable Hilbert space and let $A\in\sB(\cH, L^2(\Sg))$ be such that $\ran A \subset H^{l/2}(\Sg)$ for some 
	$l \in \{ 1, \dots, 2k+1\}$.
	Then $$A \in \sS_{2/l, \infty}\big(\cH, L^2(\Sg)\big).$$	
\end{prop}

The proof of Proposition~\ref{proposition_singular_values} uses a general result from~\cite{AA96} and some properties of the acoustic single layer potential for the 
Helmholtz equation $-\Delta+1$, which will be briefly discussed for the convenience of the reader. Recall first from \cite[Section 7.4]{T14} that the Green function 
for the differential expression $-\Dl + 1$ in $\dR^2$ is given by $\frac{1}{2\pi}K_0(|\cdot|)$, where $K_0$ is the modified Bessel function of second kind and of order $0$.
It is well known that the boundary integral operator 
\begin{equation}\label{sinti}
(\cS \varphi)(x)=\frac{1}{2\pi}\int_\Sigma K_0(\vert x-y\vert)\varphi(y) \dd \sigma(y),\qquad x\in \Sigma, 
\end{equation}
gives rise to a bounded operator 
\begin{equation}\label{shalbe}
\cS_{-1/2}:H^{-1/2}(\Sigma)\rightarrow H^{1/2}(\Sigma);
\end{equation}
cf. \cite[Theorem~6.11]{M00}. In the following lemma we provide
some other useful properties of $\cS$. The proof of (i) is inspired by the proof of~\cite[Theorem 3]{C88}. 

\begin{lem}\label{singlelem}
Let $\Sg$ be the boundary of a bounded $C^{k,1}$-domain $\Omg_{\rm i}$ with $k \ge 1$. Then the following holds.
\begin{itemize}
 \item [{\rm (i)}] For all $s\in[-\frac{1}{2}, k - \frac{1}{2}]$ the restriction of $\cS_{-1/2}$ in \eqref{shalbe} 
 onto $H^s(\Sigma)$ leads to a bijective bounded operator
\begin{equation}\label{ss}
\cS_{s}:H^{s}(\Sigma)\rightarrow H^{s+1}(\Sigma).
\end{equation}  
 \item [{\rm (ii)}] The operator $\cS_0:L^2(\Sigma)\rightarrow H^{1}(\Sigma)$ in \eqref{ss} can be viewed as nonnegative bounded self-adjoint operator
 in $L^2(\Sigma)$ with $\ran\cS_0=H^{1}(\Sigma)$. The square root $\cS_0^{1/2}$ (defined via the functional calculus for self-adjoint operators) 
 is a nonnegative bounded self-adjoint operator in $L^2(\Sigma)$
 and also leads to a bijective bounded operator 
 \begin{equation*}
\cS_0^{1/2}:L^2(\Sigma)\rightarrow H^{1/2}(\Sigma).
\end{equation*}   
\end{itemize}
In particular, the operator $\cS_0^{l/2}:L^2(\Sigma)\rightarrow H^{l/2}(\Sigma)$ is bijective and bounded
for all $l \in \{ 1, \dots, 2k+1\}$.
\end{lem}

\begin{proof}
	(i) Note first that by \cite[Theorem 7.1 and Theorem~7.2]{M00} the operator
	$\cS_s$ in \eqref{ss}  is well defined as a linear map between the 
	respective Sobolev spaces. Next, \cite[Lemma 1.14(c)]{KG08} (see also \cite[Lemma 3.2]{MPS16}) implies $\ker\cS_{-1/2}=\{0\}$ and hence also 
	$\ker\cS_s=\{0\}$ for all $s\in[-\frac{1}{2}, k - \frac{1}{2}]$.
	Moreover, 
	\begin{equation}\label{ssbb123}
	\cS_{s}\in \sB(H^{s}(\Sg),H^{s+1}(\Sg)).
	\end{equation}
	In fact, for $s=-\frac{1}{2}$ this is a consequence of \cite[Theorem~6.11]{M00} and 
	for $s > -\frac{1}{2}$ the closed graph theorem implies \eqref{ssbb123} after it has been shown that 
	$\cS_s$ is a closed operator. For this consider $(\varphi_n)\subset H^{s}(\Sigma)$ such that 
	\begin{equation*}
		\varphi_n \arr \varphi \text{ in } H^{s}(\Sigma) \quad \text{and} \quad 
		\mathcal{S}_{s} \varphi_n \arr \psi \text{ in } H^{s+1}(\Sg) \quad \text{as } n \rightarrow \infty.
	\end{equation*}
	Then $\varphi \in H^{s}(\Sigma) = \dom \mathcal{S}_s$, $\varphi_n \rightarrow \varphi$ in $H^{-1/2}(\Sigma)$ as $n \rightarrow \infty$,
   and as $\mathcal{S}_{-1/2} \in \mathfrak{B}(H^{-1/2}(\Sigma), H^{1/2}(\Sigma))$
	we have $\mathcal{S}_{s} \varphi_n = \mathcal{S}_{-1/2} \varphi_n \rightarrow \mathcal{S}_{-1/2} \varphi$ in $H^{1/2}(\Sigma)$ 
	for $n \rightarrow \infty$. On the other hand, since 
	$H^{s+1}(\Sigma)$ is continuously embedded in $H^{1/2}(\Sigma)$ we also have
	$\mathcal{S}_{-1/2} \varphi_n = \mathcal{S}_{s} \varphi_n \rightarrow \psi$ in $H^{1/2}(\Sigma)$. Thus $\mathcal{S}_{s} \varphi = \mathcal{S}_{-1/2} \varphi = \psi$ and hence
	$\mathcal{S}_{s}$ is closed. 
	
	In order to verify that $\mathcal S_s$ in \eqref{ss} is surjective for 
	$s=j-1/2$ and $j=\{0,1,...,k\}$, consider $\psi \in H^{j+1/2}(\Sigma)$. Then, in particular, $\psi \in H^{1/2}(\Sigma)$, and 
	as $\cS_{-1/2}$ is a Fredholm operator of index zero by \cite[Theorem~7.6]{M00} and $\ker\cS_{-1/2}=\{0\}$ it is clear that
	$\cS_{-1/2}$ in \eqref{shalbe} is bijective. Hence there exists a unique $\varphi\in H^{-1/2}(\Sigma)$ such that $\cS_{-1/2}\varphi=\psi$.
	Eventually, it follows from \cite[Theorem~7.16~(i)]{M00} that $\varphi \in H^{j-1/2}(\Sigma)$, so that $\cS_{j-1/2}\varphi=\psi$.
	We have shown that the operators  $\cS_s$ in \eqref{ss} for 
	$s=j-1/2$ and $j\in\{0,1,...,k\}$ are bijective. Now it follows from standard interpolation techniques that 
	$\cS_s\in\sB(H^s(\Sigma),H^{s+1}(\Sigma))$ is bijective for all $s\in[-\frac{1}{2}, k - \frac{1}{2}]$.
 
   (ii) It is clear that $\cS_0$ is a bounded operator in $L^2(\Sigma)$ with $\ran\cS_0=H^1(\Sigma)$. To see that $\cS_0$ is nonnegative and self-adjoint in $L^2(\Sigma)$
        let $\Omg_{\rm e}:=\dR^2\setminus\overline \Omg_{\rm i}$ and decompose the functions $u\in L^2(\dR^2)$ in 
        the two components $u_j := u|_{\Omg_j}$, $j \in \{ \rm i, \rm e \}$. 
	For $\varphi \in L^2(\Sg)$ there exists a unique $u\in H^1(\dR^2)$ such that
	$-\Dl u_j + u_j = 0$, $j \in \{ \rm i, \rm e\}$, and 
	$\p_\nu u_{\rm i}|_{\Sg} - \p_\nu u_{\rm e}|_{\Sg} = \varphi$, and, moreover,  one has  
	$\cS_0\varphi = u|_{\Sg}$ (see, e.g., \cite[Proposition~3.2~(ii) and Remark~3.3]{BLL13}, where 
	$\mathcal{S}_0 = \widetilde{M}(-1)$ in the notation of \cite{BLL13}).
	Hence, the first Green identity leads to
	\[
	\begin{aligned}
		(\cS_0\varphi,\varphi)_{L^2(\Sg)} 
		& =
		\left(u|_\Sg, \p_\nu u_{\rm i}|_{\Sg} - \p_\nu u_{\rm e}|_{\Sg} \right)_{L^2(\Sg)}\\
		& =
		( u_{\rm i}, \Dl u_{\rm i})_{L^2(\Omg_{\rm i})}+
		(u_{\rm e}, \Dl u_{\rm e})_{L^2(\Omg_{\rm e})} +  (\nb u, \nb u)_{L^2(\dR^2;\dC^2)} \\
		& = (u,u)_{L^2(\dR^2)} +  (\nb u, \nb u)_{L^2(\dR^2;\dC^2)},
	\end{aligned}
	\]
	which implies that $\cS_0$ is a nonnegative self-adjoint operator in $L^2(\Sigma)$.
	Eventually, by the interpolation result~\cite[Theorem 3.2]{AS13}, which applies to $\cS_0^{-1}$,  we have $\dom\cS_0^{-1/2} = H^{1/2}(\Sg)$.
	Thus, we get $\ran\cS_0^{1/2} = H^{1/2}(\Sg)$
	and $\cS_0^{1/2}$  is a bijective bounded operator 
	from $L^2(\Sg)$ onto $H^{1/2}(\Sg)$.
	
	The last assertion is a direct consequence of (i) and (ii). In fact, for even $l$ this follows from repeated applications of (i), whereas for odd $l$
	we use $\cS_0^{l/2}=\cS_0^{(l-1)/2}\cS_0^{1/2}$, (ii) and repeated applications of (i). 
\end{proof}

\begin{proof}[Proof of Proposition~\ref{proposition_singular_values}]
Assume that $\ran A\subset H^{l/2}(\Sigma)$ for some $l \in \{ 1, \dots, 2k+1\}$. 
It will be shown first that the operator $A_l\colon \mathcal{H} \rightarrow H^{l/2}(\Sigma)$, $A_l f = A f$, is continuous. 
In fact, consider a sequence $(f_n)\subset \cH$ such that 
	\begin{equation*}
		f_n \arr f \text{ in } \cH \quad \text{and} \quad 
		A_l f_n \arr g \text{ in } H^{l/2}(\Sg) \quad \text{as } n \rightarrow \infty.
	\end{equation*}
	Then $f \in \mathcal{H} = \dom A_l$ and as $A \in \mathfrak{B}(\mathcal{H}, L^2(\Sigma))$
	we have $A_l f_n = A f_n \rightarrow A f$ in $L^2(\Sigma)$ for $n \rightarrow \infty$. On the other hand, since 
	$H^{l/2}(\Sigma)$ is continuously embedded in $L^2(\Sigma)$ we also have
	$A f_n = A_l f_n \rightarrow g$ in $L^2(\Sigma)$. Thus, $A_l f = A f = g$ and hence
	$A_l$ is closed and defined on all of $\cH$. This implies $A_l\in\sB(\mathcal{H}, H^{l/2}(\Sigma))$.
	
	Now consider the operator $\cS_0$ in Lemma~\ref{singlelem} as a nonnegative bounded self-adjoint operator
 in $L^2(\Sigma)$ and note that the integral kernel in \eqref{sinti} is the kernel of the polyhomogeneous pseudodifferential operator $(-\Delta+1)^{-1}$,
 which is of order $-2$. Therefore, \cite[Theorem 2.9]{AA96} applies (for the class $\cP^0$) and yields that 
 \begin{equation*}
		\cS_0 \in \sS_{1,\infty}(L^2(\Sg)).
	\end{equation*}
	Hence, the spectral theorem implies 
	\begin{equation} \label{estimate_sk_M_0}
		\cS_0^t \in \sS_{1/t,\infty}(L^2(\Sg)), 
		\qquad t > 0.
	\end{equation}
	On the other hand, it follows from Lemma~\ref{sinti} that $\cS_0^{l/2}\in\sB(L^2(\Sigma),H^{l/2}(\Sigma))$ is bijective 
	and hence also
	$\cS_0^{-l/2}\in\sB(H^{l/2}(\Sigma),L^2(\Sigma))$.
	Since 
	\begin{equation*}
	A = \cS_0^{l/2} \cS_0^{-l/2} A_l\quad\text{and}\quad \cS_0^{-l/2} A_l \in \sB(\cH, L^2(\Sg))
	\end{equation*}
	we conclude from \eqref{estimate_sk_M_0} with $t=l/2$ that
	$A \in \sS_{2/l, \infty}(\mathcal{H}, L^2(\Sg))$.
\end{proof}

\subsection{Compact perturbations of self-adjoint operators}\label{compis}

In this subsection we discuss some special results on compact perturbations.
In the following let  $T$ be a self-adjoint operator in a Hilbert space	$\cH$ and let $\Lambda\in\dR$ be an isolated eigenvalue of $T$ of 
infinite multiplicity with the corresponding eigenprojection $P_\Lambda$. Furthermore, let $\tau_\pm > 0$ be such that
	\[
		(\Lambda - 2\tau_-,\Lambda + 2\tau_+) \cap \s(T) = \{\Lambda\}.
	\]
	For a self-adjoint operator $W$ in $\cH$ with corresponding spectral measure $E_W(\cdot)$ we denote by
	\begin{equation}\label{www}
	W_+=\int_0^\infty \lambda\, \dd E_W(\lambda)\quad \text{and}\quad W_-=-\int_{-\infty}^0 \lambda\, \dd E_W(\lambda)
	\end{equation}
	the nonnegative and nonpositive part of $W$, respectively. Note that both $W_+$ and $W_-$ are nonnegative self-adjoint operators in $\cH$ and that
	the identities
	$W=W_+-W_-$ and $\vert W\vert= W_++W_-$ hold.  
	Now assume, in addition, that the self-adjoint operator $W$ in $\cH$ is compact and
	denote by
	\[
		\mu_1^\pm \ge \mu_2^\pm \ge \mu_3^\pm \ge \dots 
		\ge 0
	\]
	the eigenvalues of $P_\Lambda W_\pm P_\Lambda \ge 0$ in nonincreasing order with multiplicities taken into account and by
	\begin{equation}\label{nota1}
		\lm_1^- \le \lm_2^- \le \dots 
		\le \Lambda \le \dots \le \lm_2^+ \le \lm_1^+ 
	\end{equation}
	the eigenvalues of $T + W$ in the interval
	$(\Lambda - \tau_-, \Lambda+\tau_+)$. If there are only finitely many $\lambda_k^+>\Lambda$ we set $\lambda_k^+=\Lambda$ for all larger $k\in\dN$, 
	the same convention is used for  $\lambda_k^-$.
In the next proposition we state double-sided
estimates of $\lm_k^\pm$ in terms of $\mu_k^\pm$,
assuming that either $W_- = 0$ or $W_+ = 0$.
\begin{prop}\cite[Proposition 2.2]{RP07}
	\label{prop:RP}
	Let $T$ and $W= W_+ - W_-$ be as above. Then the following holds.
	\begin{itemize}
	 \item[{\rm (i)}] If  $\rank(P_\Lm W_+ P_\Lm) = \infty$
	and $W_- = 0$  then the eigenvalues of $T+W$ accumulate to $\Lm$
	only from above and for $\eps > 0$ 
	there exists $\ell \in\dN$	such that 
	\[
		(1-\eps)\mu_{k+\ell}^+
		\le 
		\lm_k^+ - \Lm
		\le
		(1+\eps)\mu_{k-\ell}^+
	\]
	for all $k\in\dN$ sufficiently large.
	\item[{\rm (ii)}] If  $\rank(P_\Lm W_- P_\Lm) = \infty$
	and $W_+ = 0$  then the eigenvalues of $T+W$ accumulate to $\Lm$
	only from below and for $\eps > 0$ 
	there exists $\ell \in\dN$	such that
	\[(1-\eps)\mu_{k+\ell}^-
		\le 
		\Lm - \lm_k^- 
		\le
		(1+\eps)\mu_{k-\ell}^-
	\]
	for all $k\in\dN$ sufficiently large.
	\end{itemize}
\end{prop}

\begin{remark}\label{remmipr}
If $\rank(P_\Lm W_+ P_\Lm) < \infty$ or $\rank(P_\Lm W_+ P_\Lm) < \infty$  in Proposition~\ref{prop:RP} then still the upper estimates 
\begin{equation*}
 \lm_k^+ - \Lm
		\le
		(1+\eps)\mu_{k-\ell}^+\quad\text{or}\quad \Lm - \lm_k^- 
		\le
		(1+\eps)\mu_{k-\ell}^-,
\end{equation*}
respectively,
for $k\in\dN$ sufficiently large remain valid. This follows from the proof of \cite[Proposition 2.2]{RP07}.
\end{remark}

In the following,
we denote 
by $\cN_\cI(A)$ the number
of eigenvalues of a self-adjoint operator $A$
in an interval $\cI\subset\dR\sm \sess(A)$ counted with multiplicities. The next
standard perturbation lemma will be useful. We state
it for the convenience of the reader.
\begin{lem}\label{lem:BS}
	\cite[\S 9.3, Theorem 3 and \S 9.4,  Lemma 3]{BS87}
	Let $C,D\in\sB(\cH)$ be self-adjoint operators 
	such that $V := D-C$ is compact with $\s(V) \subseteq [v_-,v_+]$. Let $\cI = (c_-,c_+)\subset\dR$ be an interval satisfying
	$\cI\cap\sess(C) = \varnothing$. Then the following
	hold.
	\begin{myenum}
			\item If $\rank V = r < \infty$, 
			then $\cN_\cI(C) \le \cN_\cI(D) + r$.
			\item If $\cI' := (c_- + v_-, c_+ + v_+)\cap\sess(C) = \varnothing$, 
			then $\cN_{\cI}(C) \le \cN_{\cI'}(D)$.
	\end{myenum}
\end{lem}	
The next proposition complements Proposition~\ref{prop:RP} and Remark~\ref{remmipr}. 
If the definiteness assumption  on $W$ is dropped then one still obtains one-sided estimates on $\lm_k^+ - \Lm$ and $\Lm - \lm_k^-$. 
\begin{prop}\label{prop:RPpm}
Let $T$ and $W= W_+ - W_-$ be as above. Then the following holds.
	\begin{itemize}
	 \item[{\rm (i)}] 
	 For $\eps > 0$ 
	there exists $\ell \in\dN$	such that 
\[
		\lm_k^+ - \Lambda \le
		(1+\eps)\mu_{k-\ell}^+
	\]
	for all $k\in\dN$ sufficiently large.
	\item[{\rm (ii)}] 
	For $\eps > 0$ 
	there exists $\ell \in\dN$	such that 
	\[		
		\Lambda - \lm_k^- \le
		(1+\eps)\mu_{k-\ell}^-
	\]
	for all $k\in\dN$ sufficiently large.
	\end{itemize}
\end{prop}
\begin{proof}
It suffices to prove item (i); the proof of (ii) is analogous. Moreover, it is no restriction to assume $\Lm=0$.   
	Throughout the proof
	we denote the eigenvalues
	in the interval
	$[0, \tau_+)$ of the operator $S_U = T + U$ 
	with a generic compact self-adjoint perturbation $U$ by
	\begin{equation}\label{nota2}
		\lm_1^+(S_U) \ge \lm_2^+(S_U) \ge \lm_3^+(S_U)\ge \dots \ge 0,
	\end{equation}
	which are repeated with multiplicities taken into account.

	Let us fix $\eps > 0$. Since $W_-$ is compact and nonnegative, it can be decomposed as
	$W_- = F_- + R_-$,
	where $\rank F_- = r_0 < \infty$ and
	the operator $R_-$ satisfies $\s(R_-) \subseteq [0,\tau_+]$.
	Hence, the operator $S_W=T+W$ can be written as
	\[
		S_W = T + W_+ - F_- - R_-.
	\]
	If $\rank(P_\Lm W_+ P_\Lm)= \infty$ then 
	Proposition~\ref{prop:RP}~(i) applies for the operator $S_{W_+}=T+W_+$ and yields
	\begin{equation}\label{esti1}
		\lm_k^+(S_{W_+})
		\le (1+\eps)	\mu_{k-\ell_0}^+
	\end{equation}
	for some $\ell_0\in\dN$ and all $k \in\dN$ sufficiently large; in the case $\rank(P_\Lm W_+ P_\Lm)< \infty$ the estimate \eqref{esti1} 
	follows from  Remark~\ref{remmipr}. Since the rank of $F_-$ is finite, 
	Lemma~\ref{lem:BS}\,(i) with $C = S_{W_+}$
	and $D = S_{W_+ - F_-}$ and \eqref{esti1} imply
	\begin{equation}\label{esti2}
		\lm_k^+(S_{W_+ - F_-})  
		\le 
		\lm_{k - r_0}^+(S_{W_+}) 
		\le
		(1+\eps)\mu_{k-\ell_1}^+	
	\end{equation}
   for $\ell_1 := \ell_0 + r_0$ and all $k \in\dN$ sufficiently large.
	Further, we set
	\begin{equation}\label{r1}
		r_1 := \cN_{[\tau_+, 2\tau_+)}(S_{W_+ - F_-})\in\dN_0.
	\end{equation}
	Note that the operator $S_W$ can be decomposed as $S_W = S_{W_+ - F_-} - R_-$. Now we apply Lemma~\ref{lem:BS}\,(ii) with $C = S_W$, $D=S_{W_+ - F_-}$,
	$V=R_-$, $[v_-,v_+]=[0,\tau_+]$ and $\cI=(t,\tau_+)$ for $t \in (0,  \tau_+)$, and conclude together with \eqref{r1} that
	\[
		\cN_{(t, \tau_+)}(S_W)
		\le 
		\cN_{(t,  2\tau_+)}(S_{W_+ - F_-})
		=
		\cN_{(t, \tau_+)}(S_{W_+ - F_-})
		+ r_1.
	\]
	Since we only consider eigenvalues in the interval $[0,\tau_+)$ (see \eqref{nota1} and \eqref{nota2}) this estimate and \eqref{esti2} 
	with $\ell := \ell_1 + r_1$ lead to
	\[
		\lm_k^+(S_W) 
		\le
		\lm_{k-r_1}^+(S_{W_+ - F_-}) 
		\le
		(1+\eps)\mu_{k-\ell}^+
	\]		
	for all $k\in\dN$ sufficiently large.
\end{proof}
The last proposition of this subsection characterizes
the total variation of the discrete spectrum under a trace class perturbation.
\begin{prop}\cite[Corollary 5.1.2]{DKH13}\label{prop:H}
	Let $C, D\in \sB(\cH)$ be self-adjoint operators 
	such that $D-C\in\sS_1(\cH)$.
	Then 
	\[
		\sum_{\lm\in\sd(C)} \text{\rm dist}\,(\lm,\s(D)) < 	\infty.
	\]
\end{prop}
The above proposition is a variant of an older theorem by T.~Kato~\cite[Theorem~II]{K87}. In this form,
the statement is particularly convenient to apply
for perturbed Landau Hamiltonians.

\subsection{A class of Toeplitz-type operators}
\label{section_Toeplitz}
In this subsection we define and recall properties of Toeplitz-type operators related to Landau
Hamiltonians. 
In the following let $\Sigma$ be the boundary of a bounded $C^{1,1}$-domain $\Omega\subset\dR^2$ and let $\G\subset\Sigma$ be a closed subset of $\Sg$. Note that 
$\G$  and $\Sigma$ are both compact subsets of $\dR^2$.
In particular, $\G$ can be a subarc of $\Sigma$ with two endpoints, a union of finitely many such subarcs,
or coincide with $\Sigma$. The latter three geometric settings are of particular importance for our considerations.
In fact, in our applications $\G$ is typically
the essential support of the strength $\aa \in L^\infty(\Sigma)$ of the $\dl$-interaction 
for the Hamiltonian $\Op$. Recall that the (essential) support of $\aa$ is a closed subset of $\Sg$
uniquely defined by
\[
\supp \aa := \Sg \sm\bigcup
\big\{\s\subset\Sg\colon \s~\text{is open and}~~~
~~~\aa = 0~\text{a.e. in}~\s\big\};
\]
cf.~\cite[Section 1.5]{LL01}.
We introduce the Hilbert space $L^2(\G)$
with the usual inner product $(\cdot,\cdot)_{L^2(\G)}$, defined by means of
the natural arc-length measure on $\Sg$ restricted to $\G$. We denote by $|\G|$ the arc-length measure of $\G$, that is, the length of $\G$.		
Corollary~\ref{corollary_Sobolev_spaces} implies that the trace mapping
$\cH^1_\bA(\dR^2)\ni u \mapsto u|_\G \in L^2(\G)$
is well defined and bounded.

We denote by $P_q\colon L^2(\dR^2)\arr L^2(\dR^2)$, $q\in\dN_0$, the orthogonal projection onto the
spectral subspace corresponding to the eigenvalue $\Lm_q = B(2q+1)$ of the Landau Hamiltonian $\Opf$; cf. Proposition~\ref{proposition_Landau_level}.
Following the lines of~\cite[Section 4]{RP07},
we introduce a family of Toeplitz-type operators, 
which correspond to the formal product $P_q\dl_\G P_q$.
\begin{prop}
	For all $q\in\dN_0$ the symmetric sesquilinear form
	\begin{equation}\label{eq:form_Tq}
		\frt_q^\G[f,g] 
		:= 
		\big( (P_q f)|_\G, (P_q g)|_\G\big)_{L^2(\G)},
		\quad
		\dom\frt_q^\G = L^2(\dR^2),
	\end{equation}
	is well defined and bounded.
\end{prop}	
\begin{proof}
	Note that for any $f \in L^2(\dR^2)$ we have
	\begin{equation*}
	 \frt_q^\G[f,f] =\|(P_q f)|_\Gamma\|^2_{L^2(\G)} \le \|(P_q f)|_\Sg\|^2_{L^2(\Sg)} \le \varepsilon \|\nbA P_q f\|^2_{L^2(\dR^2)} 
	 + c(\varepsilon) \|P_q f\|_{L^2(\dR^2)}^2	 
	\end{equation*}
	with $\varepsilon>0$ and $c(\varepsilon)>0$ by Corollary~\ref{corollary_Sobolev_spaces}. Using \eqref{def_form_free_Op} and the first representation theorem 
	we find
\begin{equation*}
  \|\nbA P_q f\|_{L^2(\dR^2)}^2=\fra_0[P_q f,P_q f] = (\Opf P_qf,P_qf)_{L^2(\dR^2)} =\Lambda_q\Vert P_q f\Vert^2_{L^2(\dR^2)}, 
\end{equation*}
and hence $\frt_q^\G[f,f]\leq c'(\varepsilon) \|P_q f\|_{L^2(\dR^2)}^2$ for some  $c'(\varepsilon)>0$. This implies that the symmetric sesquilinear form
$\frt_q$ is  well defined and bounded.
\end{proof}	
The Toeplitz-type operators we are interested in can now be defined.

\begin{dfn}\label{def:Tn} For $q\in\dN_0$
    the bounded self-adjoint operator in $L^2(\dR^2)$ associated with the form
    $\frt_q^\G$ in~\eqref{eq:form_Tq} is denoted by $T_q^\G$.
\end{dfn}

Note that $T_q^\G = T_q^{\G'}$ for closed subsets $\Gamma,\Gamma'\subset\Sigma$ that satisfy $|(\G\sm\G')\cup(\G'\sm\G)| = 0$
and that $T_q^\G = 0$ if $|\G| = 0$.
Certain fundamental spectral properties of such Toeplitz-type operators were obtained in~\cite{FP06, RP07}. The  operators $T_q^\G$ can be viewed as variants 
of a better studied class of
Toeplitz operators
$P_q V P_q$, where $V\colon\dR^2\arr\dR$ is a regular function~\cite{FP06, RP07, PR11, RT08}. Very roughly speaking in our considerations the 
$\dl$-distribution supported on $\G$ plays the role of $V$.
Before we provide some properties of $T_q^\G$ which are
essential for our considerations we first introduce a notion from 
potential theory, see~\cite[\S II.4]{La72},~\cite[Appendix A.VIII]{ST10}, and~\cite[\S III.1]{GM05}.
\begin{dfn}
	The \emph{logarithmic energy} of a measure $\mu\ge0$  on $\dR^2$
	is given  by
	\[
	    I(\mu) := \int_{\dR^2}\int_{\dR^2}    \ln\frac{1}{|x-y|}
	    \dd \mu(x)\dd \mu(y).
	\]
	The \emph{logarithmic capacity} of a compact set $\cK\subset\dR^2$ is defined by
	\[
	    {\rm Cap}\,(\cK) := 
	    \sup\big\{e^{-I(\mu)}\colon \mu \ge 0~\text{measure on}~\dR^2,\,
	    {\rm supp}\,\mu\subset\cK,\, \mu(\cK) = 1    \big\}.
	\]
\end{dfn}	

It is well known (see, e.g., \cite[\S~III]{GM05}) that the supremum in the definition of the logarithmic capacity is in fact a maximum.
This maximum is attained by the so-called  {\it equilibrium measure}.
In the next proposition we collect some useful
properties of the logarithmic capacity.	
\begin{prop}\cite[\S III]{GM05}
	\label{prop:cap}
		Let $\cK,\cL\subset\dR^2$ be compact sets, let $\eta>0$ and consider the compact set 
		$U_\eta(\cK) := \{x\in\dR^2\colon \text{\rm dist}\,(x,\cK)\le \eta\}$.		Then the following holds.
		\begin{myenum}
			\item ${\rm Cap}\,(\cK)\le {\rm Cap}\,(\cL)$ if $\cK\subset\cL$.
			\item ${\rm Cap}\,(U_\eta(\cK))\arr {\rm Cap}\,(\cK)$ as $\eta \arr 0^+$.
		\end{myenum}	
\end{prop}

Using the notion of logarithmic capacity of $\G$ one gets
an asymptotic upper bound on the singular values of $T_q^\G$
and even exact asymptotics for them, provided that
$\G$ is smooth. Note that the singular values of $T_q^\G$
coincide with its eigenvalues since $T_q^\G$
is a self-adjoint nonnegative operator. 
Item (i) in the next proposition can be seen as consequence of~\cite[Proposition 4.1\,(i)]{RP07}.
	For the convenience of the reader we provide a short proof. Item (ii) coincides with~\cite[Proposition 4.1\,(ii)]{RP07}.
\begin{prop}
    \label{prop:Toeplitz}
    Let $\G\subset\Sigma$ be a closed subset with $|\G| > 0$.
    Then the self-adjoint Toeplitz-type operator $T_q^\G$, $q\in\dN_0$, in Definition~\ref{def:Tn}
    is compact and its singular values satisfy:
    \begin{myenum}
        \item
        $\limsup_{k\arr\infty}\big(k!\, s_k(T_q^\G)\big)^{1/k}
        \le \frac{B}{2}\big({\rm Cap}\,(\G)\big)^2$;
        \item
        $\lim_{k\arr\infty}\big(k!\, s_k(T_q^\G)\big)^{1/k}
        =
        \frac{B}{2}\big({\rm Cap}\,(\G)\big)^2$
        if, in addition, $\G$ is a $C^\infty$-smooth arc with two endpoints. In particular, the operator $T_q^\G$ is of infinite rank.
    \end{myenum}
\end{prop}

\begin{proof} 
(i)  Denote by  $U_\eta:=U_\eta(\G)\subset\dR^2$
     the $\eta$-neighborhood of $\G$ for $\eta > 0$ as in Proposition~\ref{prop:cap}
     and fix a cut-off function $\omg \in C^\infty_0(\dR^2)$, $0\le 
\omg\le 1$, such that $\omg \equiv 1$
     on $\G$ and $\omg \equiv 0$ on $\dR^2\sm U_{\eta}$.

     For $f \in L^2(\dR^2)$ the function $\omg P_q f$ belongs to $\dom\Opf$ and by Corollary~\ref{corollary_Sobolev_spaces}
     we have
     \begin{equation}\label{eq:tq_bnd}
     \begin{aligned}
     \frt_q^\G[f,f]  =\|(P_qf)|_\G\|^2_{L^2(\G)}
     & =
     \|(\omg P_q f)|_\G\|^2_{L^2(\G)} \\
     & \le
     \varepsilon \|\nbA\omg P_q f\|^2_{L^2(\dR^2;\dC^2)}
     +c(\varepsilon)\| \omg P_q f\|^2_{L^2(\dR^2)}\\
     & \le
     \varepsilon \|\nbA\omg P_q f\|^2_{L^2(\dR^2;\dC^2)}
         +c(\varepsilon)\| P_q f\|^2_{L^2(U_\eta)}
     \end{aligned}
     \end{equation}
     for $\varepsilon>0$ and suitable $c(\varepsilon)>0$.
     For $f \in L^2(\dR^2)$ it follows from \cite[Proposition 4.2]{Ra} that 
     \begin{equation}\label{eq:IMS_bnd}
     \begin{aligned}
         \|\nbA\omg P_qf\|_{L^2(\dR^2;\dC^2)}^2
         & =
         ( \Opf P_qf, \omg^2 P_qf)_{L^2(\dR^2)}
         + (|\nb\omg|^2P_qf,P_q f)_{L^2(\dR^2)}\\
         &
         = \Lm_q(\omg^2 P_qf, P_qf)_{L^2(\dR^2)}
         +
         (|\nb\omg|^2P_qf,P_q f)_{L^2(\dR^2)}\\
         & \le
         c'\|P_q f\|^2_{L^2(U_\eta)},
     \end{aligned}
     \end{equation}
     where we have also used that the supports of $\omg^2$ and $|\nb\omg|^2$ are 
contained in $U_\eta$ and $c'>0$ is some constant. 
     Hence, if $\chi_{\eta}$ denotes the characteristic function of $U_\eta$ we conclude from \eqref{eq:tq_bnd} and~\eqref{eq:IMS_bnd}
     the operator inequality
     \begin{equation*}
         T_q^\G \le
         c'' P_q\chi_{\eta} P_q,\qquad c''=\varepsilon c'+ c(\varepsilon).
     \end{equation*}
     Using~\cite[Proposition 4.1\,(i)]{RP07} we obtain that
     \[
     \limsup_{k\arr\infty}\big(k!\, s_k(T_q^\G)\big)^{1/k}
     \le     \limsup_{k\arr\infty}\big(k!\, s_k(P_q \chi_{\eta} P_q)\big)^{1/k}
     = \frac{B}{2}\big({\rm Cap}\,(U_\eta)\big)^2.
     \]
     Finally, the desired inequality follows from 
Proposition~\ref{prop:cap}~(ii) upon passing
     to the limit $\eta\arr 0^+$.

     The asymptotics in~(ii) are shown 
in~\cite[Proposition 4.1\,(ii)]{RP07}.
\end{proof}

It is a priori not clear that the rank of the Toeplitz-type operator $T_q^\G$
is infinite without extra regularity assumption on $\G$.
However, for $q = 0$ this claim can be deduced from a result
by D.~Luecking in \cite{L08} (see also its extension in~\cite{RS10}).
To this aim, we define $\Psi(z) := \frac14 B|z|^2$
and consider the \emph{Segal-Bargmann} (or \emph{Fock}) space of analytic functions 
\begin{equation*}
    \cF^2 :=
    \big\{f\colon\dC\arr\dC\colon f\,\,\text{is analytic},
    e^{-\Psi}f \in L^2(\dC)\big\}.
\end{equation*}
It was shown in \cite[Section 4.2]{RP07} that
the multiplication operator
\begin{equation}\label{eq:mult_U}
	U\colon \cF^2\arr L^2(\dR^2),
	\qquad Uf :=  e^{-\Psi} f,
\end{equation}
is unitary from $\cF^2$ onto the subspace  $\ran P_0 = \ker(\Opf - \Lm_0)$ of $L^2(\dR^2)$. Using this equivalence it follows easily that
the rank of $T_0^\G$ is infinite.
\begin{prop}\label{prop:rank0}
    Let $\G\subset\Sigma$ be a closed subset with
    $|\G| > 0$.
    Then the self-adjoint Toeplitz-type operator $T_0^\G$
    $(q=0)$ in Definition~\ref{def:Tn} has infinite rank.
\end{prop}
\begin{proof}
    According to the construction in \cite[Section 4.2]{RP07} the operator $T_0^\G$ is unitarily equivalent via $U$ in~\eqref{eq:mult_U} 
    to the classical Toeplitz operator $T_\mu^\cF$ on $\cF^2$
    defined in~\cite[Eq. (1.6)]{RS10}
    with the corresponding compactly supported measure $\mu$ in $\dR^2$ given by
    \[
        G\mapsto\mu(G) := \int_{G\cap\G} \exp(-2\Psi(z)\big) \dd \s(z),
        \qquad G\subset \dC\simeq\dR^2.
    \]
    Note that the measure $\mu$ can not be represented as a sum of finitely
    many point measures. Therefore,
    by~\cite[Theorem 1.1]{RS10} the operator $T_\mu^\cF$,
    and hence also $T_0^\G$, are of infinite rank.
\end{proof}
Later in this paper we show for the case $\Gamma=\Sigma$ in Corollary~\ref{cor:infrank} that 
the rank of $T_q^\Sg$ is infinite for all $q\in\dN$ with $C^{1,1}$-smooth $\Sg$ using a technique rather different from the one in~\cite{FP06, RP07}. 
In this context we remark that one can go beyond $C^{1,1}$-smoothness
up to a Lipschitz boundary by a small modification of the method.

\section{A quasi boundary triple for Landau Hamiltonians}\label{sec:qbt}

In this section we construct a quasi boundary triple which is suitable to define and study Landau Hamiltonians with $\dl$-perturbations 
supported on $C^{1,1}$-curves. The notion of quasi boundary triples
and their Weyl functions is recalled in Appendix~\ref{app:A}.
From now on we shall assume that the following hypothesis holds.

\begin{hyp}\label{hypohyp}
	Let $\Omg_{\rm i}$ be a bounded $C^{1,1}$-domain with the boundary $\Sg := \p\Omg_{\rm i}$
	and let $\Omg_{\rm e} := \dR^2 \sm\ov{\Omg_{\rm i}}$. The unit normal vector field pointing outward of $\Omg_{\rm i}$
	(and hence inward of $\Omg_{\rm e}$) will be denoted by $\nu$.
\end{hyp}
In the following, 
$\p_\nu = \nu\cdot\nb$ and
$\p_\nu^\bA = -\ii\nu\cdot\nbA=\p_\nu-\ii\nu\cdot \bA$ stand for the normal derivative and the magnetic normal derivative with respect to the normal 
vector $\nu$ pointing outward of $\Omg_{\rm i}$.
Further, we set 
\begin{equation*}
	\cD_{\rm i} = H^{3/2}_\Dl(\Omg_{\rm i})
	:=
	\big\{ f_{\rm i} \in H^{3/2}(\Omg_{\rm i})\colon 
			\Dl f_{\rm i} \in L^2(\Omg_{\rm i}) \big\},
\end{equation*}
where the Laplacian is understood in the distributional sense. 
Recall that the Dirichlet and Neumann trace maps
\[
	\cD_{\rm i}\ni f\mapsto f|_{\Sg}\in H^1(\Sg)
	\qquad\text{and}\qquad
	\cD_{\rm i}\ni f\mapsto \p_\nu f|_{\Sg}\in L^2(\Sg)
\]
are bounded and surjective; cf.~\cite[Lemma 3.1 and 3.2]{GM11}.
Note that the spaces $H^{3/2}_\Delta$ appear also in~\cite{BLL13} in the treatment 
of non-magnetic Schr\"{o}dinger operators with $\dl$-interactions.

In the next lemma 
we provide variants of the first and second Green identity in the present situation.

\begin{lem}\label{lemma1}
	For $f_{\rm i},g_{\rm i}\in\cD_{\rm i}$ one has $\nbA^2 f_{\rm i}, \nbA^2 g_{\rm i} \in L^2(\Omg_{\rm i})$ and the following holds.
	\begin{myenum}
		\item 
		$(\nbA^2 f_{\rm i}, 
				g_{\rm i})_{L^2(\Omg_{\rm i})}
		=
		(\nbA f_{\rm i},
				\nbA g_{\rm i})_{L^2(\Omg_{\rm i};\dC^2)} 
			-
		(\p_\nu^\bA f_{\rm i}|_\Sg, g_{\rm i}|_\Sg)_{L^2(\Sg)}$.
		\item  
		$(\nbA^2 f_{\rm i},g_{\rm i})_{L^2(\Omg_{\rm i})}
		-
		(f_{\rm i},\nbA^2 g_{\rm i})_{L^2(\Omg_{\rm i})}
		=
		(f_{\rm i}|_\Sg,
		\p_\nu^\bA g_{\rm i}|_\Sg)_{L^2(\Sg)}-
		(\p_\nu^\bA f_{\rm i}|_\Sg,
		g_{\rm i}|_\Sg)_{L^2(\Sg)}$. 
\end{myenum}
\end{lem}

\begin{proof} 
	For $f_{\rm i}\in\cD_{\rm i}$ and all 
	$h_{\rm i} \in C_0^\infty(\Omg_{\rm i})$ one has
	\begin{equation*}
	\begin{split}
		\big(f_{\rm i}, 
			\nbA^2 h_{\rm i}\big)_{L^2(\Omg_{\rm i})}
		&= 
		\big( 
			f_{\rm i},	(-\Dl + 2\ii\bA\cdot\nb 
				 + \bA^2) h_{\rm i} 
		\big)_{L^2(\Omg_{\rm i})} \\
	   &= 
	   \big(
	   	-\Dl f_{\rm i}, h_{\rm i}
	   \big)_{L^2(\Omg_{\rm i})}
	   + 
	   \big(
		   (2\ii\bA\cdot\nb  + \bA^2)f_{\rm i}, h_{\rm i} 
	   \big)_{L^2(\Omg_{\rm i})},
	  \end{split}
	\end{equation*}
	where $\nb\cdot\bA = 0$ and also $\cH^1_\bA(\Omg_{\rm i}) = H^1(\Omg_{\rm i})$  were used. This shows 
	\[
		\nbA^2 f_\text{\rm i}
		=
		-\Dl f_{\rm i}
		+
		(2\ii\bA\cdot\nb +  \bA^2)f_{\rm i}
		\in L^2(\Omg_{\rm i}).
	\]	
	
	It follows from the divergence theorem and the particular form of $\bA$ that
	\[
		\cB[f_{\rm i}, g_{\rm i}] := (\ii\nb f_{\rm i},
			\bA g_{\rm i})_{L^2(\Omg_{\rm i};\dC^2)} 
		- 
		(\bA f_{\rm i}, \ii\nb g_{\rm i})_{L^2(\Omega_\text{i};\dC^2)}
		=
		(\ii(\nu \cdot \bA f_{\rm i})|_\Sg,g_{\rm i}|_\Sg)_{L^2(\Sg)}
	\]
	holds for $f_{\rm i},g_{\rm i}\in\cD_{\rm i}$. 
	Now a simple computation
	\[
	\begin{aligned}
	(\nbA f_{\rm i},&
	\nbA g_{\rm i})_{L^2(\Omg_{\rm i};\dC^2)} -
	(\nbA^2 f_{\rm i},g_{\rm i})_{L^2(\Omg_{\rm i})} 
	\\
	& =
	\big[(\nb f_{\rm i}, \nb g_{\rm i})_{L^2(\Omg_{\rm i};\dC^2)} - \big(-\Dl f_{\rm i},g_{\rm i}\big)_{L^2(\Omg_{\rm i})}\big] 
	-
	\cB[f_{\rm i}, g_{\rm i}]
	\\
	& =	
	(\p_\nu f_{\rm i}|_\Sg, g_{\rm i}|_\Sg)_{L^2(\Sg)} 
	-
	(\ii(\nu \cdot \bA f_{\rm i})|_\Sg,g_{\rm i}|_\Sg)_{L^2(\Sg)} =
	 (\p_\nu^\bA f_{\rm i}|_\Sg, g_{\rm i}|_\Sg)_{L^2(\Sg)}
	\end{aligned}
	\]
	yields the identity in~(i). The identity in~(ii)
	follows from~(i).
\end{proof}

In order to define an appropriate counterpart of the space $\cD_{\rm i}$ on the exterior domain $\Omg_{\rm e}$
one has to pay some attention to the properties of the functions in a neighborhood of $\infty$. This leads to the following construction.
Fix some bounded open set $K$ such that $\overline{\Omega_\text{i}} \subset K$ and define
\begin{equation*}
	\cD_{\rm e} 
	:= 
	\big\{
	f_{\rm e} \in \cH^1_\bA(\Omg_{\rm e})\colon
	\nbA^2 f_{\rm e} \in L^2(\Omg_{\rm e}),
	f_{\rm e} \uhr (K \cap \Omg_{\rm e}) \in H^{3/2}_\Dl(K \cap \Omg_{\rm e}) 
	\big\},
\end{equation*}
where 
$H^{3/2}_\Dl(K \cap \Omg_{\rm e}) 
:= \{ h \in H^{3/2}(K \cap \Omg_{\rm e})\colon \Dl h \in L^2(K \cap \Omg_{\rm e}) \}$. Using
\cite[Lemma 3.1 and 3.2]{GM11} one checks that the Dirichlet and Neumann trace maps
\[
	\cD_{\rm e}\ni f\mapsto f|_{\Sg}\in H^1(\Sg)
	\qquad\text{and}\qquad
	\cD_{\rm e}\ni f\mapsto \p_\nu f|_{\Sg}\in L^2(\Sg)
\]
are bounded and surjective.

In the same way as in Lemma~\ref{lemma1} one obtains the following statements. Observe that $\nu$ is pointing inwards in $\Omega_\text{e}$,
which leads to different signs compared to Lemma~\ref{lemma1}.
\begin{lem}\label{lemma2}
	For $f_{\rm e},g_{\rm e}\in\cD_{\rm e}$ the following holds.
	\begin{myenum}
		\item  
		$(\nbA^2 f_{\rm e},g_{\rm e})_{L^2(\Omg_{\rm e})}
		=
		(\nbA f_{\rm e},\nbA g_{\rm e})_{L^2(\Omg_{\rm e};\dC^2)}  
		+
		(\p_\nu^\bA f_{\rm e}|_\Sg, 
		g_{\rm e}|_\Sg)_{L^2(\Sg)}$.
		\item 
		$(\nbA^2 f_{\rm e},g_{\rm e})_{L^2(\Omg_{\rm e})}
		-
		(f_{\rm e},	\nbA^2 g_{\rm e})_{L^2(\Omg_{\rm e})}
	 	=
	 	-(f_{\rm e}|_\Sg,\p_\nu^\bA g_{\rm e}|_\Sg)_{L^2(\Sg)}
	 	+
	 	(\p_\nu^\bA f_{\rm e}|_\Sg, g_{\rm e}|_\Sg)_{L^2(\Sg)}
	 	$.
	\end{myenum}
\end{lem}
Next, we introduce the operator $T$ acting in $L^2(\dR^2)$ by
\begin{equation*}
    T f     := 
    \nbA^2 f_{\rm i} \oplus \nbA^2 f_{\rm e},\quad
    \dom T  
    := 
    \big\{ f = f_{\rm i} \oplus 
    f_{\rm e} \in \cD_{\rm i} \oplus \cD_{\rm e}\colon
    f_{\rm i}|_\Sg = f_{\rm e}|_\Sg \big\},
\end{equation*}
and the trace mappings $\G_0, \G_1\colon \dom T \arr L^2(\Sg)$ by
\begin{equation} \label{def_Gamma}
	\G_0 f 
	:= \p_\nu^\bA f_{\rm i}|_\Sg-\p_\nu^\bA f_{\rm e}|_\Sg =
	\p_\nu f_{\rm i}|_\Sg - \p_\nu f_{\rm e}|_\Sg \and
	\G_1 f := f|_\Sg.
\end{equation}

Then we have the following result, which is important for our further investigations in the next section. 

\begin{thm} \label{theorem_triple}
	Let $T$ be as above and define 
	\begin{equation*} 
		S := \Opf \uhr 
		\bigl\{  f \in \cH^2_\bA(\dR^2)\colon f|_\Sg = 0  \bigr\}.
	\end{equation*}
	Then $S$ is a densely defined, closed, symmetric operator and $\{ L^2(\Sg), \G_0, \G_1 \}$ is a 
	quasi boundary triple for $T\subset S^*$. Moreover, 
	$T \uhr \ker \G_0$ coincides with the Landau Hamiltonian $\Opf$ and $\ran \G_0 = L^2(\Sg)$.
\end{thm}

\begin{proof}
	We apply Theorem~\ref{ratebitte} to prove the claim. 
	Using	that the traces of $f_{\rm i}, f_{\rm e}$ and $g_{\rm i}, g_{\rm e}$ coincide on $\Sg$
	for $f, g \in \dom T$, we get 
	from Lemma~\ref{lemma1}\,(ii) and
	Lemma~\ref{lemma2}\,(ii) that
	\begin{equation*}
	\begin{aligned}
		(&T f, g)_{L^2(\dR^2)} - (f, T g)_{L^2(\dR^2)}\\
		& 
		\! = \!
		(\nbA^2 f_{\rm i},g_{\rm i})_{L^2(\Omg_{\rm i})}
		-
		(f_{\rm i},	\nbA^2 g_{\rm i})_{L^2(\Omg_{\rm i})}
		+
		(\nbA^2 f_{\rm e},g_{\rm e})_{L^2(\Omg_{\rm e})}
		-
		(f_{\rm e},	\nbA^2 g_{\rm e})_{L^2(\Omg_{\rm e})}\\
		& \! = \!
		(f_{\rm i}|_\Sg,\p_\nu^\bA g_{\rm i}|_\Sg)_{L^2(\Sg)}
		\!-\!
		(\p_\nu^\bA  f_{\rm i}|_\Sg, g_{\rm i}|_\Sg)_{L^2(\Sg)} 
		\!-\!
		(f_{\rm e}|_\Sg,\p_\nu^\bA g_{\rm e}|_\Sg)_{L^2(\Sg)}
		\!+\!
		(\p_\nu^\bA f_{\rm e}|_\Sg, g_{\rm e}|_\Sg)_{L^2(\Sg)}\\
		& \! = \!
		(f|_\Sg,\p_\nu^\bA g_{\rm i}|_\Sg-\p_\nu^\bA g_{\rm e}|_\Sg)_{L^2(\Sg)}
		\!-\!
		(\p_\nu^\bA  f_{\rm i}|_\Sg - \p_\nu^\bA f_{\rm e}|_\Sg, g|_\Sg)_{L^2(\Sg)} \\
		& \! = \! (\G_1 f, \G_0 g)_{L^2(\Sg)} 
		- (\G_0 f, \G_1 g)_{L^2(\Sg)},
	\end{aligned}
	\end{equation*}
	that is, the Green identity holds.
	
	Next, it follows from the Green identity that the operator $T \uhr \ker \G_0$ is symmetric in $L^2(\dR^2)$. It is easy to see that the self-adjoint
	Landau Hamiltonian $\Opf$ is contained in $T \uhr \ker \G_0$ and consequently $\Opf=T \uhr \ker \G_0$. 
	Furthermore, let $\chi \in C_0^\infty(\dR^2)$ be a cut-off function which is identically equal to one in a neighborhood 
	of $\Omg_{\rm i}$ and set $\chi_{\rm e}=\chi|_{\Omg_{\rm e}}$. Then the space
	\begin{equation*}
		\left\{ \begin{pmatrix} f_{\rm i} \\ \chi_{\rm e}f_{\rm e}\end{pmatrix}
			\colon 
			f_{\rm i}\in H^2(\Omg_{\rm i}),\, f_{\rm e} \in 
			H^2(\Omg_{\rm e}),
			f_{\rm e}|_\Sg = f_{\rm i}|_\Sg 
		\right\},
	\end{equation*}
	is contained in $\dom T$. Thus,
	it follows from the properties of the trace mappings~\cite[Theorem~3]{M87} that 
	\begin{equation*}
		H^{1/2}(\Sg) \times H^{3/2}(\Sg)
		\subset\ran \begin{pmatrix} \G_0\\  \G_1\end{pmatrix},
	\end{equation*}
	i.e. $\ran( \Gamma_0, \Gamma_1)^\top$ is dense in $L^2(\Sigma) \times L^2(\Sigma)$. 
	Furthermore, it is clear that also $\ker( \Gamma_0, \Gamma_1)^\top = \dom S$ is dense in $L^2(\mathbb{R}^2)$.
	
	Finally, to show that $\G_0$ is surjective we use the single layer potential
	$\mathsf{SL}\colon L^2(\Sg)\arr L^2(\dR^2)$
	associated to $\Sg$ and the Helmholtz equation
	$-\Dl + 1$; cf.~\cite[Chapter 6]{M00}. To be more precise, 
	for $\varphi \in L^2(\Sg)$ define the function $f := \widetilde\chi \mathsf{SL} \varphi$, 
	where $\widetilde\chi\in C^\infty_0(\dR^2)$ is a cutoff function such that $\chi \equiv 1$ in a neighborhood of $\Sigma$.
	Then using the properties of the single layer potential from \cite[Theorem~6.11 and Theorem~6.13]{M00} we see that $f$ belongs to $\dom T$
	and $\Gamma_0 f = \varphi$. Now Theorem~\ref{ratebitte} leads to the assertions. 
\end{proof}

In the next step we compute the $\gamma$-field and the Weyl function associated to the quasi boundary triple 
$\{ L^2(\Sg), \G_0, \G_1 \}$ from Theorem~\ref{theorem_triple}.
Recall that $G_\lambda$ in \eqref{def_G_lambda} is the integral kernel of the resolvent of the Landau Hamiltonian.

\begin{prop} \label{proposition_gamma_Weyl}
	Let $\lm \in \rho(\Opf)$ and let $G_\lm$ be given by~\eqref{def_G_lambda}. Then the values of the $\gamma$-field  $\gg(\lm)$ and of 
	the Weyl function $M(\lm)$ satisfy the following.
	\begin{myenum}
		\item The operator $\gg(\lm)\in\sB(L^2(\Sg), L^2(\dR^2))$ is given by
		\begin{equation*}
			\gg(\lm) \varphi(x) 
			= 
			\int_\Sg G_\lm(x,y) \varphi(y) \dd \s(y),
			\quad 
			\varphi \in L^2(\Sg),\, x \in \dR^2,
		\end{equation*}
		and belongs to the weak Schatten-von Neumann ideal $\sS_{2/3, \infty}(L^2(\Sg), L^2(\dR^2))$.
		\item
		The adjoint operator $\gg(\lm)^*\in\sB(L^2(\dR^2), L^2(\Sg))$ is given by
		\begin{equation*}
			\gg(\lm)^* f(x) = 
			\int_{\dR^2} 
			G_{\ov\lm}(x,y) f(y) \dd y,
			\quad 
			f \in L^2(\dR^2),\, x \in \Sg,
		\end{equation*}
		and belongs to the weak Schatten-von Neumann ideal $\sS_{2/3, \infty}(L^2(\dR^2), L^2(\Sg))$.
		\item The operator $M(\lm)\in\sB(L^2(\Sg))$ is given by
		\begin{equation*}
			M(\lambda) \varphi(x) = \int_\Sigma G_\lambda(x,y) \varphi(y) \dd \sigma(y),
			\quad \varphi \in L^2(\Sigma),\, x \in \Sigma,
		\end{equation*}
		and belongs to the weak Schatten-von Neumann ideal $\sS_{1, \infty}(L^2(\Sg))$.
	\end{myenum}
	In particular, the operators $\gg(\lm)$, $\gg(\lm)^*$, and $M(\lm)$
	are compact.
\end{prop}

\begin{proof}
	First, we verify statement (ii). Since $\gg(\lm)^* = \G_1 (\Opf- \ov\lm)^{-1}$, the  
	representation of $\gg(\lm)^*$ follows directly from the form of the resolvent of $\Opf$ in 	Proposition~\ref{proposition_Landau_level}. 
	Moreover, as $\ran (\Opf - \ov\lm)^{-1} = \dom \Opf = \cH^2_\bA(\dR^2)$, and since this space coincides
	locally with $H^2(\dR^2)$, we conclude from the boundedness of $\Sg$ and the mapping properties of the trace map that 
	$\ran \gg(\lm)^* = \G_1 (H^2(\dR^2)) = H^{3/2}(\Sg)$.
	Therefore, Proposition~\ref{proposition_singular_values}
	with $k = 1$ and $l = 3$ shows
	$\gg(\lm)^* \in \sS_{2/3, \infty}(L^2(\dR^2), L^2(\Sg))$.
	
	The claim of item (i) follows from (ii) by taking adjoints, as $\overline{G}_{\overline{\lambda}}(y,x) = G_\lambda(x,y)$
	and $\dom \gamma(\lambda) = \ran \Gamma_0 = L^2(\Sigma)$.
	
	Finally, the representation of the Weyl function follows immediately from $M(\lm) = \G_1 \gg(\lm)$ and 
        item~(i). In particular, since $\ran M(\lm) \subset \ran \G_1 \subset H^1(\Sg)$
	we conclude from Proposition~\ref{proposition_singular_values}
	with $k = 1$ and $l = 2$
	that $M(\lm) \in \sS_{1, \infty}(L^2(\Sg))$. 
\end{proof}

Next, we provide a useful estimate on the decay of the Weyl function $M$, which is an application of
Theorem~\ref{wowsuperthm} for the quasi boundary triple in Theorem~\ref{theorem_triple}. Recall that 
$\min\sigma(\Opf)=B\geq 0$; cf. Proposition~\ref{proposition_Landau_level}.

\begin{prop} \label{proposition_decay_M}
	For all $\eps \in (0,\frac12)$ and all $w_0<B$ there exists a constant $D > 0$ such that 
	\begin{equation*}
		\| M(\lm) \|\le\frac{D}{|\lm - B|^{1/2-\eps}},\qquad \lm < w_0.
	\end{equation*}
\end{prop}
\begin{proof}
	Let $w_0<B$ and fix $\lm < w_0$.
	We check that the operator $\G_1 (\Opf - \lm)^{-\bb}$
	is bounded and everywhere defined for $\bb = \frac{1}{4} + \frac{\eps}{2}$. In fact,
	let $-\Delta$ be the free Laplacian defined on $H^2(\dR^2)$ and let $f \in L^2(\dR^2)$.
	Using the diamagnetic inequality~\eqref{diamagnetic_inequality_resolvent}, the trace theorem
	and the boundedness of $(-\Delta - \lm)^{-\bb}\colon L^2(\dR^2) \arr H^{2 \bb}(\dR^2)$
	we find constants $C_1,C_2 > 0$ such that
	\begin{equation*}
	\begin{split}
		\big\|\G_1(\Opf - \lm)^{-\bb}f\big\|_{L^2(\Sg)}^2
		&= 
		\int_\Sg 
		\big| (\Opf - \lm)^{-\bb} f \big|^{2} \dd \s
	   \le 
	   \int_\Sg
	   \big| (-\Dl - \lm)^{-\bb} |f| \big|^{2} \dd \s \\
		&= 
		\big\| \bigl((-\Dl - \lm)^{-\bb} |f|\bigr) \big|_\Sg
		\big\|_{L^2(\Sg)}^2
		\le
		C_1\big\| (-\Dl - \lm)^{-\bb} |f| \big\|_{H^{2 \bb}(\dR^2)}^2 \\
		&\le C_2 \| f \|_{L^2(\dR^2)}^2.
	\end{split}
	\end{equation*}
	Hence $\G_1 (\Opf - \lm)^{-\bb}$
	is bounded. Now Theorem~\ref{wowsuperthm} leads to the assertion.
\end{proof}

Finally, we provide an auxiliary lemma which is essential in the proof of Proposition~\ref{infrank}.
Recall that $\sfA^{\Omega_{\rm i}}_{\rm D}$ denotes the Landau Hamiltonian in $\Omega_{\text i}$ with Dirichlet boundary 
conditions, which was defined via the quadratic form in~\eqref{def_Dirichlet_form}. Since $\Omega_{\rm i}$ is 
bounded one has $\sess(\OpD) = \varnothing$; cf. \eqref{spectrum_Dirichlet_op}.

\begin{lem}\label{kers}
	For any $q\in\dN_0$ one has 
	\[
		\dim\ker(S-\Lm_q)
		\leq 
		\dim\ker (\sfA^{\Omega_{\rm i}}_{\rm D}-\Lm_q)
	\]
	and, in particular, the space $\ker(S-\Lm_q)$ is finite-dimensional.
\end{lem}

\begin{proof}
	Assume that 
	$\dim\ker (\sfA^{\Omega_{\rm i}}_{\rm D}-\Lm_q)=k$ 
	for some $k\in\dN_0$ and suppose that $h_1,\dots,h_{k+1}\in\ker(S-\Lm_q)$ are linearly independent. Set $h_j^{\rm i} = h_j|_{\Omg_{\rm i}}$ and $h_j^{\rm e} = h_j|_{\Omg_{\rm e}}$ 
	for $j = 1,2,\dots, k+1$.
	It is clear that $h_1^{\rm i},\dots,h_{k+1}^{\rm i}\in\ker(\sfA_{\rm D}^{\Omg_{\rm i}}-\Lm_q)$ and hence we conclude without loss of generality that  there exist $\beta_1,\dots,\beta_k\in\dC$ such that
	\begin{equation}\label{innen}
		h_{k+1}^{\rm i} = \sum_{j=1}^k
		\bb_j h_j^{\rm i}.
	\end{equation}
	Note that also $h_1^{\rm e},\dots,h_{k+1}^{\rm e}\in\ker(\sfA^{\Omega_{\rm e}}_{\rm D}-\Lm_q)$ and as $h_1,\dots,h_{k+1}\in\dom S$
	it follows that 
	\begin{equation*}
		\p_\nu h_j^{\rm e}\vert_\Sg 
		= 
		\p_\nu h_j^{\rm i}\vert_\Sg,
		\qquad j=1,\dots,k+1.
	\end{equation*}
	Now observe that for the function
	\begin{equation*}
		g^{\rm e} := 
		h_{k+1}^{\rm e} - 
		\sum_{j=1}^k\bb_j h_j^{\rm e} \in\ker(\sfA^{\Omega_{\rm e}}_{\rm D}-\Lm_q)
	\end{equation*}
	one has by \eqref{innen}
	\begin{equation*}
		\p_\nu g^{\rm e}\vert_\Sg
		=
		\p_\nu h_{k+1}^{\rm e}
		\vert_\Sg
		-
		\sum_{j=1}^k\bb_j \p_\nu h_j^{\rm e}\vert_\Sg
		=
		\p_\nu h_{k+1}^{\rm i}\vert_\Sg
		-
		\sum_{j=1}^k\bb_j \p_\nu h_j^{\rm i}\vert_\Sg = 0
	\end{equation*}
	and hence unique continuation~\cite{W93}
	(see also the proof of Proposition 2.5 in~\cite{BR12}) yields $g^{\rm e}=0$. But this implies
	\begin{equation*}
		h_{k+1}^{\rm e} 
		= 
		\sum_{j=1}^k\bb_j h_j^{\rm e}
	\end{equation*}
	and together with \eqref{innen} we conclude
	\begin{equation*}
		h_{k+1}=\sum_{j=1}^k\bb_j h_j;
	\end{equation*}
	a contradiction, since by assumption the functions 
	$h_1,\dots,h_{k+1}$ are linearly independent.
\end{proof}

\section{Landau Hamiltonians with singular potentials}
\label{sec:landau}

In this section we define and study the Landau Hamiltonian $\Op$ with a $\delta$-potential
supported on $\Sg$ with a position-dependent real strength $\aa\in L^\infty(\Sg)$.
We shall use the quasi boundary triple $\{ L^2(\Sg), \G_0, \G_1 \}$ from Theorem~\ref{theorem_triple} and its $\gg$-field and Weyl function 
to derive various properties for the operator $\Op$ and its resolvent. As in the previous section we assume that Hypothesis~\ref{hypohyp} holds.

\subsection{Definition of $\Op$, self-adjointness, and qualitative spectral properties}
\label{ssec:defOp}

Let us start with the rigorous definition of $\Op$.

\begin{dfn} \label{definition_A_delta_alpha}
	Let $\aa \in L^\infty(\Sg)$ be a real function. The Landau Hamiltonian with $\dl$-potential of strength $\aa$ 
	supported on $\Sg$ is defined as the operator 
	$\Op := T \uhr \ker(\G_0 + \aa \G_1)$ in $L^2(\dR^2)$, or, more explicitly 
	\begin{equation}\label{eq:op_delta}
	\begin{split}
		\Op f & := (\nbA^2 f_{\rm i}) \oplus (\nbA^2 f_{\rm e}) \\
		\dom \Op & := \big\{ 
			f = f_{\rm i} \oplus f_{\rm e}\in 
			\cD_{\rm i} \oplus \cD_{\rm e}\colon
			f_{\rm i}|_\Sg = f_{\rm e}|_\Sg, 
			\p_\nu f_{\rm e}|_\Sg 
					- \p_\nu f_{\rm i}|_\Sg = \aa f|_\Sg 
		\big\}.
	\end{split}
	\end{equation}
\end{dfn}

Note that the jump of the normal derivatives $\p_\nu f_{\rm e}|_\Sg - \p_\nu f_{\rm i}|_\Sg$ in \eqref{eq:op_delta} can also be replaced by the jump of the
magnetic normal derivatives $\p_\nu^\bA f_{\rm e}|_\Sg-\p_\nu^\bA f_{\rm i}|_\Sg$; cf. \eqref{def_Gamma}.

In the next theorem we prove that $\Op$ is self-adjoint, obtain a version of the Birman-Schwinger principle, and derive a Krein-type resolvent formula,
which also implies that the resolvent difference
of $\Op$ and $\Opf$ is compact. Moreover, we estimate
the decay of the singular values for this resolvent
difference. As a direct consequence, we obtain
a characterisation of the essential spectrum for
$\Op$.

\begin{thm} \label{theorem_Krein}
	Let $\{ L^2(\Sg), \G_0, \G_1 \}$ be the quasi boundary triple from Theorem~\ref{theorem_triple} with $\Opf = T\uhr\ker\G_0$,
	$\gg$-field $\gg$ and Weyl function $M$. 
	Let $\aa \in L^\infty(\Sg)$ be real and let $\Op$ be as in Definition~\ref{definition_A_delta_alpha}.
	Then the following assertions hold.
	\begin{myenum}
		\item $\Op$ is a self-adjoint operator in $L^2(\dR^2)$.
		\item $\lm \notin \s(\Opf)$ is an eigenvalue of $\Op$ if and only if
		$-1 \in \s_{\rm p}(\aa M(\lm))$.
		\item For all $\lm \in \rho(\Op) \cap \rho(\Opf)$ one has $( 1 + \aa M(\lm))^{-1}\in\sB (L^2(\Sg))$ and 
		\begin{equation}\label{resid}
			(\Op - \lm)^{-1} - (\Opf - \lm)^{-1} 
			=  
			-\gg(\lm) 
			\big( 1 + \aa M(\lm) \big)^{-1}\aa\gg(\ov\lm)^*.
		\end{equation}
		\item  For all $\lm \in \rho(\Op) \cap \rho(\Opf)$ the singular values $s_k$ of the resolvent difference \eqref{resid} are in $\cO(k^{-3})$ and, in particular, the operator 
		\eqref{resid} is in $\sS_p(L^2(\dR^2))$
		for all $p > \frac13$. 
		\item 
		$\sess(\Op) = \sess(\Opf)= \s(\Opf) =
		\{ B(2q+1)\colon q\in\dN_0\}$.
	\end{myenum}
\end{thm}
\begin{proof}
	Items (i)-(iii) follow from Corollary~\ref{besserescor} with $B=-\aa$. 
	In fact, we have
	$\| \aa M(\lambda_0) \| < 1$ for $\lambda_0 <0$ with sufficiently large absolute value using $\aa \in L^\infty(\Sg)$
	and Proposition~\ref{proposition_decay_M}.
	To prove~(iv) note that $( 1 + \aa M(\lm))^{-1}\aa\in\sB (L^2(\Sg))$. 
	By Proposition~\ref{proposition_gamma_Weyl}
	we have $\gg(\lm) \in \sS_{2/3,\infty}(L^2(\Sg),L^2(\dR^2))$ and $\gg(\ov\lm)^* \in \sS_{2/3,\infty}(L^2(\dR^2),L^2(\Sg))$,
	and together with~\eqref{equation_product_Schatten_von_Neumann} 
	this implies (iv). Finally, (v) is an immediate consequence of~(iv) and well-known perturbation results.
\end{proof}

\begin{remark}
	The estimate of the singular values in Theorem~\ref{theorem_Krein}\,(iv) is known to be sharp in the absence of a magnetic field (that is, $B=0$) 
	if both $\Sg$ and $\aa$ are $C^\infty$-smooth; cf.~\cite[Theorem C~(i)]{BLL13} and~\cite[Theorem 5.1]{BGLL15}. 
	The magnetic case is new in this setting. A similar estimate for the magnetic Robin Laplacian on an exterior domain 
	is contained in~\cite[Lemma 2.2 and Remark 2.4]{GS17}.
\end{remark}	
  
In the following proposition we show that $\Op$ can also be defined as the self-adjoint operator corresponding 
to the quadratic form $\frm$ in~\eqref{eq:form}; cf.~\cite{O06}.

\begin{prop} \label{proposition_delta_form}
	The symmetric sesquilinear form $\frm$ 
	\begin{equation}\label{eq:form222}
  \frm[f,g]
   =
  \big(\nbA f, \nbA g)_{L^2(\dR^2; \dC^2)}
  +
  \big(\aa f|_\Sg, g|_\Sg\big)_{L^2(\Sg)},\quad
  \dom\frm  = \cH_\bA^1(\dR^2),
\end{equation}
	is densely defined, closed, bounded from below,
	and $C^\infty_0(\dR^2)$ is a core for $\frm$.
	The corresponding self-adjoint operator coincides with $\Op$ in Definition~\ref{definition_A_delta_alpha}
	and, in particular, the operator $\Op$ is bounded from below and satisfies $\min\s(\Op)\leq\min\s(\Opf) = B$.
\end{prop}

\begin{proof}
	Recall first that the form $\fra_0$ corresponding to the Landau Hamiltonian in~\eqref{def_form_free_Op}
	is densely defined, nonnegative, closed, and $C_0^\infty(\mathbb{R}^2)$ is a core for $\fra_0$. Consider the form
	\begin{equation*}
		\frb_\aa[f, g] 
		:= 
		\int_\Sg \aa f|_\Sg \ov{g|_\Sg}\, \dd \s,
		\qquad 
		\dom \frb_\aa := \cH^1_\bA(\dR^2),
	\end{equation*}
	and note that $\frb_\aa$ is well defined by Corollary~\ref{corollary_Sobolev_spaces}.
	It is clear that $\frm = \fra_0 + \frb_\aa$ is densely defined.
	Choose $\varepsilon > 0$ such that $\varepsilon \| \alpha \|_{L^\infty(\Sigma)} < 1$. Then 
	by Corollary~\ref{corollary_Sobolev_spaces} 
	\begin{equation} \label{form_bound_delta}
	\begin{split}
		\big| \frb_\aa[f, f] \big| &
		\leq 
		\int_\Sg |\aa| |f|_\Sg|^2 \dd \s
		\leq 
		\| \aa \|_{L^\infty(\Sg)} 
		\| f \|_{L^2(\Sg)}^2 \\
		&\leq 
		\eps\|\aa\|_{L^\infty(\Sg)}
		\|\nbA f\|_{L^2(\dR^2;\dC^2)}^2
		+ 
		c(\eps)\|\aa\|_{L^\infty(\Sg)}
		\|f\|_{L^2(\dR^2)}^2
	\end{split}
	\end{equation}
	holds for all $f \in \cH^1_\bA(\dR^2)$. Therefore,
	$\frb_\aa$ is form bounded with respect to $\fra_0$ with form bound
	less than one and hence the KLMN theorem (see \cite[Theorem~X.17]{RSII} or \cite[$\S6$ Theorem 1.33 and Theorem 2.1]{K95}) implies that 
	$\frm$ is closed, bounded from below, and $C_0^\infty(\dR^2)$ is a core of $\frm$.
	
	In order to show that the corresponding self-adjoint operator coincides with $\Op$
	let $f \in \dom \Op \subset \cH^1_\bA(\dR^2)$ and $g \in C_0^\infty(\dR^2)$. Then
	$$
	\alpha f\vert_\Sigma=\p_\nu f_{\rm e}|_\Sg - \p_\nu f_{\rm i}|_\Sigma=\p_\nu^\bA f_{\rm e}|_\Sg - \p_\nu^\bA f_{\rm i}|_\Sigma
	$$
	and hence 
	it follows from Lemma~\ref{lemma1} and Lemma~\ref{lemma2} that
	\begin{equation*}
		(\Op f, g)_{L^2(\dR^2)} 
		= 
		\big( \nbA f, \nbA g \big)_{L^2(\dR^2;\dC^2)}
		+ 
		\big(\p_\nu^\bA f_{\rm e}|_\Sg - \p_\nu^\bA f_{\rm i}|_\Sigma, g|_\Sg\big)_{L^2(\Sg)}
		= \frm[f, g].
	\end{equation*}
	Since $C_0^\infty(\dR^2)$ is a core for  $\frm$ it follows from the first representation theorem \cite[$\S6$ Theorem 2.1]{K95} that 
	the self-adjoint operator $\Op$ is contained in the self-adjoint operator representing the form $\frm$, and hence both coincide.
	This also implies that $\Op$ is bounded from below (with the same lower bound as the form $\frm$) and 
	the inequality $\min\s(\Op)\leq\min\s(\Opf)=B$ follows from Proposition~\ref{proposition_Landau_level}.
\end{proof}

For later use we note here a simple consequence of 
Proposition~\ref{proposition_delta_form}: it follows 
from~\eqref{form_bound_delta} that there are constants $C_1, C_2$ with 
$C_1 \in (0,1)$ such that 
\begin{equation*}
   \begin{split}
     \| \nbA f \|_{L^2(\mathbb{R}^2; \mathbb{C}^2)}^2 = \fra_\aa[f] - 
\frb_\aa[f] \leq \fra_\aa[f] + C_1 \|\nbA f\|_{L^2(\dR^2;\dC^2)}^2 + C_2 
\|f\|_{L^2(\mathbb{R}^2)}^2
   \end{split}
\end{equation*}
holds for all $f \in 
\mathcal{H}^1_\bA(\mathbb{R}^2)$,
where $\frb_\aa$ is defined as in the proof above.
Hence, there exist constants $c_1, c_2>0$ such that
\begin{equation} \label{equation_form_bound_delta_form}
   \begin{split}
     \| \nbA f \|_{L^2(\mathbb{R}^2; \mathbb{C}^2)}^2 \leq c_1 
\fra_\aa[f] +  c_2 \|f\|_{L^2(\mathbb{R}^2)}^2,\quad f \in 
\mathcal{H}^1_\bA(\mathbb{R}^2).
   \end{split}
\end{equation}

\subsection{Approximation of $\Op$ by Landau Hamiltonians with regular potentials}
\label{ssec:approximation}

Before we proceed further with the spectral analysis of $\Op$, we show that this operator can be regarded
as the limit of a family of Landau Hamiltonians with squeezed regular potentials which are supported in a small
neighborhood of the interaction support $\Sigma$. This justifies $\Op$ as an 
idealized model for Landau Hamiltonians with regular potentials localized in a neighborhood of $\Sigma$.

In order to avoid complicated notation and technical difficulties we discuss the case that the bounded $C^{1,1}$-domain $\Omg_{\rm i}$ is simply connected, so that  
the boundary $\Sg = \p\Omg_{\rm i}$ is given by one regular, closed $C^{1,1}$-curve in $\dR^2$ without 
self-intersections. The more general case can be treated in a similar way.
For  $\varepsilon > 0$ we define 
\begin{equation*} 
  \Sigma_\varepsilon := \big\{ x_\Sigma + t \nu(x_\Sigma): x_\Sigma \in \Sigma, t \in (-\varepsilon, \varepsilon) \big\}.
\end{equation*}
Since $\Sigma$ is a closed and bounded $C^{1,1}$-curve, there exists some $\beta > 0$ such that the mapping
\begin{equation} \label{tubular_coordinates}
  \Sigma \times (-\varepsilon, \varepsilon) \ni (x_\Sigma, t) \mapsto x_\Sigma + t \nu(x_\Sigma) \in \Sigma_\varepsilon
\end{equation}
is bijective for all $\varepsilon \in (0, \beta)$, cf.~\cite[Section 3]{Fu89} and
~\cite[Section 1.2]{KP17}. 
Choose a fixed real $V \in L^\infty(\mathbb{R}^2)$ which is supported in $\Sigma_\beta$
and define the squeezed potentials $V_\varepsilon \in L^\infty(\mathbb{R}^2)$ by
\begin{equation} \label{def_V_eps}
  V_\varepsilon(x) := \begin{cases} \frac{\beta}{\varepsilon} V\big( x_\Sigma + \frac{\beta}{\varepsilon} t \nu(x_\Sigma) \big),
                                    & \text{if } x = x_\Sigma + t \nu(x_\Sigma) \in \Sigma_\varepsilon, \\
                                    0, & \text{if } x \notin \Sigma_\varepsilon. \end{cases}
\end{equation}
Note that the function $V_\varepsilon$ is supported in $\Sigma_\varepsilon$ by definition.
We introduce for $\varepsilon \in (0, \beta)$ in $L^2(\mathbb{R}^2)$ the operator
\begin{equation} \label{def_H_eps}
  \sfH_\varepsilon f := \Opf f + V_\varepsilon f, \qquad \dom \sfH_\varepsilon = \dom \Opf = \mathcal{H}^2_\bA(\mathbb{R}^2),
\end{equation}
which is self-adjoint, since $\Opf$ is self-adjoint and $V_\varepsilon$ is real and bounded.

The following theorem contains the result that $\sfH_\varepsilon$ converges in the norm resolvent sense to $\Op$;
we would like to point out that the interaction strength $\alpha$ of the limit operator is some suitable 
mean value of the potential $V$ along the normal direction, see~\eqref{def_alpha_limit} below.
Our proof uses a method which differs from the one in~\cite{BEHL17, EI01, EK03}.
Since this proof is of more technical nature we postpone it to Appendix~\ref{appendix_approximation}. 

\begin{thm} \label{theorem_approximation}
  Let $V \in L^\infty(\mathbb{R}^2)$ be real and supported in $\Sigma_\beta$, let $\varepsilon \in (0, \beta)$ and
  $V_\varepsilon$ be as in \eqref{def_V_eps}, let $\sfH_\varepsilon$ be given by~\eqref{def_H_eps}, and
  define $\alpha \in L^\infty(\Sigma)$ by
  \begin{equation} \label{def_alpha_limit}
    \alpha(x_\Sigma) := \int_{-\beta}^\beta V(x_\Sigma + t \nu(x_\Sigma)) \dd t, \qquad x_\Sigma \in \Sigma.
  \end{equation}
  Then for $\lambda \in \mathbb{C} \setminus \mathbb{R}$ there exists a constant $c > 0$ (depending on $\lambda$) such that 
  \begin{equation*}
    \big\| (\sfH_\varepsilon - \lambda)^{-1} - (\Op - \lambda)^{-1} \big\| \leq c \sqrt{\varepsilon}.
  \end{equation*}
  In particular, $\sfH_\varepsilon$ converges in the norm resolvent sense to $\Op$ as $\varepsilon \rightarrow 0$.
\end{thm}

In the following corollary we show a converse of Theorem~\ref{theorem_approximation}: given an $\alpha \in L^\infty(\Sigma)$ there is a potential 
$V$ such that the corresponding operators $\sfH_\varepsilon$ converge to $\Op$.

\begin{cor} \label{corollary_approximation}
  Let $\alpha \in L^\infty(\Sigma)$ be real and define almost everywhere in $\mathbb{R}^2$ the function
  \begin{equation*}
    V(x) := \begin{cases} \frac{1}{2 \beta} \alpha( x_\Sigma),
                                    & \text{if } x = x_\Sigma + t \nu(x_\Sigma) \in \Sigma_\beta, \\
                                    0, & \text{if } x \notin \Sigma_\beta, \end{cases}
  \end{equation*}
  and for $\varepsilon \in (0, \beta)$ the scaled
  potentials $V_\varepsilon$ by~\eqref{def_V_eps}.
  Then the operators $\sfH_\varepsilon$ in \eqref{def_H_eps} satisfy 
  \begin{equation*}
    \big\| (\sfH_\varepsilon - \lambda)^{-1} - (\Op - \lambda)^{-1} \big\| \leq c \sqrt{\varepsilon},\qquad \lambda \in \mathbb{C} \setminus \mathbb{R},
  \end{equation*}
  for some constant $c>0$ (depending on $\lambda$). In particular, $\sfH_\varepsilon$ converges in the norm resolvent sense to $\Op$ as $\varepsilon \rightarrow 0$.
\end{cor}

\subsection{Analysis of the resolvent difference of $\Op$ and $\Opf$}
\label{ssec:resdiff}
In this subsection we investigate the resolvent difference 
\begin{equation}\label{resi2}
W_\lambda:=-\gg(\lm) \big( 1 + \aa M(\lm) \big)^{-1}\aa\gg(\ov\lm)^*,\quad  \lambda\in\rho(\Op)\cap\rho(\Opf),
\end{equation}
in~\eqref{resid} 
in more detail. 
First of all we show a useful variant of Krein's resolvent formula for $\Op$ in which
the operator of multiplication with the strength
of interaction $\aa$ is represented as a product $\aa = \aa_2 \aa_1$ of two bounded operators 
$\aa_1$ and $\aa_2$.

\begin{lem} \label{corollary_Krein}
	Let $\aa \in L^\infty(\Sg)$ be real and let $\Op$ be as in Definition~\ref{definition_A_delta_alpha}.
	Let $\cH$ be a Hilbert space and let 
	$\aa_1\colon L^2(\Sg) \arr \cH$
	and $\aa_2\colon \cH \arr L^2(\Sg)$ be bounded operators such that the multiplication operator 
	with $\aa$ fulfils $\aa = \aa_2 \aa_1$. For all $\lm \in \rho(\Op) \cap \rho(\Opf)$ one has $(1 + \aa_1 M(\lm)\aa_2)^{-1}\in\sB(\cH)$ and 
	\begin{equation}\label{reso99}
		(\Op - \lm)^{-1} 
		= 
		(\Opf - \lm)^{-1} 
		- 
		\gg(\lm)\aa_2\big(1 + \aa_1 M(\lm)\aa_2\big)^{-1}
		\aa_1 \gg(\ov\lm)^*.
	\end{equation}
\end{lem}
\begin{proof}
	Consider first $\lm\in (-\infty,\lm_0)$, where $\lm_0<0$ is chosen such that 
	$$\| \aa_1 \| \cdot \| \aa_2 \| \cdot \| M(\lm) \| < 1,\quad \lm\in (-\infty,\lm_0).$$ 
	Note that such $\lm_0$ exists by 
	Proposition~\ref{proposition_decay_M}. 
	Then $(1 + \aa_2 \aa_1 M(\lm))^{-1}\in\sB(L^2(\Sg))$, $(1 + \aa_1 M(\lm)\alpha_2)^{-1}\in\sB(\cH)$, and a direct calculation shows that
	\begin{equation*}
	\begin{aligned}
		&  
		\big(1 + \aa_2 \aa_1 M(\lm)\big)^{-1} 
		\aa_2 - \aa_2 \big(1 + \aa_1 M(\lm) \aa_2\big)^{-1} \\
		&\,\,= 
		\big(1 + \aa_2\aa_1 M(\lm)\big)^{-1} 
		\Big[ 
		\aa_2 \big(1 + \aa_1 M(\lm) \aa_2\big) 
		- 
		\big(1 + \aa_2 \aa_1 M(\lm)\big) \aa_2 \Big] 
		\big(1 + \aa_1 M(\lm) \aa_2\big)^{-1}\\
		&\,\, = 0
	\end{aligned}
	\end{equation*}
	holds for all $\lm\in (-\infty,\lm_0)$.
	Hence, it follows from Theorem~\ref{theorem_Krein} and $\aa = \aa_2 \aa_1$ that
	\begin{equation*}
	\begin{split}
		(\Op - \lm)^{-1}
		&= 
		(\Opf - \lm)^{-1} 
		- 
		\gg(\lm) 
		\big( 1 + \aa_2 \aa_1 M(\lm) \big)^{-1}
      \aa_2 \aa_1 \gg(\ov\lm)^*\\
      &=(\Opf - \lm)^{-1} 
		- 
		\gg(\lm) 
		\aa_2\big( 1 + \aa_1 M(\lm)\aa_2 \big)^{-1}
      \aa_1 \gg(\ov\lm)^*
      \end{split}
	\end{equation*}
	which is \eqref{reso99}. Finally, we note that for arbitrary
	$\lm \in \rho(\Op) \cap \rho(\Opf)$ the formula \eqref{reso99} follows from an analytic continuation argument.
\end{proof}

Next we provide sign properties of the perturbation term $W_\lm$.
  
\begin{lem} \label{corollary_perturbation_term_sign}
	Let $\lm_0 < \min \s(\Op)$.
	If $\aa \in L^\infty(\Sg)$ is such that $\aa(x)\geq 0$ ($\aa(x)\leq 0$) for a.e. $x\in\Sg$ then $W_{\lm_0}$ 
	is a nonpositive (nonnegative, respectively) self-adjoint operator in $L^2(\dR^2)$. 
\end{lem}
\begin{proof}
	Let $\fra_0$ and $\frm$ be the sesquilinear forms corresponding to $\Opf$
	and $\Op$ in~\eqref{def_form_free_Op} and
	in~\eqref{eq:op_delta}, respectively. 
	For a nonnegative function $\aa$ and all $f \in \cH^1_\bA(\dR^2)$ one has
	$\fra_0[f] \leq \frm[f]$ and hence by~\cite[$\S6$ Theorem 2.21]{K95}
	the inequality 
	$$(\Op - \lm_0)^{-1} \leq (\Opf - \lm_0)^{-1}$$ holds for $\lm_0 < \min \s(\Op)$.
	Now \eqref{resid} implies that $W_{\lm_0}$ is nonpositive. The same argument applies for nonpositive $\alpha$.
\end{proof}

Recall that $P_q$ denotes the orthogonal projection onto the infinite dimensional eigen\-space $\ker(\Opf -\Lm_q)$ corresponding to the Landau level $\Lm_q$, $q\in\dN_0$.
Now it will be shown that
for sign-definite functions $\aa$
the compression $P_q W_\lm P_q$ of the perturbation term $W_\lm$ in~\eqref{eq:W} onto $\ker(\Opf-\Lm_q)$ is a
compact operator which has infinite rank. 

\begin{prop}\label{infrank}
	Assume that $\aa\in L^\infty(\Sg)$ and that either $\alpha>0$ a.e. or $\alpha<0$ a.e. on $\Sigma$. 
	Then there exists $\lm_0\in\rho(\Op)\cap\rho(\Opf)\cap(-\infty,0)$ such that the compact operator $P_q W_{\lm_0} P_q$ has infinite rank.
\end{prop}

\begin{proof}
	We discuss the case  $\aa(x) > 0$ for a.e. $x\in\Sg$.
	According to Proposition~\ref{proposition_decay_M} we can choose
	$\lm_0 \in (-\infty,0)$ such that $\Vert\sqrt{\aa} M(\lm_0)\sqrt{\aa}\Vert<1$.
	Using Lemma~\ref{corollary_Krein} we see that $-P_q W_{\lambda_0} P_q$ can be written in the form  
	\begin{equation}\label{pterm}
		-P_q W_{\lambda_0} P_q=
		P_q\gg(\lm_0)\sqrt{\aa}
		\bigl(1+\sqrt{\aa}M(\lm_0)\sqrt{\aa}\bigr)^{-1}
		\sqrt{\aa} \gg(\lm_0)^* P_q
	\end{equation} 
	and $W_{\lm_0}$ is compact in $L^2(\Sg)$ by Theorem~\ref{theorem_Krein}~(iv).
	It remains to show that \eqref{pterm} has infinite rank.
	For this we define
	\[
		C:= \bigl(1+\sqrt{\aa}M(\lm_0)\sqrt{\aa}\bigr)^{-1}
		\and
		D:=\sqrt{\aa}C\sqrt{\aa}.
	\]
	In the present situation $C$ is a nonnegative self-adjoint operator in $L^2(\Sg)$ such that $0\in\rho(C)$ and the operators $D$ and $\sqrt{D}$
	are both nonnegative and self-adjoint in $L^2(\Sg)$.
	We claim that $0\not\in\sp(D)$ and hence also $0\not\in\sp(\sqrt{D})$. In fact, $D\varphi=0$ for some $\varphi\in L^2(\Sg)$ implies
	\begin{equation*}
	\begin{aligned}
		|(C\sqrt{\alpha}\varphi,\psi)_{L^2(\Sg)}|
		& \leq (C\sqrt{\aa}\varphi,\sqrt{\aa}\varphi)_{L^2(\Sg)}
		(C\psi,\psi)_{L^2(\Sg)}\\
		& = 
		(D\varphi,\varphi)_{L^2(\Sg)}
		(C\psi,\psi)_{L^2(\Sg)}=0
	\end{aligned}
	\end{equation*}
	for all $\psi\in L^2(\Sigma)$ and hence $C\sqrt{\alpha}\varphi=0$. As $0\in\rho(C)$ it follows that $\sqrt{\alpha}\varphi=0$ and the assumption
	$\alpha(x) > 0$ for a.e. $x\in\Sigma$ yields $\varphi=0$. Therefore, $0\not\in\sp(D)$ and $0\not\in\sp(\sqrt{D})$.
	In particular,  $\ran\sqrt{D}$ is dense in $L^2(\Sigma)$.
	
	Next we claim that
	\begin{equation}\label{claim}
		\ran 
		\bigl( P_q\gg(\lm_0)\sqrt{D}\bigr)
		\quad\text{is dense in}\quad \ker(\Opf -\Lm_q)\ominus \ker(S-\Lm_q),
	\end{equation}
	and we recall  that the latter space is 
	infinite dimensional by Lemma~\ref{kers} and $\dim\ker(\Opf-\Lm_q)=\infty$. For \eqref{claim}
	assume that $h\in\ker(\Opf-\Lm_q)\ominus \ker(S-\Lm_q)$ satisfies
	\begin{equation*}
		(P_q\gg(\lm_0)\sqrt{D}\varphi,h)_{L^2(\dR^2)}
		=0
		\qquad\text{for all}\quad
		\varphi\in L^2(\Sg).
	\end{equation*}
	Using \eqref{starhaha}  one obtains
	\begin{equation*}
	\begin{split}
		0&=(P_q\gg(\lm_0)\sqrt{D}\varphi,h)_{L^2(\dR^2)}
		=
		(\sqrt{D}\varphi,\gg(\lm_0)^*h)_{L^2(\Sg)}\\
		&=
		\bigl(\sqrt{D}\varphi,\G_1(\Opf -\lm_0)^{-1}h\bigr)_{L^2(\Sg)}
		=
		\frac{1}{\Lm_q-\lm_0}
		(\sqrt{D}\varphi,\G_1h)_{L^2(\Sg)}
	\end{split}
	\end{equation*}
	for all $\varphi\in L^2(\Sg)$. Since $\ran\sqrt{D}$ is dense in $L^2(\Sg)$ this implies
	$\G_1 h=0$. Furthermore, since $h\in\dom \Opf$ also $\G_0h=0$. Therefore $h\in\dom S\cap\ker(\Opf-\Lm_q)$
	and hence $h\in\ker(S-\Lm_q)$. By assumption $h\in\ker(\Opf-\Lm_q)\ominus \ker(S-\Lm_q)$ and thus $h=0$, that is,~\eqref{claim} holds.
	
	Now observe that the operator in~\eqref{pterm} can be written in the form
	\begin{equation}\label{facto}
		-P_q W_{\lm_0} P_q
		=
		P_q\gg(\lm_0)D \gg(\lm_0)^* P_q= RR^*,
	\end{equation}
	where $R=P_q\gg(\lm_0)\sqrt{D}$. Since $\ker RR^*=\ker R^*$ it follows that
	\begin{equation*}
		\ov{\ran RR^*}=\ov{\ran R}
	\end{equation*}
	and $\ov{\ran R}$ is infinite dimensional by~\eqref{claim}. Hence the same is true for  $\ov{\ran RR^*}$
	and also for $\ran R R^*$. Taking into account~\eqref{facto} the assertion follows.
\end{proof}

\section{Estimates and asymptotics for the singular values of 
	\boldmath{$P_q W_\lm P_q$}}
\label{sec:asymptotics}

In this section we continue our study of the resolvent difference \eqref{resid} of the  unperturbed Landau Hamiltonian $\Opf$ and the 
Landau Hamiltonian $\Op$ with a $\delta$-potential
supported on $\Sg$. In the following 
we fix some $\lm_0 < \min\{0,\min\s(\Op)\}$ such that $\|\aa\|_\infty\|M(\lm_0)\| <1$, which is possible due to Proposition~\ref{proposition_decay_M}.
For convenience we use the notation $W := W_{\lm_0}$ for the resolvent difference, that is, 
\begin{equation}\label{resdiff}
	W = (\Op - \lm_0)^{-1} - (\Opf - \lm_0)^{-1}=-\gg(\lm_0) 
			\big( 1 + \aa M(\lm_0) \big)^{-1}\aa\gg(\lm_0)^*;
\end{equation}
cf. \eqref{resi2}. As before we write $W=W_+-W_-$, where $W_+\geq 0$ is the nonnegative part of $W$ and by $W_-\geq0$ is the nonpositive part of $W$; cf. \eqref{www}.
The orthogonal projection on the eigenspace $\ker(\Opf-\Lambda_q)$, $q\in\dN_0$, is denoted by $P_q$.
The goal is to obtain asymptotic estimates and sharp spectral asymptotics for the singular values of 
the operators $P_qW_\pm P_q$ and $P_q\vert W\vert P_q$, under different sign conditions on $\alpha$ and smoothness conditions on $\Sg$. 
This section is split in two subsections dealing with the $C^{1,1}$-case and the $C^\infty$-case, respectively.

\subsection{$C^{1,1}$-smooth $\Sg$}
In this subsection it is assumed that $\Sigma$ is the boundary of a bounded $C^{1,1}$-domain $\Omega_{\rm i}$; cf. Hypothesis~\ref{hypohyp}. 
In the first proposition we consider the compression $P_q \vert W\vert P_q$ of $\vert W\vert$ onto $\ker(\Opf-\Lambda_q)$ and estimate this operator by the 
Toeplitz-type operator in Definition~\ref{def:Tn}. For the lower bound sign-definite functions $\aa$ are required.
\begin{prop}\label{prop:operator_bounds}
	 Let $\aa \in L^\infty(\Sg)$ be real
	 with $\G := \supp\aa$, assume $|\G| > 0$, and let	
	 the resolvent difference $W$
	 be as in~\eqref{resdiff}. 
    Let $T_q^{\G}$ be the self-adjoint Toeplitz-type operator as in Definition~\ref{def:Tn}.
    Then the following holds.
    \begin{myenum}
        \item $P_q \vert W\vert P_q \le c T_q^{\G}$ and $P_q W_\pm P_q \le c_\pm T_q^{\G}$ for some $ c,c_\pm >0$.
        \item 
        If $\aa$ is nonnegative (nonpositive)
        on $\G$ and uniformly positive (uniformly negative, respectively) on
        a closed subset $\G'\subset\G$ such that $|\G'| > 0$ then        
        $P_q |W| P_q  \ge c' T_q^{\G'}$
        for some $c' > 0$.
    \end{myenum}
\end{prop}
\begin{proof}
	We start with a preliminary observation. Let $\chi_{\G_\star}\colon\Sg\arr [0,1]$ be the characteristic function of some closed 
	subset $\G_\star\subset\Sigma$ with $|\G_\star| >0$ and consider the bounded operator
   $D_{\G_\star} := P_q\gg(\lm_0)\chi_{\G_\star}\gg(\lm_0)^*P_q$.
   For $f \in L^2(\dR^2)$ we find
   \begin{equation*}
	 \begin{aligned}
       \big (D_{\G_\star} f, f\big)_{L^2(\dR^2)}
       & =
       \big (
       \chi_{\G_\star}\gg(\lm_0)^* P_q f,\chi_{\G_\star}\gg(\lm_0)^*P_q f
       \big)_{L^2(\Sigma)}\\
       & =
       \frac{\|(P_qf)|_{\G_\star}\|^2_{L^2(\G_\star)}}{(\Lm_q-\lm_0)^2} 
       =
       \frac{\frt_q^{\G_\star}[f,f]}{(\Lm_q-\lm_0)^2},
     \end{aligned}
    \end{equation*}
    where \eqref{starhaha} and $\Gamma_1 f=f\vert_\Sigma$ were used in the second equality. 
	 Hence, $D_{\G_\star}$ and the Toeplitz-type
    operator $T_q^{\G_\star}$ are related via
    \begin{equation}\label{eq:T}
        D_{\G_\star} = \frac{T_q^{\G_\star}}{(\Lm_q - \lm_0)^2}.
    \end{equation}

    (i) We prove the claim for $W_+$. The proof for $W_-$ is analogous and the estimates for $W_+$ and $W_-$ also imply the estimate 
    for $\vert W\vert=W_++W_-$.
	 Consider the mappings
	 \begin{align*}
    	 &\aa_1\colon L^2(\Sigma) \arr L^2(\G),
		 &\aa_1 \phi& := (\aa\phi)|_{\G},\\
		 &\aa_2\colon L^2(\G) \arr L^2(\Sigma),
	    &\aa_2 \psi& := 
	    \begin{cases} \psi & \text{on}\,\,\Gamma,\\ 0 & \text{on}\,\,\Sigma\setminus\Gamma.\end{cases}
    \end{align*}
    It is not difficult to see that
    the product $\aa_2\aa_1$
    coincides with multiplication operator with $\aa$.
    Hence, Krein's formula in
    Lemma~\ref{corollary_Krein} implies
    that the resolvent difference in~\eqref{resdiff}
    can be expressed as
    \begin{equation}\label{santiago}
        W = \gg(\lm_0)C\gg(\lm_0)^*,
    \end{equation}
    where 
    \[
	    C :=
	    -\aa_2\big(1 +\aa_1 M(\lm_0)\aa_2\big)^{-1}\aa_1\in\sB(L^2(\Sg))
    \]
    is self-adjoint (since $W$ in \eqref{santiago} is self-adjoint).
    The nonnegative part $C_+$ of $C$ 
    can be estimated by
    $C_+ \le 	\|C\|$ in the operator sense.
    For the nonnegative part $W_+$  of $W$ 
    we have
    \[
	    W_+ = \gg(\lm_0) C_+  \gg(\lm_0)^*= \gg(\lm_0)\chi_\G C_+ \chi_\G \gg(\lm_0)^*
	 \]	
    and from
    \begin{equation*}
     \big (P_qW_+ P_q f, f\big)_{L^2(\dR^2)}=\big ( C_+ \chi_\G \gg(\lm_0)^* P_q f, \chi_\G \gg(\lm_0)^*P_q f\big)_{L^2(\dR^2)}
     \leq \|C\|(D_\G f, f)_{L^2(\dR^2)}
    \end{equation*} 
    we obtain $P_qW_+ P_q \le
    \|C\| D_\G$. Hence, using~\eqref{eq:T} we find
    \[
        P_qW_+ P_q \le \frac{\|C\|}{(\Lm_q - \lm_0)^2}
        T_q^\G,
    \]    %
    and the estimate for $W_+$ in~(i) follows with 
    $c_+ := \frac{\|C\|}{(\Lm_q - \lm_0)^2}$.

    (ii) We prove the claim for nonnegative $\aa$. Suppose that $\aa$ (as well as $\sqrt{\aa}$) is nonnegative on $\G$ and uniformly positive on $\G'\subset\G$.
  	 Then Krein's formula in Lemma~\ref{corollary_Krein}
    with the mappings
    \begin{align*}
	    &\aa_1\colon L^2(\Sg) \arr L^2(\G),
	    &\aa_1 \phi& := (\sqrt{\aa}\phi)|_{\G},\\
	    &\aa_2\colon L^2(\G) \arr L^2(\Sg),
	    &\aa_2 \phi& := \begin{cases}  \sqrt{\aa}\phi & \text{on}\,\,\Gamma, \\ 0 & \text{on}\,\,\Sigma\setminus\Gamma,\end{cases}
    \end{align*}
    shows
    \[
        W =
        -\gg(\lm_0)\aa_2 \wh{C}
        \aa_1\gg(\lm_0)^*,
    \]
    where the middle-term
    \[
	    \wh{C} :=
	    \big(1 + \aa_1 M(\lm_0)\aa_2\big)^{-1}\in\sB(L^2(\G))
    \]
    is self-adjoint and uniformly positive
    in $L^2(\G)$. 
    Hence, the operator $W$ is nonpositive.
    Thus, we obtain from~\eqref{eq:T} in the same way as in the proof of (i) that
    \begin{equation*}
    \begin{split}
        P_q|W|P_q 
        & \ge
        \big(\inf\s(\wh{C})\big)\cdot P_q\gg(\lm_0)\chi_{\G'}\aa\chi_{\G'}\gg(\lm_0)^*
        P_q\\
        & \ge 
        \big(\inf\s(\wh{C})\big)
        \big(\inf_{x\in\G'}\aa(x)\big)
        \cdot P_q\gg(\lm_0)\chi_{\G'}\gg(\lm_0)^*P_q
        \ge c' T_q^{\G'},
    \end{split}    
    \end{equation*}
    with 
    \[
    	c' = \frac{\inf\s(\wh{C})}{(\Lm_q - \lm_0)^2}\cdot \inf_{x\in\G'}\aa(x) > 0.
    \]
    This proves the inequality in~(ii).
\end{proof}

Now we formulate three corollaries of the above proposition. The first one 
follows from the upper bound on $P_q\vert W\vert P_q$ from Proposition~\ref{prop:operator_bounds}\,(i)
and the spectral estimate for $T_q^{\G}$ in Proposition~\ref{prop:Toeplitz}\,(i).

\begin{cor}\label{thm:C2}
 	Let $\aa \in L^\infty(\Sg)$ be real with $\G= \supp\aa$, assume $|\G| > 0$, and let the resolvent difference $W$
	be as in~\eqref{resdiff}.
   Then the singular values of the operator $P_q \vert W\vert P_q$, $q\in\dN_0$, satisfy
    \[
        \limsup_{k\arr\infty}
        \big(k!\,s_k(P_q \vert W\vert  P_q)\big)^{1/k}
        \le \frac{B}{2}\big({\rm Cap}\,(\G)\big)^2.
    \]
    In particular, the singular values of the operator $P_q W_\pm P_q$, $q\in\dN_0$, satisfy
    \[
        \limsup_{k\arr\infty}
        \big(k!\,s_k(P_q W_\pm  P_q)\big)^{1/k}
        \le \frac{B}{2}\big({\rm Cap}\,(\G)\big)^2.
    \]
\end{cor}
We remark that in the present $C^{1,1}$-setting the lower bound in Proposition~\ref{prop:operator_bounds}\,(ii) in the case of a sign-definite $\alpha$ 
can not be used directly to conclude a lower bound on the singular values for $P_q W P_q$ since the estimate in Proposition~\ref{prop:Toeplitz}\,(i)
is only one-sided. However, the situation is better for the lowest Landau level $\Lambda_0$.
In fact, Proposition~\ref{prop:operator_bounds}\,(ii) and  Proposition~\ref{prop:rank0}
imply the next corollary.

\begin{cor}\label{infrank00}
	Consider the resolvent difference $W$
	in~\eqref{resdiff} and assume
	that $\aa\not\equiv 0$ is either nonnegative or nonpositive.
	Then the rank of $P_0 W P_0$ is infinite.
\end{cor}

\begin{proof}
Assume that $\aa$ is nonnegative and $\aa\not\equiv 0$. Then
there exists $\varepsilon>0$ and $S_\varepsilon\subset\Gamma$ measurable such that $\alpha(x)\geq\varepsilon$ for a.e. $x\in S_\varepsilon$.
Hence there is also a closed subset $K\subset S_\varepsilon$
such that $\vert K\vert>0$ and $\alpha>\varepsilon$ on $K$. Now Proposition~\ref{prop:operator_bounds}\,(ii) and  Proposition~\ref{prop:rank0}
lead to the statement.
\end{proof}

In Proposition~\ref{infrank} it was shown that for positive (or negative) $\alpha\in L^\infty(\Sg)$ the rank of 
$P_q|W|P_q$, $q\in\dN_0$, is infinite. This observation leads to an interesting consequence for Toeplitz-type operators.

\begin{cor}\label{cor:infrank}
	The rank of the self-adjoint Toeplitz-type operator $T_q^{\Sg}$, $q\in\dN_0$, in Definition~\ref{def:Tn} is infinite.
\end{cor}
\begin{proof}
	Consider the self-adjoint operator $\Op=\sfA_{1}$ with $\aa \equiv 1$. Fix $\lm_0 < 0$ such that $\|M(\lm_0)\| < 1$
	and note that the resolvent difference $W$ in~\eqref{resdiff} is nonpositive by Lemma~\ref{corollary_perturbation_term_sign}.
	By Proposition~\ref{infrank}
	the rank of $P_q W P_q=P_qW_-P_q$ is infinite for
	all $q\in\dN_0$. Since $P_qW_-P_q \le c T_q^\Sg$ by Proposition~\ref{prop:operator_bounds}\,(i)
	the rank of $T_q^\Sg$ is infinite
	as well.	 
\end{proof}	

\subsection{$C^\infty$-smooth setting}
Now we pass to the discussion of the $C^\infty$-smooth setting. Here, we are able
to get  more precise results.
In the formulation of the next theorem, 
and also later on,
we denote by $B_\eps(x)\subset\dR^2$ the disc
of radius $\eps >0$ centered at $x\in\dR^2$.
\begin{thm}\label{thm:Cinf}
Let $\aa \in L^\infty(\Sg)$ be real, assume that $\G = \supp\aa$ is a $C^\infty$-smooth arc and that $\aa$ is nonnegative (nonpositive)
    on $\G$ and uniformly
    positive (uniformly negative, respectively) on the
    truncated arc $\G_\eps := \{x\in\G\colon B_\eps(x)\cap \Sg \subset\G\}$ for all $\eps>0$ sufficiently small.
    Let the resolvent difference $W$
    be as in~\eqref{resdiff}.
    Then the singular values of the operator $P_q \vert W\vert P_q$, $q\in\dN_0$,
    satisfy
    \[
	    \lim_{k\arr\infty}
	    \big(k!\, s_k(P_q \vert W\vert P_q)\big)^{1/k} 
	    = 
	    \frac{B}{2}\big({\rm Cap}(\G)\big)^2.
    \]
\end{thm}
\begin{proof}
    By Corollary~\ref{thm:C2} we get
    \begin{equation*}
        \limsup_{k\arr\infty}\big(k!\,
            s_k(P_q \vert W\vert P_q)\big)^{1/k} \le
            \frac{B}{2}\left({\rm Cap}\,(\G)\right)^2
    \end{equation*}
    and for $\eps > 0$ we conclude from Proposition~\ref{prop:operator_bounds}\,(ii)
    and Proposition~\ref{prop:Toeplitz}\,(ii) that
    \begin{equation*}
            \liminf_{k\arr\infty}\big(k!\,
            s_k(P_q \vert W\vert P_q)\big)^{1/k}
            \ge
            \frac{B}{2}\left({\rm Cap}\,(\G_\eps)\right)^2.
    \end{equation*}
    Hence, the claim of the theorem follows from $\liminf_{\eps\arr 0^+}{\rm Cap}\,(\G_\eps) = {\rm Cap}\,(\G)$. 
    In fact, by Proposition~\ref{prop:cap}\,(i)
     we know that ${\rm Cap}\,(\G_\eps) \le {\rm Cap}\,(\G)$ since 
$\G_\eps \subset\G$.
       For the other inequality consider the equilibrium measure $\mu$ 
for $\Gamma$. It is no restriction to assume that $\mu$ has no point 
masses, as otherwise $I(\mu) = \infty$ and hence $\rm{Cap}(\Gamma) = 0$, 
which is a trivial case. First, it follows from the dominated 
convergence theorem that $\mu(\Gamma_\varepsilon) \rightarrow 1$, as 
$\varepsilon \rightarrow 0$. Hence, for $\varepsilon > 0$ the measure 
$\mu_\varepsilon$ acting on Borel sets $\mathcal{M} \subset \mathbb{R}^2$ as
       \begin{equation*}
         \mu_\varepsilon(\mathcal{M}) := 
\frac{1}{\mu(\Gamma_\varepsilon)} \mu(\mathcal{M} \cap \Gamma_\varepsilon)
       \end{equation*}
       is well defined and clearly, $\mu_\varepsilon  \geq 0$, $\supp 
\mu_\varepsilon = \Gamma_\varepsilon$, and 
$\mu_\varepsilon(\Gamma_\varepsilon) = 1$. Another application of the 
dominated convergence theorem yields 
       \begin{equation*}
         I(\mu_\varepsilon) = \frac{1}{\mu(\Gamma_\varepsilon)^2} 
\int_{\Gamma_\varepsilon} \int_{\Gamma_\varepsilon} \ln \frac{1}{|x-y|} 
\dd \mu(x) \dd \mu(y) \rightarrow \int_{\Gamma} \int_{\Gamma} \ln 
\frac{1}{|x-y|} \dd \mu(x) \dd \mu(y) = I(\mu)
       \end{equation*}
       as $\varepsilon \rightarrow 0$, which shows that 
$\liminf_{\varepsilon \rightarrow 0^+} \rm{Cap}(\Gamma_\varepsilon) \geq 
\rm{Cap}(\Gamma)$. 
\end{proof}

Under slightly weaker assumptions on $\alpha$ we conclude the following lower bound on the singular values $P_q \vert W\vert P_q$ 
from Proposition~\ref{prop:operator_bounds}\,(ii) and Proposition~\ref{prop:Toeplitz}\,(ii).

\begin{prop}\label{cor:Cinf}
Let $\aa \in L^\infty(\Sg)$ be real, assume that there exists a $C^\infty$-smooth subarc $\G' \subset \supp\aa$ with two endpoints, $|\G'| > 0$, and that $\aa$ is nonnegative (nonpositive)
    on $\Sg$ and uniformly
    positive (uniformly negative, respectively) on $\G'$.
    Let the resolvent difference $W$
    be as in~\eqref{resdiff}.
		Then the singular values of the operator $P_q \vert W\vert P_q$, $q\in\dN_0$,
		satisfy
		\[
		\liminf_{k\arr\infty}
		\big(k!\, s_k(P_q \vert W\vert P_q)\big)^{1/k} 
		\ge 
		\frac{B}{2}\big({\rm Cap}(\G')\big)^2.
		\]
	\end{prop}

\section{Main results on eigenvalue clustering at Landau levels}\label{sec:clusters}

In this section we prove our main results on the local
spectral properties of the perturbed Landau Hamiltonian
of $\Op$. Throughout this section we fix some $\lambda_0$ such that 
$$
\lm_0 < \min\{0,\min\s(\Op)\}.
$$
We note first that for sign-definite interaction strengths $\alpha$ accumulation of the eigenvalues from one side to each Landau level can be excluded.
This is a direct consequence of well-known perturbation results.

\begin{prop} Assume that $\aa\in L^\infty(\Sg)$ is real. Then the following holds.
		\begin{myenum}
			\item If $\aa$ is nonnegative, then there is no accumulation of eigenvalues of $\Op$ from below to the Landau levels $\Lm_q$, $q\in\dN_0$.
			\item If $\aa$ is nonpositive, then there is no accumulation of eigenvalues of $\Op$ from above to the Landau levels $\Lm_q$, $q\in\dN_0$.
		\end{myenum}
\end{prop}

\begin{proof}
	We prove only (i); the  proof of (ii)
	is analogous. Recall that 
	\begin{equation*}
		(\Op - \lm_0)^{-1}
		-
		(\Opf-\lm_0)^{-1}=
		-
		\gg(\lm_0)
		\bigl(1+\aa M(\lm_0)\bigr)^{-1}\aa\gg(\lm_0)^*\leq 0
	\end{equation*}
	by Lemma~\ref{corollary_perturbation_term_sign} and hence the eigenvalues of $(\Op -\lm_0)^{-1}$ do not accumulate from
	above to the eigenvalues $(\Lm_q - \lm_0)^{-1}$ of $(\Opf-\lm_0)^{-1}$; cf.~\cite[Chapter 9, \S 4, Theorem 7]{BS87}.
	Therefore, the eigenvalues of $\Op$ do not accumulate to $\Lm_q$ from below.  
\end{proof}

If $\alpha$ is either positive or negative on $\Sigma$ one always has accumulation of eigenvalues to each Landau level.

\begin{thm}\label{thm:acc}
	Assume that $\aa\in L^\infty(\Sg)$ is real. Then the following holds.
	\begin{myenum}
		\item If $\aa>0$ a.e. on $\Sigma$, then the eigenvalues of $\Op$ accumulate from above to $\Lm_q$, $q\in\dN_0$.
		\item If $\aa<0$ a.e. on $\Sigma$, then the eigenvalues of $\Op$ accumulate from below to $\Lm_q$, $q\in\dN_0$.
	\end{myenum}
\end{thm}

\begin{proof}
	We prove only (i). Recall that by Lemma~\ref{corollary_perturbation_term_sign}
	the perturbation term in~\eqref{resdiff} is a nonpositive
	operator. It follows from Proposition~\ref{infrank} that
	the rank of $P_qWP_q$ is infinite.
	Hence, Proposition~\ref{prop:RP}
	implies that the eigenvalues of $(\Op -\lm_0)^{-1}$ accumulate from
	below to the eigenvalues
	$(\Lm_q-\lm_0)^{-1}$ of $(\Opf-\lm_0)^{-1}$. Therefore, the eigenvalues of $\Op$ accumulate from above to each Landau level $\Lm_q$.
\end{proof}
For the lowest Landau level $\Lm_0 = B$,
it is not necessary to assume that $\alpha $ is positive or negative on all of $\Sigma$.
The proof of the next theorem is the same as the proof of Theorem~\ref{thm:acc}, but
in order to conclude that the rank of $P_0WP_0$ is infinite one uses Corollary~\ref{infrank00}.

\begin{thm}\label{thm:acc2}
	Assume that $\aa\in L^\infty(\Sg)$ is real
	and $\aa \not\equiv 0$. Then the following holds.
	\begin{myenum}
		\item If $\aa$ is nonnegative, then the eigenvalues of $\Op$ accumulate
		from above to $\Lm_0$. 
		\item If $\aa$ is nonpositive, then the eigenvalues of $\Op$ accumulate from below to $\Lm_0$.
	\end{myenum}
\end{thm}
In order to formulate our main results on the rate of accumulation of the eigenvalues of 
$\Op$ to the Landau levels the following notation is convenient:
\begin{align*}
	&q = 0: 
	&I_0^-& := (-\infty, \Lm_0), 
	&I_0^+& :=(\Lm_0, \Lm_0 + B],\\
	&q \ge 1: 
	&I_q^-& := (\Lm_q - B, \Lm_q), 
	&I_q^+& := (\Lm_q, \Lm_q + B].
\end{align*}	
Note that
\[
	\dR  = 
	\bigcup_{q=0}^\infty I_q^-\,\cup\,
	\bigcup_{q=0}^\infty I_q^+\,\cup\,
	\bigcup_{q=0}^\infty \{\Lm_q\}.
\]

In the first theorem the $C^{1,1}$-smooth case is considered. We obtain regularized
summability of the discrete spectrum of $\Op$ over all
clusters and an asymptotic spectral estimate
within each cluster. We point out that these results  
are true for sign-changing $\aa$.
\begin{thm}\label{thm:C2est}
	Let $\{\lm_k^\pm(q)\}_k$, $q\in\dN_0$, be the 
	eigenvalues of $\Op$ lying in the interval
	$I_q^\pm$, ordered in such a way that the distance from
	$\Lm_q$ is nonincreasing and with multiplicities
	taken into account. 
	Then the following holds.
	\begin{myenum}
		\item 
		$\sum_{q=0}^\infty
		\frac{1}{(2q+1)^{2}}
		\left(
		\sum_k\left|\lm_k^+(q)	- \Lm_q\right|
		+ 	
		\sum_k\left|\lm_k^-(q)	- \Lm_q\right|
		\right) < \infty$. 
		\item $\limsup_{k\arr\infty}\left(k!\,|\lm_k^\pm(q) - \Lm_q|\right)^{1/k} \le \frac{B}{2}\left({\rm Cap}\,(\G)\right)^2$.
	\end{myenum}
\end{thm}
\begin{proof}
	(i) By Theorem~\ref{theorem_Krein}\,(iv) the resolvent difference $W$ in~\eqref{resdiff} belongs to the Schatten-von Neumann class
	$\sS_p(L^2(\dR^2))$ for all $ p > \frac13$
	and, in particular, for $p = 1$.
	Again we use that the spectrum of $D := (\Opf - \lm_0)^{-1}$
	consists of the infinite dimensional eigenvalues $\{(\Lm_q -\lm_0)^{-1}\}_{q\in\dN_0}$. Recall also that $\lm_0 < \min\{0,\min\s(\Op)\}$.
	One verifies that there exists $c_\pm,c_0 > 0$ such that for all $q \in\dN_0$ we have
	\[
	\begin{aligned}
		\mathfrak{d}_k^+(q) := &\, \text{\rm dist}\,\left(
		\frac{1}{\lm_k^+(q) -\lm_0},	\s(D)\right)\\
		 = &\,
		\min\left\{
		\frac{1}{\lm_k^+(q) -\lm_0}- \frac{1}{\Lm_{q+1} - \lm_0},
		\frac{1}{\Lm_q -\lm_0}- \frac{1}{\lm_k^+(q) - \lm_0}\right\}\\
		\ge &\, 
		\frac{c_+(\lm_k^+(q) - \Lm_q)}{\Lm_{q}^2},
	\end{aligned}
	\]	
	and for all $q \in\dN$ 
	\[
	\begin{aligned}
		\mathfrak{d}_k^-(q) 
		 :=& \,
		\text{\rm dist}\,
		\left(\frac{1}{\lm_k^-(q) -\lm_0},\s(D)\right)\\
		 =& \,
		\min\left\{
			\frac{1}{\lm_k^-(q) -\lm_0} - 
			\frac{1}{\Lm_{q} - \lm_0},
			\frac{1}{\Lm_{q-1} - \lm_0} -	\frac{1}{\lm_k^-(q) -\lm_0}
		\right\}\\
		\ge&\, 
		\frac{c_-(\Lm_q - \lm_k^-(q))}{\Lm_{q}^2},
		\end{aligned}
	\]
	and for $q=0$
	\begin{equation*}
		\mathfrak{d}_k^-(0)  :=  
		\text{\rm dist}\,\left(
		\frac{1}{\lm_k^-(0) -\lm_0},
		\s(D)\right)
		 = \frac{\Lm_0-\lm_k^-(0)}{(\Lm_{0} - \lm_0)
		 (\lm_k^-(0) - \lm_0)} 
		\ge \frac{c_0(\Lm_0 - \lm_k^-(0))}{\Lm_0^2}.
	\end{equation*}
	Hence, we get with $C = (\Op - \lm)^{-1}$ 
	\[
	\begin{split}
		\sum_{\lm\in\sd(C)}\!\text{\rm dist}\,(\lm,\s(D))
		&\!	=\!
		\sum_{q=0}^\infty
		\sum_k \left(\mathfrak{d}_k^+(q)
		+ \mathfrak{d}_k^-(q)\right)\\
		&\!\ge
		\sum_{q=0}^\infty
		\frac{c}{B^2(2q+1)^2}
		\sum_k
		\left(
		\left|\lm_k^+(q)	- \Lm_q\right|
		+ 	
		\left|\lm_k^-(q)	- \Lm_q\right|
		\right)
	\end{split}
	\]
	and the claim follows from
	Proposition~\ref{prop:H}.	
	
	(ii) We shall use Proposition~\ref{prop:RPpm}
	with %
	\begin{subequations}\label{eq:setting} 
	\begin{align} 
		& W = W_{\lm_0} \text{~in~\eqref{resdiff}},
		\!& T& = (\Opf -\lm_0)^{-1}, \!&\Lm &= \frac{1}{\Lm_q-\lm_0}, \\
		&P_\Lm = P_q,
		\!&\eps& = \frac12,
		\!&\tau_\pm& = \pm\frac{1}{2}\left[\frac{1}{\Lambda_q\mp B-\lambda_0}-\frac{1}{\Lm_q-\lambda_0}\right].
	\end{align}
	\end{subequations}
	Note that the eigenvalues of $T+W$
	in the interval $(\Lm -2\tau_-,\Lm+2\tau_+)$
	are given by
	\[
		\frac{1}{\lm_1^+(q) - \lm_0} \le 
		\frac{1}{\lm_2^+(q) - \lm_0} \le \dots
		\le
		\Lambda\le
		\dots \le \frac{1}{\lm_2^-(q) - \lm_0}
		\le
		\frac{1}{\lm_1^-(q) - \lm_0}.
	\]
	We conclude from Proposition~\ref{prop:RPpm} that there exists a constant $\ell = \ell(q) \in\dN$ such that
	\[
		\left|\frac{1}{\lm_k^\pm(q) - \lm_0} - \frac{1}{\Lm_q - \lm_0}\right| 
		\le 
		\frac32 s_{k-\ell}(P_q W_\mp P_q)\,,
	\]
	for all $k\in\dN$ large enough. 
	Using Corollary~\ref{thm:C2} we find 
	\[
	\begin{split}
		&\limsup_{k\arr\infty}
		\left(k!\,|\lm_k^\pm(q) - \Lm_q|\right)^{1/k}\\
		& \qquad =
		\limsup_{k\arr\infty}\,
		(\lm_k^\pm(q) - \lm_0)^{1/k}(\Lm_q - \lm_0)^{1/k}
		\left(k!\,\left|\frac{1}{\lm_k^\pm(q) -\lm_0}- \frac{1}{\Lm_q -\lm_0}\right|\right)^{1/k}\\
		& \qquad \leq
		\limsup_{k\arr\infty}\,
		\bigl(k!\,s_{k-\ell}(P_q W_\mp P_q)\bigr)^{1/k},\\
		& \qquad =
		\limsup_{k\arr\infty}\,
		\bigl(k!\,s_{k}(P_q W_\mp P_q)\bigr)^{1/k}\le \frac{B}{2}\bigl( {\rm Cap}\,(\G)\bigr)^2,
	\end{split}	
	\]
	where we have used $\lim_{k\arr\infty} a^{\frac{1}{k}} = 1$ for $a > 0$ and 
	$\limsup_{k\arr\infty}(k!\,\xi_{k \pm \ell})^{1/k}
	=
	\limsup_{k\arr\infty}(k!\,\xi_{k})^{1/k}$ for any nonincreasing nonnegative sequence $\{\xi_k\}_k$; cf. \cite[Section 2.2]{RP07}.
\end{proof}

Now we present our main result on the local spectral
asymptotics for $\Op$ within each cluster; here we rely on Theorem~\ref{thm:Cinf} and hence we have to assume that $\supp\aa$ is $C^\infty$-smooth.

\begin{thm}\label{thm:asymp}
Let $\aa \in L^\infty(\Sg)$ be real, assume that $\G = \supp\aa$ is a $C^\infty$-smooth arc and that $\aa$ is nonnegative (nonpositive)
    on $\G$ and uniformly
    positive (uniformly negative, respectively) on the
    truncated arc $\G_\eps := \{x\in\G\colon B_\eps(x)\cap \Sg \subset\G\}$ for all $\eps>0$ sufficiently small.
	Let $\{\lm_k(q)\}_k$, $q\in\dN_0$, be the 
	eigenvalues of $\Op$ lying in the interval
	$I_q^+$ ($I_q^-$, respectively).
	Then
	\[
		\lim_{k\arr\infty}\left(k!|\lm_k(q) - \Lm_q|\right)^{1/k} 
		= \frac{B}{2}\bigl({\rm Cap}\,(\G)\bigr)^2
	\]
	and, in particular, the eigenvalues of $\Op$
	accumulate to $\Lm_q$ from above (from below, respectively) for all $q\in\dN_0$.
\end{thm}	
\begin{proof}
	We discuss the case $\aa\geq 0$.
	By Theorem~\ref{thm:acc} the eigenvalues of $\Op$ accumulate to $\Lm_q$ from above and there is no accumulation from
	below. It follows from Theorem~\ref{thm:Cinf} that $\rank P_q W P_q = \infty$. Using Proposition~\ref{prop:RP}
	with $W$, $T$, $\Lm$, $P_\Lambda$, $\eps$, and $\tau_\pm$ as in~\eqref{eq:setting} we obtain that there exists a constant $\ell = \ell(q) \in\dN$ such that
	\[
		\frac12 s_{k+\ell}(P_qWP_q) 
		\le 
		\left|\frac{1}{\lm_k(q) - \lm} - 
		\frac{1}{\Lm_q - \lm}\right| 
		\le 
		\frac32 s_{k-\ell}(P_qWP_q)
	\]
	for all $k\in\dN$ sufficiently large.
	These estimates and the asymptotics of the singular values of $P_qWP_q$ in Theorem~\ref{thm:Cinf} yield the claim
	in the same way as in the proof of Theorem~\ref{thm:C2est}\,(ii).
\end{proof}		

Mimicking the proof of the above theorem,
but using Proposition~\ref{cor:Cinf} instead of Theorem~\ref{thm:Cinf}
we get an asymptotic lower bound
within each cluster under relaxed assumptions on $\aa$  and $\G$.
\begin{prop} \label{proposition_lower_bound}
Let $\aa \in L^\infty(\Sg)$ be real, assume that there exists a $C^\infty$-smooth subarc $\G' \subset \supp\aa$ with two endpoints, $|\G'| > 0$, and that $\aa$ is nonnegative (nonpositive)
    on $\Sg$ and uniformly
    positive (uniformly negative, respectively) on $\G'$.
	Let $\{\lm_k(q)\}_k$, $q\in\dN_0$, be the 
	eigenvalues of $\Op$ lying in the interval
	$I_q^+$ ($I_q^-$, respectively).
	Then
	\[
	\lim_{k\arr\infty}\left(k!|\lm_k(q) - \Lm_q|\right)^{1/k} 
	\ge \frac{B}{2}\left({\rm Cap}\,(\G')\right)^2
	\]
	and, in particular, the eigenvalues of $\Op$
	accumulate to $\Lm_q$
	from above (from below, respectively)
	for all $q\in\dN_0$.
\end{prop}

The above proposition applies to several 
additional cases of interest. E.g., $\aa$ can be a nonnegative or nonpositive function which is continuous (and does not vanish identically), or
$\supp\aa$ may consist of finitely many disjoint arcs. In both situations one can choose a $C^\infty$-smooth subarc 
$\G' \subset \supp\aa$ with two endpoints, such that $|\G'| > 0$ and $\aa$ uniformly
    positive (or uniformly negative) on $\G'$.
Moreover, Proposition~\ref{proposition_lower_bound} can also be applied if
the support of $\aa$ is not $C^\infty$-smooth itself but contains a
$C^\infty$-smooth subarc with two endpoints
on which $\aa$ is uniformly positive (or uniformly negative).

\appendix

\section{Quasi boundary triples and their Weyl functions}\label{app:A}

In this appendix we provide a brief introduction to the abstract notion of quasi boundary triples and their Weyl functions from extension theory of
symmetric operators. For more details and complete proofs we refer the reader to \cite{BL07,BL12}.

In the following let $\cH$ be a Hilbert space and assume that $S$ is a densely defined closed symmetric operator in $\cH$. 

\begin{dfn}\label{qbtdef}
 Assume that $T$ is a linear operator in $\cH$ such that $\overline T=S^*$.
 A triple $\{\mathcal G,\Gamma_0,\Gamma_1\}$ is called a {\em quasi boundary triple} for $T\subset S^*$ if $\mathcal G$ is a Hilbert space and 
 $\Gamma_0,\Gamma_1:\dom T \rightarrow\mathcal G$ are linear mappings such that the following holds:
 \begin{itemize}
  \item [{\rm (i)}] The abstract Green identity
  \begin{equation*}
  (Tf,g)_{\cH}-(f,Tg)_{\cH}=(\Gamma_1 f,\Gamma_0 g)_{\cG}-(\Gamma_0 f,\Gamma_1 g)_{\cG}
 \end{equation*}
 is valid for all $f,g\in\dom T$.
 \item [{\rm (ii)}] The map $\Gamma=(\Gamma_0,\Gamma_1)^\top:\dom T\rightarrow\cG\times\cG$ has dense range.
 \item [{\rm (iii)}] The operator $A_0:=T\upharpoonright\ker\Gamma_0$ is self-adjoint in $\cH$.
 \end{itemize}
\end{dfn}

We recall that a quasi boundary triple $\{\mathcal G,\Gamma_0,\Gamma_1\}$ for $T\subset S^*$ exists
if and only if the deficiency indices $n_\pm(S)=\dim \ker(S^*\mp i)$ coincide, in which case
one has $\dim\cG=n_\pm(S)$. We also note that for a quasi boundary triple $\{\mathcal G,\Gamma_0,\Gamma_1\}$  for $T\subset S^*$ one automatically
has
$$\dom S=\ker\Gamma_0\cap\ker\Gamma_1$$
and that the extension $A_1:=T\upharpoonright\ker\Gamma_1$ of $S$ is symmetric in $\cH$ but in general not closed or self-adjoint. Furthermore, 
if $\dim\cG=n_\pm(S)$ is finite then $T$ and $S^*$ coincide, the abstract Green identity in Definition~\ref{qbtdef}~(i) holds for all $f,g\in\dom S^*$ and
the map $\Gamma=(\Gamma_0,\Gamma_1)^\top:\dom S^*\rightarrow\cG\times\cG$ in Definition~\ref{qbtdef}~(i) is surjective. A triple $\{\mathcal G,\Gamma_0,\Gamma_1\}$
with these two properties is an {\em ordinary boundary triple} in the sense of \cite{BGP08,DM91,GG91,S12}. Also recall the notion of {\it generalized boundary triples}:
If $\overline T=S^*$ and a triple $\{\mathcal G,\Gamma_0,\Gamma_1\}$ with linear mappings $\Gamma_0,\Gamma_1:\dom T \rightarrow\mathcal G$ satisfies (i) and (iii) 
in Definition~\ref{qbtdef} and instead of (ii) the stronger condition $\ran\Gamma_0=\cG$ then  $\{\mathcal G,\Gamma_0,\Gamma_1\}$ is said to be a 
{\it generalized boundary triple}; cf. \cite[Definition 6.1 and Lemma 6.1~(3)]{DM95}.

When determining a quasi boundary triple it is often nontrivial to prove that the operator $T$ satisfies $\overline T=S^*$. The following theorem from \cite[Theorem 2.3]{BL07}
offers a way to circumvent this problem. Theorem~\ref{ratebitte} is applied in proof of Theorem~\ref{theorem_triple}.

\begin{thm}\label{ratebitte}
Let $\cH$ and $\cG$ be Hilbert spaces, let
$T$ be a linear operator in $\cH$ and assume that there are linear mappings 
$\Gamma_0,\Gamma_1:\dom T\rightarrow\cG$ such that the following holds:
    \begin{itemize}
    \item[\rm (i)] For all $f,g\in\dom T$ one has
      \begin{equation*}
  (Tf,g)_{\cH}-(f,Tg)_{\cH}=(\Gamma_1 f,\Gamma_0 g)_{\cG}-(\Gamma_0 f,\Gamma_1 g)_{\cG}.
      \end{equation*}
    \item[\rm(ii)] The kernel and range of $\Gamma = ( \Gamma_0, \Gamma_1)^\top: \dom T \rightarrow \cG \times \cG$
      are dense in $\cH$ and $\mathcal G\times\mathcal G$, respectively.
    \item[\rm(iii)] The operator $T\upharpoonright\ker \Gamma_0$ contains a self-adjoint operator $A_0$.
  \end{itemize}
 Then 
 \begin{equation*}
  S:=T\upharpoonright\bigl(\ker\Gamma_0\cap\ker\Gamma_1\bigr)
 \end{equation*}
is a densely defined closed symmetric operator in $\cH$ and $\overline T=S^*$. Moreover, 
$\{\mathcal G,\Gamma_0,\Gamma_1\}$
is a quasi boundary triple for $T\subset S^*$ such that
$A_0=T\upharpoonright\ker\Gamma_0$.
\end{thm}

Next the $\gamma$-field and Weyl function corresponding to a quasi boundary triple will be introduced; formally the definitions are the same as for ordinary and
generalized boundary triples, see \cite{DM91,DM95}. In the following let $\{\mathcal G,\Gamma_0,\Gamma_1\}$
be a quasi boundary triple for $T\subset S^*$ and consider the self-adjoint operator $A_0=T\upharpoonright\ker\Gamma_0$. It is not difficult to verify that for all 
$\lambda\in\rho(A_0)$ the following direct sum decomposition of $\dom T$ is valid:
\begin{equation*}
 \dom T=\dom A_0\,\dot+\,\ker(T-\lambda)=\ker \Gamma_0\,\dot+\,\ker(T-\lambda),\qquad \lambda\in\rho(A_0).
\end{equation*}
Therefore the restriction $\Gamma_0\upharpoonright\ker(T-\lambda)$ is invertible for all $\lambda\in\rho(A_0)$ and we define the
$\gamma$-field corresponding to $\{\mathcal G,\Gamma_0,\Gamma_1\}$ as the operator function 
\begin{equation*}
 \lambda\mapsto \gamma(\lambda):=\bigl(\Gamma_0\upharpoonright\ker(T-\lambda)\bigr)^{-1}
\end{equation*}
defined on $\rho(A_0)$. It is clear that the values $\gamma(\lambda)$ of the $\gamma$-field are densely defined linear operators from $\cG$ into $\cH$
with $\dom\gamma(\lambda)=\ran\Gamma_0$ and $\ran\gamma(\lambda)=\ker(T-\lambda)$. It can be shown that $\gamma(\lambda)$ is a bounded operator for all
$\lambda\in\rho(A_0)$ and hence admits a closure $\overline{\gamma(\lambda)}\in\sB(\cG,\cH)$. The function $\lambda\mapsto\overline{\gamma(\lambda)}\in\sB(\cG,\cH)$ is holomorphic
on $\rho(A_0)$. For the adjoint operators one verifies as a consequence of
the abstract Green identity the relation
\begin{equation}\label{starhaha}
 \gamma(\lambda)^*=\Gamma_1(A_0-\overline{\lambda})^{-1}\in\sB(\cH,\cG),\qquad \lambda\in\rho(A_0).
\end{equation}
For more properties and detailed proofs we refer the reader to \cite[Proposition 2.6]{BL07} and \cite[Proposition 6.13]{BL12}. 
An important analytic object associated with the  quasi boundary triple $\{\mathcal G,\Gamma_0,\Gamma_1\}$ is the Weyl function $M$. It is defined on $\rho(A_0)$ by
\begin{equation*}
 \lambda\mapsto M(\lambda)=\Gamma_1 \bigl(\Gamma_0\upharpoonright\ker(T-\lambda)\bigr)^{-1},
\end{equation*}
and it is clear from the definition that $M(\lambda)$, $\lambda\in\rho(A_0)$, is a densely defined linear operator in $\cG$ with $\dom M(\lambda)=\ran\Gamma_0$
and $\ran M(\lambda)\subset\ran\Gamma_1$. In contrast to ordinary and generalized boundary triples the values $M(\lambda)$ of the 
Weyl function can be unbounded and non-closed operators in $\cG$. However, one has the relation
\begin{equation*}
 M(\overline\lambda)\subset M(\lambda)^*,\qquad\lambda\in\rho(A_0),
\end{equation*}
and hence $M(\lambda)$ is a closable operator in $\cG$. Furthermore, the Weyl function and $\gamma$-field are connected via
\begin{equation*}
 M(\lambda)-M(\mu)^*=(\lambda-\overline\mu)\gamma(\mu)^*\gamma(\lambda),\qquad\lambda,\mu\in\rho(A_0);
\end{equation*}
cf. \cite[Proposition 2.6]{BL07} and \cite[Proposition 6.14]{BL12} for more details. For the present paper the special case that $\ran\Gamma_0=\cG$
holds is of particular interest. In this situation one has $\dom\gamma(\lambda)=\dom M(\lambda)=\cG$ and it follows, in particular, 
that the values $M(\lambda)$ of the Weyl function are bounded operators in $\cG$.

In the following our interest will be in restrictions of $T$ defined by
\begin{equation}\label{abchen}
  A_{[B]}f=Tf, \qquad
  \dom A_{[B]}=\bigl\{f\in\dom T:\Gamma_0 f= B\Gamma_1 f\bigr\},
\end{equation}
where $B$ is a linear operator in $\cG$. If $B$ is not defined on the whole space $\cG$ the boundary condition in \eqref{abchen} is understood for only those $f\in\dom T$
such that $\Gamma_1 f\in\dom B$. Typically the interest is to conclude from qualitative properties of $B$ qualitative properties of $A_{[B]}$. In the present
situation we will focus on self-adjointness. Suppose first that $B$ is a symmetric operator in $\cG$. Then it follows together with the abstract Green identity
in Definition~\ref{qbtdef}~(i) that for $f,g\in\dom A_{[B]}$ we have
\begin{equation*}
\begin{split}
 (A_{[B]}f,g)_\cH-(f,A_{[B]}g)_\cH&=(Tf,g)_\cH-(f,Tg)_\cH\\
 &=(\Gamma_1 f,\Gamma_0 g)_\cG-(\Gamma_0 f,\Gamma_1 g)_\cG\\
 &=(\Gamma_1 f,B\Gamma_1 g)_\cG-(B\Gamma_1 f,\Gamma_1 g)_\cG\\
 &=0
\end{split}
 \end{equation*}
and therefore the operator $A_{[B]}$ is symmetric in $\cH$. However, self-adjointness of $B$ in $\cG$ does not automatically imply that $A_{[B]}$ is
self-adjoint in $\cH$. In fact, this conclusion does not even hold for bounded self-adjoint operators $B$ and hence one has to impose additional
conditions. Such conditions may involve mapping properties of the Weyl function, the parameter $B$, or the boundary mappings $\Gamma_0$
and $\Gamma_1$. In this context we recall \cite[Corollary~4.4]{BLLR18} and a special case of it below. For more general boundary conditions we refer the reader to
\cite{BLLR18}.

\begin{thm}\label{gutescor}
Let $\{\cG,\Gamma_0,\Gamma_1\}$ be a quasi boundary triple for $T\subset S^*$
with corresponding $\gamma$-field $\gamma$ and Weyl function $M$.
Let $B\in\sB(\cG)$ be a self-adjoint operator and assume that for
some $\lambda_0\in\rho(A_0)\cap\dR$ the following conditions hold:
\begin{itemize}
 \item [{\rm (i)}]
 $1\in\rho(B\overline{M(\lambda_0)})$;
 \item [{\rm (ii)}]
 $B(\ran\overline{M(\lambda_0)})\subset\ran\Gamma_0$;
 \item [{\rm (iii)}]
 $B(\ran\Gamma_1)\subset\ran\Gamma_0\quad$ or $\quad\lambda_0\in\rho(A_1)$.
\end{itemize}
Then the operator $A_{[B]}$ in \eqref{abchen}
is a self-adjoint extension of $S$ in $\cH$ such that $\lambda_0\in\rho(A_{[B]})$. Furthermore, $\lambda \in \rho(A_0)$ is an 
eigenvalue of $A_{[B]}$ if and only if $1 \in \sigma_{\rm p}(BM(\lambda))$, 
for all $\lambda\in\rho(A_{[B]})\cap\rho(A_0)$ one has $(1-BM(\lambda))^{-1}\in\sB(\cG)$ and 
\begin{equation*}
  (A_{[B]}-\lambda)^{-1}=(A_0-\lambda)^{-1}
  + \gamma(\lambda)\bigl(1-BM(\lambda)\bigr)^{-1}B\gamma(\overline\lambda)^*.
\end{equation*}
\end{thm}

For our purposes it is convenient to state the following special case of Theorem~\ref{gutescor}, where the quasi boundary triple is
even a generalized boundary triple, that is, we require $\ran\Gamma_0=\cG$. In this situation it is clear that (ii) and (iii) in
Theorem~\ref{gutescor} hold and $\overline{M(\lambda_0)}=M(\lambda_0)\in\sB(\cG)$.

\begin{cor}\label{besserescor}
Let $\{\cG,\Gamma_0,\Gamma_1\}$ be a quasi boundary triple for $T\subset S^*$
with corresponding $\gamma$-field $\gamma$ and Weyl function $M$, and assume, in addition, that $\ran\Gamma_0=\cG$.
Let $B\in\sB(\cG)$ be a self-adjoint operator and assume that $1\in\rho(BM(\lambda_0))$ for
some $\lambda_0\in\rho(A_0)\cap\dR$. 
Then the operator $A_{[B]}$ in \eqref{abchen}
is a self-adjoint extension of $S$ in $\cH$ such that $\lambda_0\in\rho(A_{[B]})$. Furthermore, $\lambda \in \rho(A_0)$ is an 
eigenvalue of $A_{[B]}$ if and only if $1 \in \sigma_{\rm p}(BM(\lambda))$, for all $\lambda\in\rho(A_{[B]})\cap\rho(A_0)$ one has $(1-BM(\lambda))^{-1}\in\sB(\cG)$ and 
\begin{equation*}
  (A_{[B]}-\lambda)^{-1}=(A_0-\lambda)^{-1}
  + \gamma(\lambda)\bigl(1-BM(\lambda)\bigr)^{-1}B\gamma(\overline\lambda)^*.
\end{equation*} 
\end{cor}

A typical way to satisfy the condition $1\in\rho(BM(\lambda_0))$ in Corollary~\ref{besserescor} (or Theorem~\ref{gutescor})
is to prove that $\Vert M(\lambda_0)\Vert\rightarrow 0$ for $\lambda_0\rightarrow -\infty$ if $A_0$ is bounded from below.
The next result contains a useful sufficient condition for the decay of the Weyl function along the negative half-line. 
Theorem~\ref{wowsuperthm} is a special case of \cite[Theorem 6.1]{BLLR18}, where in a more general setting the decay of the Weyl functions 
in different sectors of $\dC$ is discussed.

\begin{thm}\label{wowsuperthm}
Let $\{\cG,\Gamma_0,\Gamma_1\}$ be a quasi boundary triple for $T\subset S^*$
with corresponding Weyl function $M$,
assume that $\ran\Gamma_0=\cG$, that $A_0$ is bounded from below and that
\begin{equation*}
  \Gamma_1|A_0 -\mu|^{-\beta} : \cH \supset \dom(\Gamma_1|A_0
  -\mu|^{-\beta}) \to \cG
\end{equation*}
is bounded for some $\mu\in\rho(A_0)$ and some $\beta\in\bigl(0,\frac{1}{2}\bigr]$.
Then  for all $w_0<\min\sigma(A_0)$
there exists $D > 0$ such that
\begin{equation*}
   \|M(\lambda)\|
   \leq \frac{D}{\bigl(\text{\rm dist}(\lambda,\sigma(A_0))\bigr)^{1-2\beta}}
\end{equation*}
holds for all $\lambda<w_0$.\end{thm}

\section{Proof of Theorem~\ref{theorem_approximation}} 
\label{appendix_approximation}

In order to prove Theorem~\ref{theorem_approximation}, we show that the quadratic forms corresponding to 
$\sfH_\eps$ and $\Op$ are close to each other in a suitable sense. 
We fix a sufficiently small $\beta > 0$ such that the map in~\eqref{tubular_coordinates}
is bijective.
Let $\fra_\aa$ be the quadratic form associated
to $\Op$ introduced in~\eqref{eq:form222} and define for $\eps \in (0, \beta)$
\begin{equation} \label{def_form_approximation}
	\frh_\eps[f, g] 
	:= 
	\big( \nbA f, \nbA g \big)_{L^2(\dR^2; \dC^2)}
	+ 
	( V_\eps f, g \big)_{L^2(\dR^2)},\qquad 
	\dom \frh_\eps := \cH^1_\bA(\dR^2).
\end{equation}
It is not difficult to see that $\frh_\eps$ is a densely defined, closed, symmetric, and 
semi-bounded form which is associated to $\sfH_\eps$.
In the first lemma we show that the forms $\frh_\eps$ are uniformly bounded from below.
\begin{lem} \label{lemma_uniformly_bounded}
	Let  $\eps \in (0, \bb)$ and consider the form $\frh_\eps$ in~\eqref{def_form_approximation}.
	Then there exists a constant $\lm_1 \in \dR$ such that $\frh_\eps \geq \lm_1$
	for all $\eps \in (0, \bb)$. In particular, $(-\infty, \lm_1) \subset \rho(\sfH_\eps)$
	for all $\eps \in (0,\bb)$.
\end{lem}
\begin{proof}
	It follows from~\cite[Proposition~3.1]{BEHL17}\footnote{\label{note1}Note that this result is formulated in \cite{BEHL17} only for $C^2$-hypersurfaces but remains valid in the slightly less regular situation considered here. In fact, the key ingredient in the proof of \cite[Proposition~3.1]{BEHL17} that needs to be ensured for a regular, closed $C^{1,1}$-curve in $\dR^2$ is \cite[Hypothesis~2.3~(c)]{BEHL17}, which follows from \cite[Theorem~5.1 and Theorem~5.7]{DZ94}.} 
	that there exists $\lm_1 \in \dR$ such that
	\begin{equation*}
		\big( \nb |f|, \nb |f| \big)_{L^2(\dR^2; \dC^2)} + 
		( V_\eps f, f \big)_{L^2(\dR^2)}
		\geq 
		\lm_1 \| f \|^2_{L^2(\dR^2)}
	\end{equation*}
	holds for all $f \in C_0^\infty(\dR^2)$. 
	Combining this with the diamagnetic 
	inequality~\eqref{diamagnetic_inequality_gradient} one concludes that 
	$\frh_\eps[f] \geq \lm_1 \| f \|^2_{L^2(\dR^2)}$
	for all $f \in C_0^\infty(\dR^2)$. Now the result follows from the fact that $C_0^\infty(\dR^2)$
	is dense in $\cH^1_\bA(\dR^2)$.
\end{proof}

Next, we verify that the form $\fra_0$ corresponding to Landau Hamiltonian $\Opf$
is relatively bounded with respect to the form $\frh_\eps$
with constants which are independent of $\eps$.

\begin{lem} \label{lemma_V_eps_relatively_bounded}
	Let $V \in L^\infty(\dR^2)$ be real and supported in $\Sg_\bb$, let $\eps \in (0, \bb)$,
	define the function $V_\eps$ as in~\eqref{def_V_eps}, and let the quadratic form 
	$\mathfrak{h}_\varepsilon$ be as in~\eqref{def_form_approximation}. Then there exist constants $c_1, c_2 > 0$ independent of $\varepsilon$ such that 
	\begin{equation} \label{equation_bound2}
		\| \nbA f \|_{L^2(\dR^2;\dC^2)}^2
		\leq 
		c_1 \mathfrak{h}_\varepsilon[f]
		+ 
		c_2 \| f \|_{L^2(\dR^2)}^2
	\end{equation}
	holds for all $f \in \mathcal{H}^1_\bA(\mathbb{R}^2)$.
\end{lem}
\begin{proof}
	Fix $\delta > 0$ and let $v_\varepsilon := \sqrt{|V_\varepsilon|}$.
	Using the diamagnetic inequality~\eqref{diamagnetic_inequality_resolvent} and a similar estimate as in \cite[Proposition~3.1~(ii)]{BEHL17}\footnoteref{note1} we deduce that there is a $\lm_0 < 0$ depending on $\delta$
	such that for all $\lm \leq \lm_0$ and all $g \in L^2(\dR^2)$
	\begin{equation*}
	  \begin{split}
		\| (\Opf - \lambda)^{-1/2} v_\varepsilon  g \|_{L^2(\dR^2)}^2
		&= \big( v_\varepsilon (\Opf - \lambda)^{-1} v_\varepsilon g, g \big)_{L^2(\mathbb{R}^2)} \\
		&\leq \big\| v_\varepsilon (\Opf - \lambda)^{-1} v_\varepsilon g \big\|_{L^2(\mathbb{R}^2)} \cdot \|g\|_{L^2(\mathbb{R}^2)}\\
		&\leq 
		\big\| v_\varepsilon (-\Delta - \lambda)^{-1} v_\varepsilon |g| \big\|_{L^2(\mathbb{R}^2)} \cdot \|g\|_{L^2(\mathbb{R}^2)}
		\leq 
		\dl \| g \|^2_{L^2(\dR^2)}
      \end{split}
	\end{equation*}
	is true. By taking adjoint we get that also $\| v_\varepsilon (\Opf - \lambda)^{-1/2}  g \|_{L^2(\dR^2)}^2 \leq \delta \|g\|_{L^2(\mathbb{R}^2)}^2$ 
	for all $g \in L^2(\mathbb{R}^2)$. This implies for $f \in \mathcal{H}^2_\bA(\mathbb{R}^2)$
	\begin{equation} \label{equation_relative_bounded}
	\begin{split}
		\big|(V_\varepsilon f, f)_{L^2(\mathbb{R}^2)}\big| &\leq 
		\| v_\eps(\Opf - \lm)^{-1/2}(\Opf - \lm)^{1/2} f\|_{L^2(\dR^2)}^2 \\
		&\leq
		\dl\|(\Opf - \lm)^{1/2} f\|_{L^2(\dR^2)}^2 
		= \delta \big( (\Opf - \lambda) f, f \big)_{L^2(\mathbb{R}^2)} \\
		&=
		\dl \|\nbA f\|_{L^2(\dR^2;\dC^2)}^2 -
		\dl \lm\|f\|^2_{L^2(\dR^2)}
	\end{split}
	\end{equation}
	and since $\mathcal{H}^2_\bA(\mathbb{R}^2)$ is dense in $\mathcal{H}^1_\bA(\mathbb{R}^2)$ this estimate extends to $f \in \mathcal{H}^1_\bA(\mathbb{R}^2)$.
	Eventually, 
	from~\eqref{equation_relative_bounded} we conclude
	\begin{equation*}
	\begin{split}
		\| \nbA f \|_{L^2(\dR^2;\dC^2)}^2 
		&=
		\mathfrak{h}_\varepsilon[f] -
		(V_\eps f, f )_{L^2(\dR^2)} \\
		&\leq 
		\mathfrak{h}_\varepsilon[f] +
		\delta \| \nbA f \|_{L^2(\dR^2;\dC^2)}^2  - \dl \lambda \| f \|_{L^2(\mathbb{R}^2)}^2.
	\end{split}
	\end{equation*}
	Choosing $\delta \in (0, 1)$ this implies the claim \eqref{equation_bound2}.
\end{proof}

Let us denote by $\kp = \dot \gg_2 \ddot \gg_1 - \dot \gg_1 \ddot \gg_2$ the signed curvature of $\Sg$, 
where $\gg = (\gg_1, \gg_2) \colon I \arr \dR^2$ is any natural parametrization of $\Sg$ ($|\dot\gg| = 1$).
In the following we will often make use of the transformation to tubular coordinates, 
which yields for $h \in L^1(\Sg_\eps)$ (see e.g. \cite[Proposition~2.6]{BEHL17} or \cite{EI01})
\begin{equation} \label{transformation_tubular_coords}
	\int_{\Sg_\eps} h(x) \dd x 
	= 
	\int_\Sg \int_{-\eps}^\eps 
	h\big(x_\Sg + t \nu(x_\Sg)\big) 
	\big(1 - t \kp(x_\Sg) \big)\dd t \dd \s(x_\Sg).
\end{equation}

In the next lemma we show a variant of the trace theorem which will be very useful for the proof of Theorem~\ref{theorem_approximation}.
For the sake of brevity, we use the following notation
\[
	\map(x_\Sg,s) := x_\Sg + s\nu(x_\Sg)\and
	\cJ(x_\Sg, s) := 1 - s\kp(x_\Sg).
\]
\begin{lem} \label{lemma_trace_theorem}
	Let $\Sg$ be the boundary of the simply connected $C^{1,1}$-domain $\Omega_{\rm i}$ and let $\bb > 0$ be such that the mapping in~\eqref{tubular_coordinates}
	is bijective. 
	Then there exists a constant $C > 0$
	independent of $s\in (-\bb, \bb)$ such that 
	\begin{equation*}
		\int_\Sg 
		\big|f(\map(x_\Sg,s))\big|^2 \dd \s(x_\Sg)
		\leq 
		C \|  f \|_{\mathcal{H}^1_\bA(\dR^2)}^2
	\end{equation*}
	holds for all $f \in \cH_\bA^1(\dR^2)$.
\end{lem}
\begin{proof}
	Throughout the proof $c > 0$ denotes a generic
	positive constant, which varies from line to line.
	It suffices to show the claim for functions in the dense subspace $C_0^\infty(\mathbb{R}^2)$ of $\mathcal{H}_\bA^1(\mathbb{R}^2)$.
      For $f \in C_0^\infty(\mathbb{R}^2)$ 
	the main theorem of calculus, the chain rule, and $\frac{\dd }{\dd t} \map(x_\Sigma, st) = s \nu(x_\Sigma)$ yield
	\begin{equation} \label{eq:diff}
	\begin{split}
		\Big| \big|f(\map(x_\Sg,s))\big|^2 
		&- \big|f(\map(x_\Sg, 0))\big|^2 \Big|	= 
		\left|\int_0^1 \frac{\dd}{\dd t} (|f|^2)(	\map(x_\Sg,st)) \dd t \right| \\
		&\leq
		\int_0^1 \left|\left\langle \nb (|f|^2)(\map(x_\Sg,st)), 
		s \nu(x_\Sg) \right\rangle \right| \dd t \\
		&\leq
		2 |s| \int_0^1 \big||f| \cdot \nb (|f|)(\map(x_\Sg,st)) \big| \dd t \\
		&\leq 
		|s| \int_0^1 
		\left[ 
		\big|\nb (|f|) (\map(x_\Sg,st))\big|^2 + 
		\big|f(\map(x_\Sg,st))\big|^2 
		\right] \dd t \\
		&\leq 
		\int_{0}^\bb 
		\left[ \big|\nb (|f|) (\map(x_\Sg,r))\big|^2 + 
		\big|f(\map(x_\Sg,r)) \big|^2 \right] \dd r,
	\end{split}
	\end{equation}
	where the substitution $r = s t$ was employed in the last step.
	Next, by applying Corollary~\ref{corollary_Sobolev_spaces} we obtain
	\begin{equation}\label{eq:I1}
		I_1 := \int_\Sg \big|f(x_\Sg)\big|^2 \dd \s(x_\Sg) \le 
		c \big( 
		\| \nbA f \|_{L^2(\dR^2; \mathbb{C}^2)}^2 + 
		\| f \|_{L^2(\dR^2)}^2 \big).
	\end{equation}
	Using that there is some $c>0$ such that $1 \leq c\cJ(x_\Sg,r)$ 
	for all sufficiently small $r \leq \beta$, formula~\eqref{transformation_tubular_coords}, the diamagnetic inequality
	\eqref{diamagnetic_inequality_gradient},
	and~\eqref{eq:diff} we get 
	\begin{equation} \label{eq:I2}
	\begin{split}
		I_2 & := 
		\left|\int_\Sg \left(
		\big|f(\map(x_\Sg,s))\big|^2 
		- \big|f(\map(x_\Sg, 0)) \big|^2
		\right) \dd \s(x_\Sg)\right|\\
		&\leq 
		 \int_\Sg 
		\int_0^{\bb} 
		\left[
			\left|\nb (|f|) (\map(x_\Sg,r))\right|^2 
			+ 
			\left|f(\map(x_\Sg,r))\right|^2 
		\right]  \dd r \dd \s(x_\Sg) \\
		&
		\leq 
		c \int_\Sg \int_{0}^\bb 
		\left[
			\left|\nb (|f|) (\map(x_\Sg,r))\right|^2 
			+ 
			\left|f(\map(x_\Sg,r))\right|^2
		\right]\! 
		\cJ(x_\Sg,r)  \dd r \dd \s(x_\Sg)\\
		&
		\leq 
		c \big( 
		\| \nbA f \|_{L^2(\dR^2; \mathbb{C}^2)}^2 + 
		\| f \|_{L^2(\dR^2)}^2 \big).
	\end{split}
	\end{equation}
	Combining~\eqref{eq:I1} and~\eqref{eq:I2}
	we arrive at 
	\begin{equation*}
		\int_\Sg 
		\big|f(\map(x_\Sg, s))\big|^2 \dd \s(x_\Sg)
		\le
		I_1 + I_2\le	C \big( 
		\| \nbA f \|_{L^2(\dR^2; \mathbb{C}^2)}^2 + 
		\| f \|_{L^2(\dR^2)}^2 \big)
	\end{equation*}
	which is the claim of this lemma.
\end{proof}

\begin{proof}[Proof of Theorem~\ref{theorem_approximation}]
	 According to Lemma~\ref{lemma_uniformly_bounded}
	 the operators $\sfH_\eps$, $\eps\in (0,\beta)$, are uniformly boun\-ded
	 from below by $\lm_1 \in\dR$. Moreover, by Proposition~\ref{proposition_delta_form}
	 the operator $\Op$ is semibounded.
	 From now on we fix $\lm_0 \in \rho(\Op)\cap(-\infty,\lm_1)$ and we use the notations
    $\sfR_\eps := (\sfH_\eps - \lm_0)^{-1}$ and  $\sfR_{\aa}' := (\Op - \lm_0)^{-1}$. Note that $\Vert \sfR_\eps\Vert \leq (\lambda_1-\lambda_0)^{-1}$ for
    $\eps\in (0,\beta)$.
    We claim that there is a constant $c > 0$ such that
    \begin{equation} \label{equation_convergence}
        \big\| \sfR_{\eps} - \sfR_{\aa}' \big\|
        \leq
        c \sqrt{\eps},\qquad \eps\in (0,\beta).
    \end{equation}
    In fact, note first that
    \begin{equation*}
    \begin{split}
        \big\| \sfR_{\eps} - \sfR_{\aa}'  \big\|
        &= \sup_{\| u \|, \| v \| = 1}
            \Big|\big( \big( \sfR_{\eps} - \sfR_{\aa}'\big) u, v \big)_{L^2(\mathbb{R}^2)}\Big| \\
        &= \sup_{\| u \|, \| v \| = 1}
            \Big| \big( \sfR_{\eps} u, (\Op - \lm_0) \sfR_{\aa}' v \big)_{L^2(\mathbb{R}^2)}
        -\big( (\sfH_{\eps} - \lm_0)\sfR_{\eps} u,
             \sfR_{\aa}' v \big)_{L^2(\mathbb{R}^2)} \Big| \\
        &= \sup_{\| u \|, \| v \| = 1}
            \Big|\fra_{\aa} \big[\sfR_{\eps} u,
            \sfR_{\aa}'  v \big]
            -
            \frh_{\eps}
            \big[\sfR_{\eps} u,  \sfR_{\aa}' v \big] \Big|.
    \end{split}
    \end{equation*}
    The estimate~\eqref{equation_convergence} follows if we prove 
    \begin{equation} \label{estimate_forms}
        \big|\fra_{\aa}[f,g] - \frh_{\eps}[f,g]\big|
       \leq
       c \sqrt{\eps}
       \big(\|f \|_{\mathcal{H}^1_\bA(\mathbb{R}^2)}^2+ \|g \|_{\mathcal{H}^1_\bA(\mathbb{R}^2)}^2\big),\quad f, g \in \cH^1_\bA(\dR^2),
    \end{equation}
    since with the choice $f = \sfR_{\eps} u$ and $g = \sfR_\aa' v$ the inequality \eqref{estimate_forms} together with
    \eqref{equation_form_bound_delta_form} and Lemma~\ref{lemma_V_eps_relatively_bounded}
    shows
    \begin{equation*}
      \begin{split}
        \Big|\fra_{\aa} \big[\sfR_{\eps} u,
            \sfR_{\aa}'  v \big]
            &-
            \frh_{\eps}
            \big[\sfR_{\eps} u,  \sfR_{\aa}' v \big] \Big| \\
            &\leq 
            c \sqrt{\eps}
       \big(\frh_\eps[\sfR_{\eps} u] + \| \sfR_\eps u \|_{L^2(\mathbb{R}^2)}^2 +  \fra_\aa[\sfR_\aa' v] + \| \sfR_\aa' v \|_{L^2(\mathbb{R}^2)}^2 \big) \\
       &=c \sqrt{\eps}
       \bigl( (\sfR_{\eps} u,u)_{L^2(\mathbb{R}^2)} + (1+\lambda_0)\| \sfR_\eps u \|_{L^2(\mathbb{R}^2)}^2 \\
       &\qquad\qquad\qquad\qquad +  (\sfR_\aa' v,v)_{L^2(\mathbb{R}^2)} + (1+\lambda_0)\| \sfR_\aa' v \|_{L^2(\mathbb{R}^2)}^2 \big) \\
       &\leq 
            c \sqrt{\eps}
       \big(\| u \|_{L^2(\mathbb{R}^2)}^2 + \| v \|_{L^2(\mathbb{R}^2)}^2 \big),
      \end{split}
    \end{equation*}
    where $\Vert \sfR_\eps\Vert \leq (\lambda_1-\lambda_0)^{-1}$ was used in the last estimate.
    Thanks to the polarization identity it suffices to prove~\eqref{estimate_forms} for $f=g$.
    Furthermore, it is sufficient to consider $f \in C_0^\infty(\dR^2)$.
    By the definition of the forms $\fra_\aa$ and $\frh_\eps$, using $\supp V_\eps \subset \Sg_\eps$, and~\eqref{transformation_tubular_coords} we find
    \begin{equation*}
    \begin{split}
        \fra_\aa[f] - \frh_\eps[f]
       &= \int_\Sigma \alpha(x_\Sg) |f(x_\Sigma)|^2 \dd \sigma(x_\Sigma) - 
       \int_{\mathbb{R}^2} V_\eps(x) |f(x)|^2 \dd \sigma(x_\Sigma) \\
       &= \int_\Sigma \int_{-\beta}^\beta V( \map(x_\Sg, t)) |f(x_\Sg)|^2 \dd t \dd \sigma(x_\Sigma) \\ 
       &\qquad- \frac{\beta}{\eps} \int_\Sg \int_{-\eps}^\eps
       V\big(\map(x_\Sg, \tfrac{\bb s}{\eps})\big)
       \big| f\big(\map(x_\Sg, s)\big) \big|^2 \cJ(x_\Sg, s) \dd s \dd \s(x_\Sg),
    \end{split}
    \end{equation*}
    where in the last step the definitions of $\aa$ and $V_\eps$ from~\eqref{def_alpha_limit} and~\eqref{def_V_eps}
    were substituted. Using the transformation $t = \frac{\bb}{\eps} s$ in the last integral on the right hand side we find
    \begin{equation} \label{difference_form_tubular_coords}
    \begin{split}
      \fra_\aa[f] - \frh_\eps[f]
      &=
      \frac{\eps}{\bb}
       \int_\Sg \int_{-\bb}^\bb
       V\big(\map(x_\Sg, t)\big)
       \big| f\big(\map(x_\Sg, \tfrac{\eps t}{\bb})\big)\big|^2 t \kp(x_\Sg) \dd t \dd \s(x_\Sg) \\
        &\quad \!+ \!
      \int_\Sg \int_{-\bb}^\bb
      V\big(\map(x_\Sg, t)\big)
      \big[
      |f(x_\Sg)|^2
      \!- \!
      \big|f \big(\map(x_\Sg,\tfrac{\eps t}{\bb})\big)\big|^2 \big] \dd t \dd \s(x_\Sg)\\
      & := I_1+ I_2. 
    \end{split}
    \end{equation}
    Since $\kappa, V \in L^\infty(\dR^2)$ we obtain from Lemma~\ref{lemma_trace_theorem} for the first integral $I_1$ in \eqref{difference_form_tubular_coords}
    the estimate
    \begin{equation} \label{estimate_I_1}
        |I_1| \leq c \eps \|f\|^2_{\mathcal{H}^1_\bA(\mathbb{R}^2)}.
    \end{equation}
    In order to estimate the second integral $I_2$ in \eqref{difference_form_tubular_coords} we note first that by the main theorem of calculus 
    \begin{equation*} 
    \begin{split}
       \Big|\big| f\big(\map(x_\Sg,0)\big)
        \big|^2
        -
        \big|f\big(\map(x_\Sg, \tfrac{\eps t}{\bb}) \big) \big|^2 \Big| 
        &=
        \left|\int_{0}^{\eps}
        \frac{\dd}{\dd r} (|f|^2)
        \big(\map(x_\Sg, \tfrac{rt}{\bb}) \big) \dd r \right| \\
         &\leq
         \int_0^\eps
         \left|\left\langle \nb (|f|^2)
         \big(\map(x_\Sg,\tfrac{r t}{\bb}\big)\big),
        \tfrac{t}{\bb}  \nu(x_\Sigma) \right\rangle \right| \dd r  \\
        &\leq
         \frac{|t|}{\beta} \int_0^\eps \left|\nb (|f|^2)
         \big(\map(x_\Sg, \tfrac{rt}{\beta})\big)\right| \dd r  \\
        &\leq
        c \int_0^\eps
        \left|\nb (|f|) \big(\map(x_\Sg, s) \big) \right| \cdot
        \left|f \big(\map(x_\Sg, s)\big) \right| \dd s ,
    \end{split}
    \end{equation*}
    where the substitution $s = \tfrac{1}{\beta} r t$ was used in the last step. This and the Cauchy-Schwarz inequality lead to
    \begin{equation} \label{estimate_I_3}
    \begin{split}
        |I_2|^2 &\leq c
        \left(\int_\Sg \int_{-\bb}^\bb \int_0^\eps
        \left|\nb (|f|) \big(\map(x_\Sg, s) \big) \right| \cdot
        \left|f \big(\map(x_\Sg, s)\big) \right|\big) \dd s \dd t \dd \s(x_\Sg) \right)^2 \\
        &\leq c \int_\Sg \int_0^\eps
        \big|\nb (|f|) \big(\map(x_\Sg, s) \big) \big|^2 \dd s \dd \s(x_\Sg) \cdot \int_\Sg \int_0^\eps
        \big|f \big(\map(x_\Sg, s)\big) \big|^2 \dd s \dd \s(x_\Sg).
    \end{split}
    \end{equation}
    Choose a constant $c$ such that $1 \leq c \cJ(x_\Sg,s)$. Then using formula~\eqref{transformation_tubular_coords} and the diamagnetic inequality~\eqref{diamagnetic_inequality_gradient} we find that the first integral in the last equation can be estimated by
    \begin{equation*}
     \begin{split}
       c \int_\Sg \int_{-\eps}^\eps
        \big|\nb (|f|) \big(\map(x_\Sg, s) \big) \big|^2 \cJ(x_\Sg,s) \dd s \dd \s(x_\Sg)
        \leq c \int_{\Sg_\eps} |\nbA f|^2 \dd x
      &\leq c \|f\|^2_{\mathcal{H}^1_\bA(\mathbb{R}^2)}.
     \end{split}
    \end{equation*}
    Moreover, the second integral on the right hand side of~\eqref{estimate_I_3} can be estimated with Lemma~\ref{lemma_trace_theorem} by $c \eps \|f\|_{\mathcal{H}^1_\bA(\mathbb{R}^2)}^2$.
    Combining this with~\eqref{estimate_I_1} and~\eqref{difference_form_tubular_coords} we deduce~\eqref{estimate_forms} and hence~\eqref{equation_convergence}.

    Finally, we extend the result from~\eqref{equation_convergence} from $\lm_0 \in \rho(\Op)\cap(-\infty,\lm_1)$ to all $\lm \in \dC\sm \dR$.
    For this we consider
    $\sfD_\eps(\lm) := (\sfH_\eps - \lm)^{-1} - (\Op - \lm)^{-1}$. A simple computation shows
    \begin{equation*}
        \sfD_\eps(\lm)
        =
        \big[ 1 + (\lm - \lm_0)
        \big(\Op - \lm\big)^{-1} \big]\cdot \sfD_\eps(\lm_0)\cdot
        \big[ 1 + (\lm - \lm_0)
        \big(\sfH_\eps - \lm \big)^{-1} \big].
    \end{equation*}
    Hence the claimed convergence result is true for all $\lm \in \dC \sm \dR$
    and the order of convergence is $\sqrt{\eps}$.
    This finishes the proof of Theorem~\ref{theorem_approximation}.
\end{proof}

\bibliographystyle{alpha}

\newcommand{\etalchar}[1]{$^{#1}$}

\end{document}